%% file: pants_cobordism_group_5.tex
\documentclass{amsart}
\usepackage{amssymb,amsmath,mathrsfs,amsfonts}
\usepackage{amscd, amssymb, amsmath, amsthm, graphics,graphicx}
\usepackage{amsmath,amsfonts,amsthm,amssymb}
\usepackage{slashbox}
\usepackage{fancyhdr}
\usepackage[all]{xy}
\usepackage{graphicx}
\usepackage{color}
\usepackage[colorlinks=true]{hyperref}
\usepackage{multirow}
\usepackage{enumitem}
\numberwithin{equation}{section}

\newtheorem{theorem}{Theorem}[section]
\newtheorem{lemma}[theorem]{Lemma}
\newtheorem{proposition}[theorem]{Proposition}

\newtheorem{condition}[theorem]{Condition}

\theoremstyle{definition}
\newtheorem{definition}[theorem]{Definition}

\newtheorem{question}[theorem]{Question}
\newtheorem{remark}[theorem]{Remark}

\begin{document}
\sloppy

\title[The panted cobordism group of cusped hyperbolic $3$-manifolds]{The panted cobordism group of cusped hyperbolic $3$-manifolds}

\author{Hongbin Sun}
\address{Department of Mathematics, Rutgers University - New Brunswick, Hill Center, Busch Campus, Piscataway, NJ 08854, USA}
\email{hongbin.sun@rutgers.edu}


\thanks{The author is partially supported by NSF Grant No. DMS-1840696 and Simons Collaboration Grants 615229.}

\date{\today}
\begin{abstract}
For any oriented cusped hyperbolic $3$-manifold $M$, we study its $(R,\epsilon)$-panted cobordism group, which is the abelian group generated by $(R,\epsilon)$-good curves in $M$ modulo the oriented boundaries of $(R,\epsilon)$-good pants. In particular,   we prove that for sufficiently small $\epsilon>0$ and sufficiently large $R>0$, some modified version of the $(R,\epsilon)$-panted cobordism group of $M$ is isomorphic to $H_1(\text{SO}(M);\mathbb{Z})$.
\end{abstract}

\maketitle
\vspace{-.5cm}
\section{Introduction}
	
The celebrated good pants construction was initiated by Kahn and Markovic. In \cite{KM1}, Kahn and Markovic proved the surface subgroup theorem: any closed hyperbolic $3$-manifold $M$ admits a $\pi_1$-injective immersed closed hyperbolic subsurface $S\looparrowright M$. This is the first step of Agol's proof of the virtual Haken and virtual fibering conjectures (\cite{Agol}), which are the most important results on $3$-manifold topology in the past decade.

Although the above statement is purely topological, Kahn and Markovic actually used geometric method to construct $\pi_1$-injective subsurfaces, and the nice geometric properties make these subsurfaces useful for studying other aspects of hyperbolic $3$-manifolds. More precisely, Kahn and Markovic used tools in geometry and dynamics to paste a lot of $(R,\epsilon)$-good pants in $M$ along $(R,\epsilon)$-good curves in a nearly geodesic manner, to get the desired $\pi_1$-injective subsurface (see Section \ref{pantsclosed} for definitions). This is the starting point of a collection of works under the theme of {\it good pants construction}, which use various versions of good pants to construct $\pi_1$-injective subsurfaces in geometrically interesting spaces (\cite{KM1,KM2,LM,Ham,KW1}).

In \cite{KM2}, Kahn and Markovic used the good pants construction to resolve the Ehrenpreis conjecture, that is, for any $\epsilon>0$ and any pair of closed hyperbolic surfaces $S_1$ and $S_2$, each of them has a finite cover $\tilde{S}_1$ and $\tilde{S}_2$ respectively, such that they are $(1+\epsilon)$-quasi-isometric to each other. 

To prove the Ehrenpreis conjecture, Kahn and Markovic proved that any closed hyperbolic surface has a finite cover that admits a pants decomposition, such that each cuff is an $(R,\epsilon)$-good curve and the two pairs of $(R,\epsilon)$-pants adjacent to it are pasted by a nearly $1$-shift. In the proof, Kahn and Markovic developed the theory of {\it good pants homology} on closed hyperbolic surfaces, which is essentially a homology theory with $\mathbb{R}$-coefficient (good curves modulo oriented boundaries of good pants,).

In \cite{LM}, among other works, Liu and Markovic studied the $\mathbb{Z}$-coefficient good pants homology of closed oriented hyperbolic $3$-manifolds, and they called it the {\it panted cobordism group} $\Omega_{R,\epsilon}(M)$. Let $\text{SO}(M)$ be the frame bundle over $M$, then they proved that $\Omega_{R,\epsilon}(M)$ is isomorphic to $H_1(\text{SO}(M);\mathbb{Z})$. (See Section \ref{pantscobordismgroup} for more details for Liu and Markovic's work in \cite{LM}.)

All the above results on good pants construction only work for closed hyperbolic $3$-manifolds, since good pants are not equidistributed along good curves in cusped hyperbolic $3$-manifolds. In \cite{KW1}, Kahn and Wright overcame this difficulty, and generalized Kahn and Markovic's surface subgroup theorem to cusped hyperbolic $3$-manifolds. Besides $(R,\epsilon)$-good pants, they introduced a new object called $(R,\epsilon)$-good hamster wheel, as another building block of the $\pi_1$-injective subsurface. Our paper will not use the main result in \cite{KW1}. However, we will use some geometric ideas in \cite{KM1} and adopt some basic notions in that paper, including a height function in cusped hyperbolic $3$-manifolds. 

This paper is devoted to generalize Liu and Markovic's result on panted cobordism group (in \cite{LM}) to cusped hyperbolic $3$-manifolds. Here we can not prove exactly the same result as in the closed manifold case. For any $h'>h>0$, we will work on the panted cobordism group with height control, which is denoted by $\Omega_{R,\epsilon}^{h,h'}(M)$. Here $\Omega_{R,\epsilon}^{h,h'}(M)$ is generated by all $(R,\epsilon)$-good curves of height at most $h$, modulo the oriented boundary of linear combinations of $(R,\epsilon)$-good pants of height at most $h'$ such that the oriented boundary of this linear combination has height at most $h$. We will also study $\Omega_{R,\epsilon}^h(M)=\Omega_{R,\epsilon}^{h,\infty}(M)$. (The height function in a cusped hyperbolic $3$-manifold will be defined in Section \ref{pantscusped}, and the precise definition of $\Omega_{R,\epsilon}^{h,h'}(M)$ will be introduced in Section \ref{pantscobordismsection}.)

\begin{theorem}\label{main1}
For any oriented cusped hyperbolic $3$-manifold $M$, any numbers $\beta>\alpha\geq 4$ with $\beta-\alpha\geq 3$ and any $\epsilon\in (0,10^{-2})$, there exists $R_0=R_0(M,\epsilon)>0$, such that for any $R>R_0$, we have $$\Omega_{R,\epsilon}^{\alpha\ln {R},\beta\ln{R}}(M)\cong H_1(\text{SO}(M);\mathbb{Z}).$$ Moreover, we also have $$\Omega_{R,\epsilon}^{\alpha\ln {R}}(M)\cong H_1(\text{SO}(M);\mathbb{Z}).$$
\end{theorem}

The ``moreover'' part of Theorem \ref{main1} is a nice and neat result by itself. However, in order to study virtual domination of $3$-manifolds in a forthcoming paper \cite{Sun}, we need to adopt the construction in \cite{KW1}, which requires the result in Theorem \ref{main1} on $\Omega_{R,\epsilon}^{\alpha\ln {R},\beta\ln{R}}(M)$.




Given Theorem \ref{main1}, it is natural to ask the following question on the panted cobordism group with no height control. Here $\Omega_{R,\epsilon}(M)$ denotes the panted cobordism group $\Omega^{h,h'}_{R,\epsilon}(M)$ with $h=h'=\infty$.
\begin{question}\label{noheightcontrol}
For a cupsed hyperbolic $3$-manifold $M$, small enough $\epsilon>0$ and large enough $R>0$, is $\Omega_{R,\epsilon}(M)$ isomorphic to $H_1(SO(M);\mathbb{Z})$?
\end{question}

The main difficulty for Question \ref{noheightcontrol} is that we do not have a {\it connection principle} that works for any two points in a cusped hyperbolic $3$-manifold without much condition on their heights. In Section \ref{connection_principle}, we will prove a connection principle in cusped hyperbolic $3$-manifolds (Theorem \ref{connection_principle_with_height}), whose input contains two points in the manifold that are not too high (comparing to the length of the object to be constructed), and the output is an oriented $\partial$-framed segment connecting these two points with controlled height. 

The key ingredient in the proof of Theorem \ref{connection_principle_with_height} is a counting result in \cite{KW2} on Lie groups. The proof of Theorem \ref{connection_principle_with_height} is essentially a combination of proofs of Theorem 3.11 of \cite{KW2} and Theorem 3.2 of \cite{KW1}. 

Given the new connection principle in Theorem \ref{connection_principle_with_height}, the main strategy for proving Theorem \ref{main1} is similar to the strategy in \cite{LM}. However, in almost all lemmas, we need to have some height control since we are working in cusped hyperbolic $3$-manifolds. For completeness, we will give full proofs of all necessary lemmas in \cite{LM} that need height control, while many geometric constructions are borrowed from \cite{LM}. Moreover, some proofs do need extra efforts and new geometric ideas to take care of the cusped manifold case.

The organization of this paper is as the following. In Section \ref{pregoodpants}, we review some basic notions and results on the  good pants construction in \cite{KM1}, \cite{LM} and \cite{KW1}. In Section \ref{prehyperbolic}, we give some elementary results on hyperbolic geometry, and most of them are well-known for experts. In Section \ref{connection_principle}, we prove our connection principle in cusped hyperbolic $3$-manifolds (Theorem \ref{connection_principle_with_height}). In Section \ref{construction}, we give a few geometric constructions (with height control) in cusped hyperbolic $3$-manifolds. In Section \ref{cobordism}, we give the proof of Theorem \ref{main1}, which is the heart of this paper. We break Section \ref{cobordism} into a few subsections. In Section \ref{pantscobordismsection}, we define the panted cobordism group with height control $\Omega_{R,\epsilon}^{h,h'}(M)$ and construct a homomorphism $\Phi:\Omega_{R,\epsilon}^{\alpha \ln{R},\beta \ln{R}}(M)\to H_1(SO(M);\mathbb{Z})$. In Section \ref{inversehomomorphism}, we construct a homomorphism $\Psi^{\text{ab}}: H_1(SO(M);\mathbb{Z})\to \Omega_{R,\epsilon}^{\alpha \ln{R},\beta \ln{R}}(M)$ in the inverse direction. The homomorphism $\Psi^{\text{ab}}$ is actually induced by a homomorphism $\Psi:\pi_1(SO(M))\to \Omega_{R,\epsilon}^{\alpha \ln{R},\beta \ln{R}}(M)$, and $\Psi$ is induced by a set-theoretic map $\Psi_1$ from a finite generating set of $\pi_1(SO(M))$ to $\Omega_{R,\epsilon}^{\alpha \ln{R},\beta \ln{R}}(M)$. In Section \ref{verifications}, we check that $\Phi\circ \Psi^{\text{ab}}=id_{H_1(SO(M);\mathbb{Z})}$ and $\Psi^{\text{ab}}$ is surjective, so $\Phi$ is an isomorphism.

\bigskip

{\bf Acknowledgement.} The author thanks Alex Wright for answering several questions on \cite{KW2}, and thanks Yi Liu for helpful conversations on this paper.




\bigskip
\bigskip

\section{Preliminary on the good pants construction}\label{pregoodpants}

In this section, we review the good pants construction of closed hyperbolic $3$-manifolds in \cite{KM1}, the panted cobordism group studied in \cite{LM}, and the good pants construction of cusped hyperbolic $3$-manifolds in \cite{KW1}.

\subsection{The good pants construction in closed hyperbolic $3$-manifolds}\label{pantsclosed}

In \cite{KM1}, Kahn and Markovic proved the following celebrated surface subgroup theorem, which was used as the first step of Agol's proof of Thurston's virtual Haken and virtual fibering conjectures (\cite{Agol}).
	
	\begin{theorem}[Surface subgroup theorem \cite{KM1}]\label{surface}
		For any closed hyperbolic $3$-manifold $M$,
		there exists an immersed closed hyperbolic subsurface $f\colon S\looparrowright M$,
		such that $f_*\colon \pi_1(S)\rightarrow 	\pi_1(M)$ is an injective map.
	\end{theorem}
	
The immersed subsurface of Kahn and Markovic is built by pasting a large collection of {\it $(R,\epsilon)$-good pants} along {\it $(R,\epsilon)$-good curves} in a nearly geodesic way, which is called the {\it good pants construction} in general. More details of terminologies are given in the following.

We fix a closed oriented hyperbolic $3$-manifold $M$, a small number $\epsilon>0$ and a large number $R>0$. 

\begin{definition}\label{goodcurves}
An {\it $(R,\epsilon)$-good curve} is an oriented closed geodesic in $M$ with complex length satisfying $|l(\gamma)-2R|<2\epsilon$. The set consists of all such $(R,\epsilon)$-good curves is denoted by ${\bold \Gamma}_{R,\epsilon}$.
\end{definition}

Here the complex length of $\gamma$ is defined by $l(\gamma)=l+i\theta \in \mathbb{C}/2\pi i\mathbb{Z}$, where $l\in \mathbb{R}_{>0}$ is the length of $\gamma$, and $\theta\in \mathbb{R}/2\pi \mathbb{Z}$ is the rotation angle of the loxodromic isometry of $\mathbb{H}^3$ corresponding to $\gamma$. Note that ${\bf \Gamma}_{R,\epsilon}$ is a finite set.

\begin{definition}\label{goodpants}
A pair of {\it $(R,\epsilon)$-good pants} is a homotopy class of immersed oriented pair of pants $\Pi \looparrowright M$, such that the three cuffs of $\Pi$ are mapped to $(R,\epsilon)$-good curves $\gamma_i\in {\bold \Gamma}_{R,\epsilon}$ for $i=1,2,3$, and the complex half length $\bold{hl}_{\Pi}(\gamma_i)$ of each $\gamma_i$ with respect to $\Pi$ satisfies
		$$\left|\bold{hl}_{\Pi}(\gamma_i)-R\right|<\epsilon.$$
The set consists of all such $(R,\epsilon)$-good pants is denoted by ${\bold \Pi}_{R,\epsilon}$, which is also a finite set. 
\end{definition}

Here the complex half length is defined as the following. Let $s_{i-1}$ and $s_{i+1}$ be the common perpendicular segments (also called seams) from $\gamma_i$ to $\gamma_{i-1}$ and $\gamma_{i+1}$ respectively, let $\vec{v}_{i-1}$ and $\vec{v}_{i+1}$ be the tangent vectors of $s_{i-1}$ and $s_{i+1}$ at their intersections with $\gamma_i$ respectively (and give the orientations of $s_{i-1}$ and $s_{i+1}$). Then the complex half length $\bold{hl}_{\Pi}(\gamma_i)$ is defined to be the complex distance between $\vec{v}_{i-1}$ and $\vec{v}_{i+1}$ along $\gamma_i$. More precisely, we have $\bold{hl}_{\Pi}(\gamma_i)=l+i\theta \in \mathbb{C}/2\pi i\mathbb{Z}$ where $l$ is the length of the geodesic subsegment of $\gamma_i$ from $s_{i-1}$ to $s_{i+1}$ and $\theta$ is the angle from the parallel transport of $\vec{v}_{i-1}$ along this geodesic segment to $\vec{v}_{i+1}$. Note that the definition of $\bold{hl}_{\Pi}(\gamma_i)$ does not change if $s_{i-1}$ and $s_{i+1}$ are swapped.

Note that in this paper, we adopt the convention in \cite{KW1} that good curves have length close to $2R$, instead of following the convention in \cite{KM1} that good curves have length close to $R$.
	
In \cite{KM1}, an important innovation is that good pants are pasted along good curves with nearly $1$-shifts, rather than exactly matching seams along the common cuff. The nearly $1$-shift is a crucial condition to guarantee that $f_*: \pi_1(S)\rightarrow \pi_1(M)$ is injective. Moreover, Kahn and Markovic showed that, for any $(R,\epsilon)$-good curve $\gamma$, the feet of $(R,\epsilon)$-good pants on $\gamma$ are nearly evenly distributed along $\gamma$. So $M$ contains a large collection of $(R,\epsilon)$-good pants, and they can be pasted together by nearly $1$-shift. Therefore, the asserted $\pi_1$-injective immersed subsurface can be constructed. 

\subsection{The panted cobordism group of closed hyperbolic $3$-manifolds}\label{pantscobordismgroup}

In \cite{LM}, Liu and Markovic studied the panted cobordism group of closed oriented hyperbolic $3$-manifolds, and we review some material in \cite{LM} in this section. The main goal of this paper is to generalize part of the work in \cite{LM} to oriented cusped hyperbolic $3$-manifolds.

We fix a closed oriented hyperbolic $3$-manifold $M$, a small number $\epsilon>0$ and a large number $R>0$. Let $\mathbb{Z}{\bf \Gamma}_{R,\epsilon}$ be the free abelian group generated by ${\bf \Gamma}_{R,\epsilon}$, modulo the relation $\gamma+\bar{\gamma}=0$ for all $\gamma\in {\bf \Gamma}_{R,\epsilon}$. Here $\bar{\gamma}$ denotes the orientation reveral of $\gamma$. Let $\mathbb{Z}{\bf \Pi}_{R,\epsilon}$ be the free abelian group generated by ${\bf \Pi}_{R,\epsilon}$, modulo the relation $\Pi+\bar{\Pi}=0$ for all $\Pi\in {\bf \Pi}_{R,\epsilon}$. By taking the oriented boundary of $(R,\epsilon)$-good pants, we get a homomorphism $\partial:\mathbb{Z}{\bf \Pi}_{R,\epsilon}\to \mathbb{Z}{\bf \Gamma}_{R,\epsilon}$. The panted cobordism group $\Omega_{R,\epsilon}(M)$ is defined as the following in \cite{LM}.

\begin{definition}\label{closedcobordismgroup} 
The {\it panted cobordism group} $\Omega_{R,\epsilon}(M)$ is defined to be the cokernel of the homomorphism $\partial$, i.e. $\Omega_{R,\epsilon}(M)$ fits into the following exact sequence 
$$\mathbb{Z}{\bf \Pi}_{R,\epsilon}\xrightarrow{\partial} \mathbb{Z}{\bf \Gamma}_{R,\epsilon}\to \Omega_{R,\epsilon}(M)\to 0.$$
\end{definition}

Actually, there is an alternative definition of $\Omega_{R,\epsilon}(M)$ in \cite{LM}, which is more convenient for this paper. This definition is based on the following concept of $(R,\epsilon)$-panted subsurface in hyperbolic $3$-manifolds (Definition 2.7 of \cite{LM}).

\begin{definition}\label{pantedsurface}
An {\it $(R,\epsilon)$-panted subsurface} in a hyperbolic $3$-manifold $M$ is a (possibly disconnected) compact oriented surface $F$ with a pants decomposition and an immersion $j:F\looparrowright M$ such that the restriction of $j$ to each pair of pants in the pants decomposition of $F$ gives a pair of $(R,\epsilon)$-good pants.
\end{definition}

An $(R,\epsilon)$-multicurve in $M$ is defined to be a finite collection of $(R,\epsilon)$-good curves in $M$ with multiplicity.
On the set of $(R,\epsilon)$-multicurves in $M$, for two $(R,\epsilon)$-multicurves $L$ and $L'$, we say that they are equivalent if there is an $(R,\epsilon)$-panted subsurface $F$ in $M$, such that the oriented boundary of $F$ is $L\cup \overline{L'}$. Then the quotient of the set of all $(R,\epsilon)$-multicurves under this equivalence relation also gives rise to $\Omega_{R,\epsilon}(M)$.

To describe the result in \cite{LM} that we want to generalize, we also need the definition of the frame bundle of an oriented hyperbolic $3$-manifold.

\begin{definition}\label{framebundle}
For an oriented hyperbolic $3$-manifold $M$, and a point $p\in M$, a \emph{special orthonormal frame} (or simply a frame) of $M$ at $p$ is a triple of unit tangent vectors $(\vec{t}_p,\vec{n}_p,\vec{t}_p\times \vec{n}_p)$ such that $\vec{t}_p,\vec{n}_p\in T_p^1M$ with $\vec{t}_p\perp \vec{n}_p$, and $\vec{t}_p\times \vec{n}_p\in T_p^1M$ is the cross product with respect to the orientation of $M$. We use $\text{SO}(M)$ to denote the frame bundle over $M$ that consists of all special orthonormal frames of $M$.
\end{definition}

For simplicity, we will denote each frame in $\text{SO}(M)$ by its base point and the first two vectors of the frame, as $(p,\vec{t}_p,\vec{n}_p)$, since the third vector is determined by the first two. We call $\vec{t}_p$ and $\vec{n}_p$ the tangent vector and the normal vector of this frame, respectively.

One of the main result in \cite{LM} is that the panted cobordism group $\Omega_{R,\epsilon}(M)$ is isomorphic to the $\mathbb{Z}$-coefficient first homology group of $\text{SO}(M)$.
	
\begin{theorem}\label{homology}(Theorem 5.2 of \cite{LM})
Given a closed oriented hyperbolic $3$-manifold $M$, for small enough $\epsilon>0$ depending on $M$ and large enough $R>0$ depending on $\epsilon$ and $M$, there is a natural isomorphism 
		$$\Phi:{\bold \Omega}_{R,\epsilon}(M)\rightarrow H_1(\text{SO}(M);\mathbb{Z})$$
from the panted cobordism group to the first homology group of $\text{SO}(M)$.
	\end{theorem}
	
The goal of this paper is to generalize Theorem \ref{homology} to cusped hyperbolic $3$-manifolds. To do this generalization,
we need several definitions of geometric objects in \cite{LM}. These definitions will be given at several spots in this paper, and the most important one is the following.
 
\begin{definition}\label{segment} 
An {\it oriented $\partial$-framed segment} in $M$ is a triple 
$$\mathfrak{s}=(s,\vec{n}_{\text{ini}},\vec{n}_{\text{ter}}),$$ 
such that $s$ is an immersed oriented compact geodesic segment (or simply called a geodesic segment), $\vec{n}_{\text{ini}}$ and $\vec{n}_{\text{ter}}$ are unit normal vectors of $s$ at the initial and terminal points of $s$ respectively. 
\end{definition}

We have the following objects associated to an oriented $\partial$-framed segment $\mathfrak{s}$:
\begin{itemize}
\item The {\it carrier segment} of $\mathfrak{s}$ is the (oriented) geodesic segment $s$.
\item The {\it initial endpoint} $p_{\text{ini}}(\mathfrak{s})$ and the {\it terminal point} $p_{\text{ter}}(\mathfrak{s})$ are the initial and terminal points of $s$ respectively.
\item The {\it initial framing} $\vec{n}_{\text{ini}}(\mathfrak{s})$ and the {\it terminal framing} $\vec{n}_{\text{ter}}(\mathfrak{s})$ are the unit normal vectors $\vec{n}_{\text{ini}}$ and $\vec{n}_{\text{ter}}$ respectively.
\item The {\it initial direction} $\vec{t}_{\text{ini}}(\mathfrak{s})$ and the {\it terminal direction} $\vec{t}_{\text{ter}}(\mathfrak{s})$ are the unit tangent vectors in the direction of $s$ at $p_{\text{ini}}(\mathfrak{s})$ and $p_{\text{ter}}(\mathfrak{s})$ respectively.
\item The {\it length} $l(\mathfrak{s})\in(0,\infty)$ of $\mathfrak{s}$ is the length of the geodesic segment $s$, the {\it phase} $\varphi(\mathfrak{s})\in \mathbb{R}/2\pi \mathbb{Z}$ of $\mathfrak{s}$ is the angle from the parallel transport of $\vec{n}_{\text{ini}}$ along $s$ to $\vec{n}_{\text{ter}}$.
\end{itemize}

Moreover, the {\it orientation reversal} of $\mathfrak{s}=(s,\vec{n}_{\text{ini}},\vec{n}_{\text{ter}})$ is defined to be 
$$\bar{\mathfrak{s}}=(\bar{s},\vec{n}_{\text{ter}},\vec{n}_{\text{ini}}).$$ 
The {\it frame flipping} of $\mathfrak{s}=(s,\vec{n}_{\text{ini}},\vec{n}_{\text{ter}})$ is defined to be 
$$\mathfrak{s}^*=(s,-\vec{n}_{\text{ini}},-\vec{n}_{\text{ter}}).$$ 
For any angle $\phi \in \mathbb{R}/2\pi\mathbb{Z}$, the {\it frame rotation} of $\mathfrak{s}$ by $\phi$ is defined to be
$$\mathfrak{s}(\phi)=(s,\vec{n}_{\text{ini}}\cos{\phi}+(\vec{t}_{\text{ini}}\times \vec{n}_{\text{ini}})\sin{\phi},\vec{n}_{\text{ter}}\cos{\phi}+(\vec{t}_{\text{ter}}\times \vec{n}_{\text{ter}})\sin{\phi}).$$

\subsection{The good pants  construction in cusped hyperbolic $3$-manifolds}\label{pantscusped}

In \cite{KW1}, Kahn and Wright generalized Kahn and Markovic's surface subgroup theorem (Theorem \ref{surface}) to cusped hyperbolic $3$-manifolds.

\begin{theorem}\label{surfaceincusp}(Theorem 1.1 of \cite{KW1})
		Let $\Gamma<PSL_2(\mathbb{C})$ be a Kleinian group and assume that $\mathbb{H}^3/\Gamma$ has finite volume and is not compact. Then for all $K>1$, there exists $K$-quasi-Fuchsian surface subgroups in $\Gamma$.
	\end{theorem}
	
To prove Theorem \ref{surfaceincusp}, Kahn and Wright introduced a new geometric object called {\it $(R,\epsilon)$-hamster wheel} in \cite{KW1}. We will not review the proof of Theorem \ref{surfaceincusp} in \cite{KW1}, since the proof will not be used in this paper. Alternatively, we will review some basic notions on cusped hyperbolic $3$-manifolds introduced in \cite{KW1}.

At first, by the Margulis lemma, there exists $\epsilon_0>0$, such that the subset of $M$ consisting of points of injectivity radii at most $\epsilon_0$ is a disjoint union of solid tori and cusp neighborhoods of ends of $M$. By \cite{Mey}, we can take $\epsilon_0$ to be $0.1$. We will call the complement of the cusp neighborhoods the {\it thick part} of $M$ (note that the usual convention of thick part also excludes the aforementioned solid tori). We will call these cusp neighborhoods of ends of $M$ by {\it cusps}.

Now we define a height function on $M$. For any point in the thick part of $M$, we define its height to be $0$. For any point $p$ in a cusp $C\subset M$, we define the height of $p$ to be the distance between $p$ and the boundary of $C$. Note that this height function is slightly different from the height function defined in \cite{KW1}.

For a geodesic segment or a closed geodesic in $M$, we define its {\it height} to be the maximal height of points on it. For an oriented $\partial$-framed segment, its height is the height of its carrier segment. For a pair of $(R,\epsilon)$-good pants in $M$, its height is the maximal height of its three cuffs. For an $(R,\epsilon)$-panted subsurface, its height is the maximal height of all $(R,\epsilon)$-good pants contained in it. 

For a geodesic segment or a closed geodesic $s$ in $M$, a {\it cusp excursion} of $s$ is a maximal subsegment of $s$ such that each point in it has positive height. If $s$ is a geodesic segment, a cusp excursion of $s$ that contains the initial (terminal) point of $s$ is called the {\it initial (terminal) cusp excursion} of $s$. Any other cusp excursion is called an {\it intermediate cusp excursion} of $s$. The initial, terminal and intermediate cusp excursions of an oriented $\partial$-framed segment are defined to be the initial, terminal and intermediate cusp excursions of its carrier segment, respectively

For a cusped hyperbolic $3$-manifold $M$ and some $h>0$, we use ${\bf \Gamma}_{R,\epsilon}^h$ to denote the subset of ${\bf \Gamma}_{R,\epsilon}$ that consists of $(R,\epsilon)$-good curves of height at most $h$, and use ${\bf \Pi}_{R,\epsilon}^{h}$ to denote the subset of ${\bf \Pi}_{R,\epsilon}$ that consists of $(R,\epsilon)$-good pants of height at most $h$. Moreover, we use $\mathbb{Z}{\bf \Gamma}_{R,\epsilon}^h$ to denote the subgroup of $\mathbb{Z}{\bf \Gamma}_{R,\epsilon}$ generated by ${\bf \Gamma}_{R,\epsilon}^h$, and use $\mathbb{Z}{\bf \Pi}_{R,\epsilon}^h$ to denote the subgroup of $\mathbb{Z}{\bf \Pi}_{R,\epsilon}$ generated by ${\bf \Pi}_{R,\epsilon}^h$.
 
Then there is a boundary map $\partial^h: \mathbb{Z}{\bf \Pi}_{R,\epsilon}^h\to \mathbb{Z}{\bf \Gamma}_{R,\epsilon}^h$ given by the restriction of the boundary map $\partial: \mathbb{Z}{\bf \Pi}_{R,\epsilon}\to \mathbb{Z}{\bf \Gamma}_{R,\epsilon}$. For any $h'>h>0$, $\mathbb{Z}{\bf \Gamma}_{R,\epsilon}^h$ is naturally a subset of $\mathbb{Z}{\bf \Gamma}_{R,\epsilon}^{h'}$. For the boundary map $\partial^{h'}: \mathbb{Z}{\bf \Pi}_{R,\epsilon}^{h'}\to \mathbb{Z}{\bf \Gamma}_{R,\epsilon}^{h'}$, we use $\mathbb{Z}{\bf \Pi}_{R,\epsilon}^{h,h'}$ to denote $(\partial^{h'})^{-1}(\mathbb{Z}{\bf \Gamma}_{R,\epsilon}^h)$. 

Then the good panted cobordism group $\Omega_{R,\epsilon}^{h,h'}(M)$ in Theorem \ref{main1} is defined to be the cokernel of $\partial^{h'}|_{\mathbb{Z}{\bf \Pi}_{R,\epsilon}^{h,h'}}: \mathbb{Z}{\bf \Pi}_{R,\epsilon}^{h,h'}\to \mathbb{Z}{\bf \Gamma}_{R,\epsilon}^h$ and it fits into the following exact sequence 
$$\mathbb{Z}{\bf \Pi}_{R,\epsilon}^{h,h'}\xrightarrow{\partial^{h'}|_{\mathbb{Z}{\bf \Pi}_{R,\epsilon}^{h,h'}}} \mathbb{Z}{\bf \Gamma}_{R,\epsilon}^h\to \Omega_{R,\epsilon}^{h,h'}(M)\to 0.$$ 
However, this definition is not very convenient for the proof in this paper, so we will give an alternative definition of $\Omega_{R,\epsilon}^{h,h'}(M)$ in Section \ref{pantscobordismsection}, which is our official definition.

\section{Preliminary on hyperbolic geometry}\label{prehyperbolic}

In this section, we give some elementary results on hyperbolic geometry, and they will be repeatly used in the remainder of this paper. Most of these results are either well-known or straight forward for experts.

\begin{lemma}\label{cuspheightvslength}
Let $\gamma$ be a geodesic segment in the upper-half space model of $\mathbb{H}^3$ that is contained in $\{(x,y,z)\ |\ z\geq 1\}\subset \mathbb{H}^3$, such that the $z$-coordinates of the initial and terminal points of $\gamma$ are both $1$. Let $|\gamma|$ be the length of $\gamma$ and let $h_{\gamma}$ be the height of $\gamma$ with respect to the plane $z=1$, then $|\gamma|\leq 2(h_{\gamma}+\ln{2})$.
\end{lemma}

\begin{proof}
We can assume that $\gamma$ is the subsegment of the bi-infinite geodesic from $(-e^{h_{\gamma}},0,0)$ to $(e^{h_{\gamma}},0,0)$, such that the $z$-coordinates of its initial and terminal points are both $1$.

Then $|\gamma|$ is the distance from $(-\sqrt{e^{2h_{\gamma}}-1},0,1)$ to $(\sqrt{e^{2h_{\gamma}}-1},0,1)$, which can be computed by 
$$\int_{\theta}^{\pi-\theta}\frac{e^{h_{\gamma}}}{e^{h_{\gamma}}\sin{t}}\ dt=-2\ln{(\tan{\frac{\theta}{2}})}=2(h_{\gamma}+\ln{(1+\sqrt{1-e^{-2h_{\gamma}}})})<2(h_{\gamma}+\ln{2}),$$
with $\theta=\arcsin{(e^{-h_{\gamma}})}$.
\end{proof}



For two points $A,B$ in the hyperbolic plane or the hyperbolic $3$-space, we use $AB$ to denote the geodesic segment from $A$ to $B$, and use $|AB|$ to denote the length of this segment.

\begin{lemma}\label{hyperbolicconstant}
Let $ABC$ be a hyperbolic triangle, then $AC$ lies in the $\ln{(\sqrt{2}+1)}\approx 0.882$ neighborhood of $AB\cup BC$.
\end{lemma}
\begin{proof}
This is simply the fact that  $\ln{(\sqrt{2}+1)}$ is the hyperbolic constant of the hyperbolic plane $\mathbb{H}^2$ and the hyperbolic $3$-space $\mathbb{H}^3$.
\end{proof}

\begin{lemma}\label{exponentialdecay}
Let $AB$ be a hyperbolic geodesic segment of length at least $2D$ and $\gamma$ is a bi-infinite geodesic such that $d(A,\gamma)<1$ and $d(B,\gamma)<1$ hold. Then for any point $W$ on $AB$ such that $d(W,A),d(W,B)>D$, the distance between $W$ and $\gamma$ is at most $3e^{-D}$.
\end{lemma}
\begin{proof}

Let $\delta$ be the bi-infinite geodesic containing $AB$. 

We first suppose that $\gamma$ and $\delta$ share a common limit point at infinity. We can assume that this common limit point is $\infty$ in the upper-half space model of $\mathbb{H}^3$. By translation and scaling, we can assume that $\gamma$ goes between $0$ and $\infty$ and $\delta$ goes between $1$ and $\infty$. For a point on $\delta$ with $z$-coordinate $c$, its distance from $\gamma$ is $\ln(\frac{\sqrt{c^2+1}+1}{c})$. So the $z$-coordinates of $A$ and $B$ are at least $\frac{2e}{e^2-1}$. Then the $z$-coordinate of $W$, which is denoted by $w$, is at least $\frac{2e}{e^2-1}e^D$, and its distance from $\gamma$ is at most 
$$\ln{(\sqrt{1+\frac{1}{w^2}}+\frac{1}{w})}\leq \frac{1}{w}+\frac{1}{2w^2}\leq \frac{e^2-1}{2e}e^{-D}+\frac{(e^2-1)^2}{8e^2}e^{-2D}<2e^{-D}.$$

If $\gamma$ and $\delta$ do not share limit points at infinity, then let $M$ be the point on $\delta$ that is closest to $\gamma$. Let $\eta$ be the common perpendicular geodesic segment of $\gamma$ and $\delta$ (may degenerate to a point), then $\eta$ intersects with $\delta$ at $M$. Let $d$ be the length of $\eta$ and let $\theta$ be the angle between $\gamma$ and the parallel transport of $\delta$ along $\eta$. For any point $X$ on $\delta$ with distance $x$ from $M$, by a direct computation in hyperbolic geometry, the distance between $X$ and $\gamma$ satisfies 
$$(\sinh{d(X,\gamma)})^2=(\cosh{x})^2(\sinh{d})^2+(\sinh{x})^2(\sin{\theta})^2.$$

Without loss of generality, we can assume that $A$ and $W$ lies on the same component of $\delta\setminus \{M\}$, and $W$ is closer to $M$ than $A$. Let their distances from $M$ be denoted by $a$ and $w$ respectively, then $0\leq w\leq a-D$ holds. Since $d(A,\gamma)<1$, we have 
$$(\sinh{1})^2>(\cosh{a})^2(\sinh{d})^2+(\sinh{a})^2(\sin{\theta})^2.$$
By using $w\leq a-D$, we have 
\begin{align*}
&(\sinh{d(W,\gamma)})^2=(\cosh{w})^2(\sinh{d})^2+(\sinh{w})^2(\sin{\theta})^2\\
\leq\  & 4e^{-2D}((\cosh{a})^2(\sinh{d})^2+(\sinh{a})^2(\sin{\theta})^2)<(2e^{-D}\sinh{1})^2.
\end{align*}
So $d(W,\gamma)<\sinh{d(W,\gamma)}<2e^{-D}\sinh{1}<3e^{-D}$ holds.
\end{proof}

\begin{lemma}\label{anglechange}
Let $X,Y,Z,W$ be four points on $\mathbb{H}^3$, such that the lengths of geodesic segments $XY,YZ,ZW$ are all at least $D>3$. Also suppose that the angle between $XY,YZ$and the angle between $YZ,ZW$  (denoted by $\angle XYZ$ and $\angle YZW$ respectively) are at least $\delta$ for some $\delta>\frac{\pi}{4}$. Then the following statements hold.
\begin{enumerate}
\item The angle between $XY$ and $XZ$ is at most $20e^{-|XY|}$.
\item If the length of $YZ$ is at least $2D$, then the angle between $XY$ and $XW$ is at most $40e^{-D}$.
\end{enumerate}
\end{lemma}

\begin{proof}
For (1), by the cosine and sine laws in hyperbolic geometry, we have 
$$\cosh{|XZ|}=\cosh{|XY|}\cosh{|YZ|}-\sinh{|XY|}\sinh{|YZ|}\cos{\angle XYZ}$$
and $$\frac{\sinh{|XZ|}}{\sin{\angle XYZ}}=\frac{\sinh{|YZ|}}{\sin{\angle YXZ}}.$$
So we have
\begin{align*}
& (\sin{\angle YXZ})^2=\frac{(\sinh{|YZ|}\sin{\angle XYZ})^2}{(\cosh{|XY|}\cosh{|YZ|}-\sinh{|XY|}\sinh{|YZ|}\cos{\angle XYZ})^2-1}\\
\leq \ & \frac{(\sin{\angle XYZ})^2}{(\sinh{|XY|})^2(1-\cos{\angle XYZ})^2-1}\leq 2(\frac{\sin{\angle XYZ}}{\sinh{|XY|}(1-\cos{\angle XYZ})})^2.
\end{align*}
So $$\angle YXZ\leq\frac{\pi}{2}\sin{\angle YXZ} \leq\frac{\sqrt{2}\pi}{2}\frac{\sin{\angle XYZ}}{\sinh{|XY|}(1-\cos{\angle XYZ})} \leq \frac{\sqrt{2}\pi}{2-\sqrt{2}}\frac{1}{\sinh{|XY|}}<20e^{-|XY|}.$$

To prove (2), let $M$ be the middle point of $YZ$, then both $YM$ and $MZ$ have length at least $D$. By (1), the angles $\angle YXM, \angle XMY, \angle ZMW$ are all at most $20e^{-D}$. So the angle $\angle XMW$ is at least $\pi-20e^{-D}\times 2>\frac{\pi}{4}$. The length of $XM$ satisfies 
\begin{align*}
&\cosh{|XM|}=\cosh{|XY|}\cosh{|YM|}-\sinh{|XY|}\sinh{|YM|}\cos{\angle XYM}\\
\geq \ &\sinh{|XY|}\sinh{|YM|}(1-\cos{\angle XYM})\geq (1-\frac{\sqrt{2}}{2})(\sinh{D})^2\geq \cosh{D}.
\end{align*} So $|XM|$ is at least $D$, and so does $|MW|$. 

We use (1) again to deduce than $\angle MXW<20e^{-D}$. So $$\angle YXW\leq \angle YXM+\angle MXW<40e^{-D}.$$
\end{proof}

Before stating the next lemma, we need a few geometric definitions from Section 4 of \cite{LM}. 

\begin{definition}\label{chainsandcycles}
Let $0<\delta<\frac{\pi}{3}$, $L>0$ and $0<\theta<\pi$ be three constants.
\begin{enumerate}
\item Two oriented $\partial$-framed segments $\mathfrak{s}$ and $\mathfrak{s}'$ are {\it $\delta$-consecutive} if the terminal point of $\mathfrak{s}$ is the initial point of $\mathfrak{s}'$, and the terminal framing of $\mathfrak{s}$ is $\delta$-close to the initial framing of $\mathfrak{s}'$. The {\it bending angle} between $\mathfrak{s}$ and $\mathfrak{s}'$ is the angle between the terminal direction of $\mathfrak{s}$ and the initial direction of $\mathfrak{s}'$.
\item A {\it $\delta$-consecutive chain} of oriented $\partial$-framed segments is a finite sequence $\mathfrak{s}_1,\cdots,\mathfrak{s}_m$ such that $\mathfrak{s}_i$ is $\delta$-consecutive to $\mathfrak{s}_{i+1}$ for $i=1,\cdots,m-1$. It is a {\it $\delta$-consecutive cycle} if furthermore $\mathfrak{s}_m$ is $\delta$-consecutive to $\mathfrak{s}_1$. A $\delta$-consecutive chain or cycle is {\it $(L,\theta)$-tame} if each $\mathfrak{s}_i$ has length at least $2L$ and each bending angle is at most $\theta$.
\item For an $(L,\theta)$-tame $\delta$-consecutive chain $\mathfrak{s}_1,\cdots,\mathfrak{s}_m$, the {\it reduced concatenation}, denoted by $\mathfrak{s}_1\cdots \mathfrak{s}_m$, is the oriented $\partial$-framed segment defined as the following. The carrier segment of $\mathfrak{s}_1\cdots \mathfrak{s}_m$ is homotopic to the concatenation of carrier segments of $\mathfrak{s}_1,\cdots,\mathfrak{s}_m$, with respect to the endpoints. The initial and terminal framings of $\mathfrak{s}_1\cdots\mathfrak{s}_m$ are the closest unit normal vectors to the initial framing of $\mathfrak{s}_1$ and the terminal framing of $\mathfrak{s}_m$ respectively.
\item For an $(L,\theta)$-tame $\delta$-consecutive cycle $\mathfrak{s}_1,\cdots,\mathfrak{s}_m$, the {\it reduced cyclic concatenation}, denoted by $[\mathfrak{s}_1\cdots \mathfrak{s}_m]$, is the oriented closed geodesic freely homotopic to the cyclic concatenation of carrier segments of $\mathfrak{s}_1,\cdots,\mathfrak{s}_m$, assuming it is not null-homotopic.
\end{enumerate}
\end{definition}

We need the following lemma from \cite{LM}, which estimates the length and phase of a concatenation of oriented $\partial$-framed segments. Here the function $I(\cdot)$ is defined by $I(\theta)=2\ln{(\sec{\frac{\theta}{2}})}$.

\begin{lemma}\label{lengthphase} (Lemma 4.8 of \cite{LM})
Given any positive constants $\delta,\theta,L$ with $0<\theta<\pi$ and $L\geq I(\theta)+10\ln{2}$, the following statements hold in any oriented hyperbolic $3$-manifold.
\begin{enumerate}
\item If $\mathfrak{s}_1,\cdots,\mathfrak{s}_m$ is an $(L,\theta)$-tame $\delta$-consecutive chain of oriented $\partial$-framed segments, denoting the bending angle between $\mathfrak{s}_i$ and $\mathfrak{s}_{i+1}$ as $\theta_i\in[0,\theta)$, then 
$$\Big|l(\mathfrak{s}_1\cdots \mathfrak{s}_m)-\sum_{i=1}^ml(\mathfrak{s}_i)+\sum_{i=1}^{m-1}I(\theta_i)\Big|<\frac{(m-1)e^{(-L+10\ln{2})/2}\sin{(\theta/2)}}{L-\ln{2}}$$
and 
$$\Big|\varphi(\mathfrak{s}_1\cdots\mathfrak{s}_m)-\sum_{i=1}^m\varphi(\mathfrak{s}_i)\Big|<(m-1)(\delta+e^{(-L+10\ln{2})/2}\sin{(\theta/2)}),$$
where $|\cdot|$ on $\mathbb{R}/2\pi\mathbb{Z}$ is understood as the distance from zero valued in $[0,\pi]$.
\item If $\mathfrak{s}_1,\cdots,\mathfrak{s}_m$ is an $(L,\theta)$-tame $\delta$-consecutive cycle of oriented $\partial$-framed segments, denoting the bending angle between $\mathfrak{s}_i$ and $\mathfrak{s}_{i+1}$ as $\theta_i\in[0,\theta)$ with $\mathfrak{s}_{m+1}$ equal to $\mathfrak{s}_1$ by convention, then 
$$\Big|l([\mathfrak{s}_1\cdots \mathfrak{s}_m])-\sum_{i=1}^ml(\mathfrak{s}_i)+\sum_{i=1}^{m}I(\theta_i)\Big|<\frac{me^{(-L+10\ln{2})/2}\sin{(\theta/2)}}{L-\ln{2}}$$
and 
$$\Big|\varphi([\mathfrak{s}_1\cdots\mathfrak{s}_m])-\sum_{i=1}^m\varphi(\mathfrak{s}_i)\Big|<m(\delta+e^{(-L+10\ln{2})/2}\sin{(\theta/2)}),$$
where $|\cdot|$ on $\mathbb{R}/2\pi\mathbb{Z}$ is understood as the distance from zero valued in $[0,\pi]$.
\end{enumerate}
\end{lemma}

The following lemma bounds the distance between a $\delta$-consecutive cycle of geodesic segments and the corresponding closed geodesic, which will be used for bounding heights of closed geodesics arised from geometric constructions. In the following, an $(L,\theta)$-tame cycle of geodesic segments is a sequence of geodesic segments that satisfies the definition of an $(L,\theta)$-tame cycle of oriented $\partial$-framed segments except the $\delta$-closeness condition on framings.

\begin{lemma}\label{distance}
Given any positive constants $\theta,L$ where $0<\theta<\pi$ and $L\geq 4(I(\theta)+10\ln{2})$, the following statement holds in any oriented hyperbolic $3$-manifold. If $s_1,\cdots,s_m$ is an $(L,\theta)$-tame cycle of geodesic segments with $m\leq L$, denoting the bending angle between $s_i$ and $s_{i+1}$ by $\theta_i\in[0,\theta)$ with $s_{m+1}$ equal to $s_1$ by convention, then the closed geodesic $[s_1\cdots s_m]$ lies in the $1$-neighborhood of the union $\cup_{i=1}^m s_i$. Moreover, for
$$K=\max_{i=1,\cdots,m}{\Big\{\ln{\frac{1+\tan{\frac{\theta_i}{4}}}{1-\tan{\frac{\theta_i}{4}}}}\Big\}},$$ $[s_1\cdots s_m]$ and $\cup_{i=1}^m s_i$ lie in the 
$(K+\frac{1}{10})$-neighborhood of each other.
\end{lemma}

\begin{proof}
For $i=1,\cdots,m$, let $n_i$ be the middle point of $s_i$, and let $t_i$ be the geodesic segment from $n_i$ to $n_{i+1}$ (with natural identification $n_{m+1}=n_1$).

Then the closed geodesic $[s_1\cdots s_m]$ is same with the closed geodesic $[t_1\cdots t_m]$. By Lemma \ref{hyperbolicconstant}, $t_i$ lies in the $\ln{(\sqrt{2}+1)}$-neighborhood of $s_i\cup s_{i+1}$.

By Lemma \ref{lengthphase} (1), we have 
$$\Big|l(t_i)-\frac{1}{2}(l(s_i)+l(s_{i+1}))+I(\theta_i)\Big|<\frac{e^{(-L/2+10\ln{2})/2}\sin{(\theta/2)}}{L/2-\ln{2}}.$$ 
By Lemma \ref{lengthphase} (2), we have
$$\Big|l([t_1\cdots t_m])-\sum_{i=1}^ml(s_i)+\sum_{i=1}^m I(\theta_i)\Big|<\frac{me^{(-L/2+10\ln{2})/2}\sin{(\theta/2)}}{L/2-\ln{2}},$$ and $l([t_1\cdots t_m])\geq mL-mI(\theta)-m\geq 10\ln{2}$.
So  $$\Big|l([t_1\cdots t_m])-\sum_{i=1}^ml(t_i)\Big|<\frac{2me^{(-L/2+10\ln{2})/2}\sin{(\theta/2)}}{L/2-\ln{2}}\leq \frac{1}{400}.$$

We lift the closed geodesic $\gamma=[t_1\cdots t_m]$ to $\mathbb{H}^3$ to get a bi-infinite geodesic $l$, and lift the cyclic concatenation $t_1,\cdots,t_m$ to get a bi-infinite piecewise geodesic in $\mathbb{H}^3$ that is asymptotic to $l$ on both ends. For each $i$, we take one lift $\tilde{t}_i$ of $t_i$ in the above bi-inifinite piecewise geodesic, and denote its initial point by $z_i$. Then we want to bound $D_i=d(z_i,l)$.

For simplicity, we use $z$ and $D$ to denote $z_i$ and $D_i$ respectively. At first, we have $L_2=d(z,\gamma z)\leq \sum_{i=1}^ml(t_i)$ and the translation length of $\gamma$ on $l$ is $L_1=l([t_1\cdots t_m])$, with $0\leq L_2-L_1\leq \frac{1}{400}$ holds. By a direct computation in hyperbolic geometry, we have $$\cosh{L_2}\geq \cosh{L_1}(\cosh{D})^2-(\sinh{D})^2=(\cosh{L_1}-1)(\cosh{D})^2+1.$$

So 
$$(1+\frac{D^2}{2})^2\leq (\cosh{D})^2\leq \frac{\cosh{L_2}-1}{\cosh{L_1}-1}\leq e^{2(L_2-L_1)},$$ 
and 
$$D\leq \sqrt{2(e^{L_2-L_1}-1)}\leq \sqrt{2\times 2(L_2-L_1)}<\frac{1}{10}.$$
The fact $d(z,l)<\frac{1}{10}$ implies that $[t_1\cdots t_m]$ and $\cup_{i=1}^m t_i$ lie in the $\frac{1}{10}$-neighborhood of each other. Since $\ln{(\sqrt{2}+1)}<\frac{9}{10}$, $[t_1\cdots t_m]=[s_1\cdots s_m]$ lies in the $1$-neighborhood of $\cup_{i=1}^ms_i$.

For the moreover part, it follows from the fact that, in a hyperbolic triangle $ABC$, if $\angle BAC=\pi-\theta$, then the distance from $B$ to $AC$ is at most $\ln{\frac{1+\tan{\frac{\theta}{4}}}{1-\tan{\frac{\theta}{4}}}}$. So $BC$ and $AB\cup AC$ lie in the $(\ln{\frac{1+\tan{\frac{\theta}{4}}}{1-\tan{\frac{\theta}{4}}}})$-neighborhood of each other. 

The above observation implies that $\cup_{i=1}^m s_i$ and $\cup_{i=1}^mt_i$ lie in the $K=\max_{i=1,\cdots,m}{\{\ln{\frac{1+\tan{\frac{\theta_i}{4}}}{1-\tan{\frac{\theta_i}{4}}}}\}}$-neighborhood of each other. Then $[s_1\cdots s_m]=[t_1\cdots t_m]$ and $\cup_{i=1}^ms_i$ lie in the $(K+\frac{1}{10})$-neighborhood of each other.

\end{proof}

Now we prove a result on triangular generating sets of hyperbolic $3$-manifold groups. Although this result is not absolutely necessary for proving Theorem \ref{main1}, it makes the logic of the proof slightly simpler.

Let $G$ be a finitely presented group, and $S\subset G$ be a symmetric finite generating set of $G$. We say that $S$ is {\it a triangular generating set of $G$} if there is a presentation of $G$ whose generating set is $S$ and all relations have length at most $3$. It is equivalent to that, for any word of $S$ that represents the trivial element $e\in G$, it can be transformed to the trivial word of $S$ by relations of length at most $3$ (i.e. $s_1s_2=e$ or $s_1s_2s_3=e$ for $s_1,s_2,s_3\in S$). Note that every finite generating set of a finitely presentable group is contained in a finite triangular generating set. 

Let $M$ be a finite volume hyperbolic $3$-manifold, and let $*\in M$ be a base point. Then for each $g\in \pi_1(M,*)$, we define $|g|$ to be the length of the geodesic segment with endpoints at $*$ that represents $g$. Then we want to prove the following statement. 

\begin{proposition}\label{triangulargeneratingset}
Let $M$ be a finite volume hyperbolic $3$-manifold (closed or cusped), and let $*\in M$ be a base point. Then there exists $K>0$ only depends on $M$ and $*$, such that for any $k\geq K$, $$\mathscr{B}_k=\{g\in \pi_1(M,*)\ | \ |g|<k,\ g\ne e\}$$ is a triangular generating set of $\pi_1(M,*)$.
\end{proposition}

To prove Proposition \ref{triangulargeneratingset}, we first prove the following lemma.

\begin{lemma}\label{reducing}
Let $M$ be a finite volume hyperbolic $3$-manifold (closed or cusped), and let $*\in M$ be a base point. Then there exists $D>0$ only depend on $M$ and $*$, such that for any $g\in \pi_1(M,*)$ with $|g|\geq D$, there exist $h_1,h_2\in \pi_1(M,*)$, such that $g=h_1h_2$ and $|h_1|,|h_2|<|g|-\frac{1}{2}$.
\end{lemma}

\begin{proof}
Here we only prove this lemma for cusped hyperbolic $3$-manifolds, and the proof for the closed manifold case is actually easier. By shrinking the cusps of $M$ if necessary, we can assume that $*$ lies in the thick part of $M$. In this proof, we define the height function on $M$ with respect to these new cusps.

We take the boundaries of all cusps of $M$, and denote the maximum of their diameters (under the induced Euclidean metric) by $D_1$. Let $D_2$ be the diameter of the thick part of $M$, and let $$D_3=\max{\{\frac{1}{2}\ln{6}, \ln{(1+5D_1^2)}\}}.$$ Then we take $D=2(D_2+D_3)+2$. 

For any $g\in \pi_1(M,*)$ with $|g|\geq D$, we want to find the triangular relation $g=h_1h_2$ according to the height of $g$. 

{\bf Case I.} The height of $g$ is at least $D_3$. Let $x$ be one of the highest points on $g$, with height $h\geq D_3$, let $C$ be the cusp of $M$ that contains $x$ and let $T$ be its boundary. Then $g$ is a concatenation of four geodesic segments $g=g_1g_2g_3g_4$, where $x$ is the terminal point of $g_2$ and the initial point of $g_3$, while $g_2g_3$ is the intermediate cusp excursion of $g$ that contains $x$ and lies in $C$. Then Lemma \ref{cuspheightvslength} implies 
$$|g_2|=|g_3|=\ln{(e^h+\sqrt{e^{2h}-1})}=h+\ln{(1+\sqrt{1-e^{-2h}})}.$$ 

Without loss of generality, we assume that $|g_1|\leq |g_4|$.  Let $\alpha$ be the shortest geodesic segment from $x$ to $T$ (which is perpendicular to $T$), with $\alpha \cap T=y$, and let $z$ be the initial point of $g_2$ that lies in $T$. Since the diameter of $T$ is at most $D_1$, the distance between $y$ and $z$ is at most $D_1$ (under the Euclidean metric on $T$). Let $\beta$ be the shortest geodesic from $x$ to $z$, then a computation in hyperbolic geometry gives 
$$|\beta|<h\sqrt{1+(\frac{D_1}{e^{h}-1})^2}<h(1+\frac{D_1^2}{2(e^h-1)^2}).$$
So we have the following estimate $$|g_2|-|\beta|>\ln{(1+\sqrt{1-e^{-2h}})}-\frac{hD_1^2}{2(e^h-1)^2}>\ln{1.9}-0.1>\frac{1}{2}.$$ Here the second inequality holds by the choice of $D_3$ and the fact that $h>D_3$.

Then the concatenations $g_1g_2\beta g_1^{-1}$ and $g_1\beta^{-1}g_3g_4$ represent elements in $\pi_1(M,*)$ and they are the desired elements $h_1$ and $h_2$. The condition $g=h_1h_2$ obviously holds. Since $|g_1|\leq |g_4|$, $|g_2|=|g_3|$ and $|g_2|-|\beta|>\frac{1}{2}$, we have 
$$|g|-|h_1|\geq (|g_1|+|g_2|+|g_3|+|g_4|)-(|g_1|+|g_2|+|\beta|+|g_1|)\geq |g_3|-\beta>\frac{1}{2},$$ and 
$$|g|-|h_2|\geq (|g_1|+|g_2|+|g_3|+|g_4|)-(|g_1|+|\beta|+|g_3|+|g_4|)\geq |g_2|-\beta>\frac{1}{2}.$$

{\bf Case II.} If the height of $g$ is at most $D_3$, let $x$ be the middle point of $g$ and it divides $g$ to a concatenation $g=g_1g_2$ with $|g_1|=|g_2|$. Since the height of $x$ is at most $D_3$ and the diameter of the thick part of $M$ is $D_2$, there exists a geodesic segment $\beta$ from $x$ to $*$ of length at most $D_2+D_3$.

Then we have $$|g_1|-|\beta|\geq \frac{D}{2}-(D_2+D_3)>\frac{1}{2},$$ by the choice of $D$.
The desired elements $h_1,h_2\in \pi_1(M,*)$ are represented by concatenations $g_1\beta$ and $\beta^{-1}g_2$ respectively. The estimates are given by $$|g|-|h_1|>(|g_1|+|g_2|)-(|g_1|+|\beta|)=|g_2|-|\beta|>\frac{1}{2}$$ and $$|g|-|h_2|>(|g_1|+|g_2|)-(|\beta|+|g_2|)=|g_1|-|\beta|>\frac{1}{2}.$$ 
\end{proof}

Now we are ready to prove Proposition \ref{triangulargeneratingset}.

\begin{proof}[Proof of Proposition \ref{triangulargeneratingset}]

We take $D>0$ to be the constant given in Lemma \ref{reducing}, and take a finite generating set $S\subset \pi_1(M,*)$. Then $S\cup \mathscr{B}_{D}$ is also a finite generating set of $\pi_1(M,*)$. We extend $S\cup \mathscr{B}_{D}$ to a triangular generating set $S_{\text{tri}}$ of $\pi_1(M,*)$, then $S_{\text{tri}}$ is contained in $\mathscr{B}_K$ for some $K>0$, which is our desired constant $K$. Note that $K\geq D$ holds.

For any $k\geq K$, suppose that  we have a relation $g_1g_2\cdots g_n=e$ with $g_1,g_2,\cdots,g_n\in \mathscr{B}_k$. Then for any $g_i$ with $|g_i|\geq D$, Lemma \ref{reducing} gives us a triangular relation $g_i=h_i^{1}h_i^2$ in $\mathscr{B}_k$, with $|h_i^1|, |h_i^2|<|g_i|-\frac{1}{2}$. So we can transform the relation $g_1g_2\cdots g_n=e$ to another relation $h_1h_2\cdots h_m=e$ by triangular relations in $\mathscr{B}_k$, such that 
$$\max{\{|h_j|\ |\ j=1,\cdots,m\}}\leq \max{\{|g_i|\ |\ i=1,\cdots,n\}}-\frac{1}{2}.$$
Then we repeat this process inductively. Since the length of the longest geodesic segment decreases by $\frac{1}{2}$ each time, this process will end up with $k_1\cdots k_l=e$ such that $|k_i|<D$ for all $i=1,\cdots,l$.

Here we have $k_i\in  \mathscr{B}_{D}\subset S_{\text{tri}}$ holds for all $i$. Since $S_{\text {tri}}$ is a triangular generating set, the relation $k_1\cdots k_l=e$ can be transformed to the trivial word by triangular relations in $S_{\text{tri}}\subset \mathscr{B}_{K}\subset \mathscr{B}_{k}$. So $\mathscr{B}_{k}$ is a triangular generating set of $\pi_1(M,*)$.
\end{proof}

\bigskip
\bigskip

\section{Connection principle with height countrol}\label{connection_principle}

In this section, we will prove a connection principle with height control in cusped hyperbolic $3$-manifolds. It is an enhanced version of the connection principle in closed hyperbolic $3$-manifolds (e.g. Lemma 4.15 of \cite{LM}). Actually most ideas in the proof of Theorem \ref{connection_principle_with_height} are scattered in \cite{KW1} and \cite{KW2}. Since this result is crucial for this paper and we need a very precise height control, we give a full proof of our connection principle based on a more fundamental result in \cite{KW2} (Theorem 3.9 of \cite{KW2}).

For a point $p\in M$ and a unit tangent vector $\vec{t}_p\in T_pM$, we define a number $\alpha_{\vec{t}_p}\in[0,\frac{\pi}{2}]$ as the following.
\begin{itemize}
\item If $p$ lies in the thick part of $M$, or $p$ lies in the thin part and $\vec{t}_p$ points down the cusp, we define $\alpha_{\vec{t}_p}=0$. Here ``$\vec{t}_p$ points down the cusp'' means that, for the cusp whose horotorus boundary goes through $p$, $\vec{t}_p$ points outward the cusp (as a manifold with boundary).
\item If $p$ lies in the thin part of $M$ and $\vec{t}_p$ points up the cusp, we define $\alpha_{\vec{t}_p}\in [0,\frac{\pi}{2}]$ to be the angle between $\vec{t}_p$ and the horotorus going through $p$.
\end{itemize}

\begin{theorem}\label{connection_principle_with_height}
For any oriented cusped hyperbolic $3$-manifold $M$, there exist constants $T>0$, $C>2$ and $0<\kappa\leq 1$ only depend on $M$, such that for any $\delta\in (0,10^{-2})$ and any $t\geq \max{\{T,C(\ln{\frac{1}{\delta}}+1)\}}$, the following holds. 

Let $p,q\in M$ be two points in $M$ with height $h_p,h_q<\kappa t$. Let $\vec{t}_p,\vec{n}_p\in T_pM$ and $\vec{t}_q,\vec{n}_q\in T_qM$ be pairs of orthogonal unit vectors at $p$ and $q$ respectively. Then for any $\theta \in \mathbb{R}/2\pi \mathbb{Z}$, there exists an oriented $\partial$-framed segment $\mathfrak{s}$ from $p$ to $q$ such that the following holds.
\begin{enumerate}
\item The length and phase of $\mathfrak{s}$ are $\delta$-close to $t$ and $\theta$ respectively.
\item The initial direction and framing of $\mathfrak{s}$ are $\delta$-close to $\vec{t}_p$ and $\vec{n}_p$ respectively, the terminal direction and framing of $\mathfrak{s}$ are $\delta$-close to $\vec{t}_q$ and $\vec{n}_q$ respectively.
\item The initial and terminal cusp excursions of $\mathfrak{s}$ have height at most $$h_p+\min{\left\{\ln{(\sec{\alpha_{\vec{t}_p}})},\ln{\frac{1}{\delta}}+C\right\}}\ \text{and}\ h_q+\min{\left\{\ln{(\sec{\alpha_{(-\vec{t}_q)}})},\ln{\frac{1}{\delta}}+C\right\}}$$ respectively, if they exist.
\item All intermediate cusp excursions of $\mathfrak{s}$ have height at most $$\frac{1}{2}\ln{t}+C\ln{\frac{1}{\delta}}+C.$$
\end{enumerate}
\end{theorem}

The proof of Theorem \ref{connection_principle_with_height} uses Kahn and Wright's work on Lie groups in \cite{KW2}. To state Kahn and Wright's result, we need to set up some notations in \cite{KW2}, and we will specialize to the case that the Lie group $G=\text{PSL}(2,\mathbb{C})=\text{Isom}_+(\mathbb{H}^3)$. 

Since $G$ can be identified with the frame bundle of the $3$-dimensional hyperbolic space $\text{SO}(\mathbb{H}^3)$ (via choosing a base frame $f_0\in \text{SO}(\mathbb{H}^3)$), the hyperbolic metric on $\mathbb{H}^3$ induces a left-invariant Riemannian metric on $G$. This Riemannian metric induces a metric space structure on $G$ and we denote this metric by $d$.

Let $M$ be a cusped hyperbolic $3$-manifold, then it is identified with $\Gamma \backslash \mathbb{H}^3$ for a non-uniform lattice $\Gamma<G$. For any $h\in G$, we define $$\epsilon(h)=\min{(\frac{1}{2}\inf_{\gamma\in \Gamma\setminus \{1\}}d(h,\gamma h),1)}.$$

Let $\mathfrak{g}=\text{sl}(2,\mathbb{C})$ be the Lie algebra of $G$. Take $A=
\begin{pmatrix}
\frac{1}{2} & 0 \\
0 & -\frac{1}{2}
\end{pmatrix}\in \mathfrak{g}$. Under the adjoint action of $A$, $\mathfrak{g}$ splits as $\mathfrak{g}=\mathfrak{h}_-\oplus \mathfrak{h}_0\oplus \mathfrak{h}_+$, where $$\mathfrak{h}_-=\text{span}_{\mathbb{C}}
\begin{pmatrix}
0 & 0\\
1 &0
\end{pmatrix},\ \mathfrak{h}_0=\text{span}_{\mathbb{C}}
\begin{pmatrix}
\frac{1}{2} & 0\\
0 & -\frac{1}{2}
\end{pmatrix},\ \mathfrak{h}_+=\text{span}_{\mathbb{C}}
\begin{pmatrix}
0 & 1\\
0 & 0
\end{pmatrix}$$ are the eigenspaces of $\text{ad}_A$ corresponding to eigenvalues $-1,0,1$ respectively. All of $\mathfrak{h}_-, \mathfrak{h}_0,\mathfrak{h}_+$ are Lie subalgebras of $\mathfrak{g}$, and their corresponding Lie subgroups are $$H_-=\left\{
\begin{pmatrix}
1 & 0\\
\mu & 1
\end{pmatrix}\ \biggr\rvert\ \mu\in \mathbb{C}\right \}, H_0=\left\{
\begin{pmatrix}
\lambda & 0\\
0 & \lambda^{-1}
\end{pmatrix}\ \biggr\rvert\ \lambda \in \mathbb{C}\setminus \{0\}\right \}, H_+=\left\{
\begin{pmatrix}
1 & \mu\\
0 & 1
\end{pmatrix}\ \biggr\rvert\ \mu\in \mathbb{C}\right \}$$ respectively. Moreover, we have a constant $K_A=\text{tr}_{\mathbb{R}}(\text{ad}_A|_{\mathfrak{h}_+})$ and it equals $2$ in our case. The left-invariant Riemannian metric on $G$ induces left-invariant metrics on $H_-$, $H_0$ and $H_+$, and we denote the corresponding volume forms on $\mathfrak{h}_-,\mathfrak{h}_0, \mathfrak{h}_+$ by $\eta_{ \mathfrak{h}_-}, \eta_{ \mathfrak{h}_0},\eta_{ \mathfrak{h}_+}$, respectively. 

Next, we consider the following Lie subalgebra of $\mathfrak{g}$: $$\mathfrak{e}_+=\mathfrak{e}_-=\text{span}_{\mathbb{R}}\left\{
\begin{pmatrix}
i & 0\\
0 & -i
\end{pmatrix}, \begin{pmatrix}
0 & i\\
i & 0
\end{pmatrix},\begin{pmatrix}
0 & -1\\
1 & 0
\end{pmatrix}
\right\}.$$ We take two copies of the Lie subalgebra here since $\mathfrak{e}_+$ may not be same with $\mathfrak{e}_-$ in the general setting of \cite{KW1}. Here we need condition that the projection $\pi_{\mathfrak{h}_+}:\mathfrak{e}_+ \to \mathfrak{h}_+$ ($\pi_{\mathfrak{h}_-}:\mathfrak{e}_- \to \mathfrak{h}_-$) induced by the direct sum decomposition $\mathfrak{g}=\mathfrak{h}_-\oplus \mathfrak{h}_0\oplus \mathfrak{h}_+$ is surjective and the kernel is contained in $\mathfrak{h}_0$. The Lie subgroup $E_+=E_-$ corresponding to $\mathfrak{e}_+=\mathfrak{e}_-$ is the stabilizer of the point $(0,0,1) \in \mathbb{H}^3$, under the upper-half space model. So $E_+=E_-\cong \text{SO}(3)$ holds.

Let $E_{0\pm}=E_{\pm}\cap H_0=\left\{
\begin{pmatrix}
e^{i\theta} & 0 \\
0 & e^{-i\theta}
\end{pmatrix}\ \biggr\rvert\ \theta \in \mathbb{R} \right\}$ and let $E_0=H_0$, then we define $$E=E_-\times E_0\times E_+/\sim,$$ with $(e_-e_{0-},e_0,e_{0+}e_+)\sim (e_-,e_{0-}e_0e_{0+},e_+)$ for $e_{0-}\in E_{0-},e_{0+}\in E_{0+}$. 

Now we define a measure on $E$. Let $\mathfrak{e}_{0\pm}$ be the Lie algebras of $E_{0\pm}$, then they are also kernels of $\pi_{\mathfrak{h}_{\pm}}:\mathfrak{e}_{\pm}\to \mathfrak{h}_{\pm}$. Let $\hat{\mathfrak{e}}_{\pm}$ be a complement of $\mathfrak{e}_{0\pm}$ in $\mathfrak{e}_{\pm}$. For example, we can take $\hat{\mathfrak{e}}_{\pm}=\text{span}_{\mathbb{R}}\left\{
\begin{pmatrix}
0 & i\\
i & 0
\end{pmatrix}, \begin{pmatrix}
0 & -1\\
1 & 0
\end{pmatrix}
\right\}.$ Then $\pi_{\mathfrak{h}_{\pm}}|_{\hat{\mathfrak{e}}_{\pm}}:\hat{\mathfrak{e}}_{\pm}\to \mathfrak{h}_{\pm}$ is an isomorphism, and we denote $\eta_{\hat{\mathfrak{e}}_{\pm}}=(\pi_{\mathfrak{h}_{\pm}}|_{\hat{\mathfrak{e}}_{\pm}})^{-1}(\eta_{\mathfrak{h}_{\pm}})$. In our case, $\eta_{\hat{\mathfrak{e}}_{\pm}}$ equals $\frac{1}{4}$ of the natural volume form on $\hat{\mathfrak{e}}_{\pm}$. Then $\eta_{\hat{\mathfrak{e}}_-}\wedge \eta_{\mathfrak{h}_0}\wedge \eta_{\hat{\mathfrak{e}}_+}$ defines a volume form on $T_eE$ and it is independent of the choice of $\hat{\mathfrak{e}}_{\pm}$, where $e$ is the equivalent class of the identity element of $E_-\times E_0\times E_+$. By left translations of $E_-, E_0$ and right translation of $E_+$, this volume form on $T_0E$ gives a volume form $\eta_E$ on $E=E_-\times E_0\times E_+/\sim$. This volume form is well-defined since both $E_+$ and $E_-$ are compact and $E_0$ is abelian, and any left-invariant form on each of them is also right-invariant. The volume form $\eta_E$ also induces a measure on $E$ and we denote it by the same notation.

Now we are ready to state a result in \cite{KW2}, which is the main ingredient for proving Theorem \ref{connection_principle_with_height}.

\begin{theorem}\label{counting}(\cite{KW2} Theorem 3.9)
Let $K\subset E$ be compact, there exists constants $t_0=t_0(E)>0$, $a=a(E)>0$ only depend on $E$, $q=q(M)>0$ only depends on the manifold $M$ and $c=c(K,M)>0$ only depends on $K$ and $M$, such that the following hold.

For any $S\subset K$, let $$S_t=\{a_- \exp{(tA)} a_0a_+\ |\ [(a_-,a_0,a_+)]\in S \}.$$ For any $t\geq t_0$, we denote $\delta'=ce^{-aq t}$. Then for any pair of elements $g,h\in G$ with $\epsilon(g),\epsilon(h)>\delta'$, we have
$$(1-\delta')\eta_{E}(\mathcal{N}_{-\delta'}(S))<e^{-tK_A}\#(S_t\cap g^{-1}\Gamma h)<(1+\delta')\eta_{E}(\mathcal{N}_{\delta'}(S)).$$ Here $\mathcal{N}_{-\delta'}(S)$ and $\mathcal{N}_{\delta'}(S)$ denote the inner and outer $\delta'$-neighborhoods of $S$ in $E$ respectively.
\end{theorem} 
Note that $K_A=2$ holds in our case.

\begin{remark}
Here we state this result as counting $\#(S_t\cap g^{-1}\Gamma h)$ while \cite{KW2} Theorem 3.9 is stated as counting $\#(S_t\cap g\Gamma h)$. The author believes that it is a typo in \cite{KW2}, since Theorem 3.9 of \cite{KW2} is proved by citing Theorem 1.3 of \cite{KW2}  and the function $\Sigma_t(f,r,s)$ in Theorem 1.3 is defined by $$\Sigma_t(f,r,s)=\sum_{\gamma\in \Gamma}((Z_t)_*f)(r^{-1}\gamma s),$$ instead of  $$\Sigma_t(f,r,s)=\sum_{\gamma\in \Gamma}((Z_t)_*f)(r\gamma s).$$
\end{remark}

Now we are ready to prove Theorem \ref{connection_principle_with_height}.

\begin{proof}[Proof of Theorem \ref{connection_principle_with_height}]

Recall that $E_0=H_0\cong \mathbb{R}\times S^1$, and both $E_+$ and $E_-$ are compact. We take a compact set $K'=E_-\times ([-1,1]\times S^1) \times E_+\subset E_-\times E_0\times E_+$, and let $K$ be the image of $K'$ in $E$. For this compact set $K\subset E$, Theorem \ref{counting} gives us constants $t_0, a, q, c>0$.

We take the desired constants as following:
$$T=\max{\left\{t_0, \frac{2}{aq}\ln{\frac{c}{c_1}},\frac{1}{aq}\ln{\frac{2c}{l}}\right\}},$$ 
$$C=\max{\left\{\frac{1}{aq}, \frac{1}{aq}\ln{(8c)}, 20\ln{2}+\frac{1}{2}\ln{c_4},20\ln{2}\right\}}, $$ and $$\kappa=\min{\{\frac{aq}{2},1\}}.$$ Here $t_0,a,q,c>0$ are constants in Theorem 4.2, $l$ is the length of the shortest closed geodesic in $M$, while $c_1$ and $c_4$ are universal constants that will show up in the following proof. Actually both $a$ and $q$ in Theorem \ref{counting} are small positive numbers, so $\frac{aq}{2}<1$ would automatically holds. Note that $\delta\in (0,10^{-2})$ implies that $\delta$ is smaller than the $3$-dimensional Margulis constant, by \cite{Mey}.

{\bf Step I.} Recall that the identification between $G$ and $\text{SO}(\mathbb{H}^3)$ is made via a base frame $f_0\in \text{SO}(\mathbb{H}^3)$. We take elements $g,h\in G$ such that the projections of $g f_0$ and $hf_0\in \text{SO}(\mathbb{H}^3)$ in $\text{SO}(M)$ are $(p,\vec{t}_p,\vec{n}_p)$ and $(q,\vec{t}_p,\vec{n}_q)$ respectively. Note that $g$ and $h$ are unique up to left multiplying by elements in $\Gamma$. Since $h_p,h_q<\kappa t$, the injectivity radii of $p$ and $q$ are bounded below by $c_1e^{-\kappa t}$ for some universal constant $c_1>0$. Then we claim that $\epsilon(g),\epsilon(h)>ce^{-aqt}$ holds, thus the assumption in Theorem \ref{counting} holds. Actually, if $p$ lies in the thin part of $M$, we have $\epsilon(g)\geq c_1e^{-\kappa t}>\delta'=ce^{-aqt}$; if $p$ lies in the thick part, we have $\epsilon(g)\geq \frac{l}{2}>\delta'=ce^{-aqt}$.

We use $e$ to denote the identity element in a group, and use $B_r(g,G)$ to denote the $r$-neighborhood of $g\in G$ in a group $G$ (with a fixed metric). We take $S$ to be the image of $B_{\delta}(e, E_-)\times B_{\delta}((0,\theta), E_0)\times B_{\delta}(e, E_+) \subset K'$ in $E$. Here the metrics on $E_{\pm}$ and $E_0$ are the natural metrics on $\text{SO}(3)$ and $S^1\times \mathbb{R}$, respectively.

Since $t\geq T\geq t_0$, Theorem \ref{counting} gives 
$$\#(S_t\cap g^{-1}\Gamma h)>e^{tK_A}(1-\delta')\eta_{E}(\mathcal{N}_{-\delta'}(S))>\frac{1}{2}e^{2t}(1-\delta')(\frac{1}{4})^2(\delta-\delta')^6>2^{-7}e^{2t}\delta^6.$$ 
Here we use the fact that  $t>C(\ln{\frac{1}{\delta}}+1)$ to deduce that $\delta>8\delta'=8ce^{-aqt}$.

By definition, each element in $S_t\cap g^{-1}\Gamma h$ gives an element $\gamma \in \Gamma$, such that $$g^{-1}\gamma h=a_-e^{tA}a_0a_+$$ for some $a_-\in B_{\delta}(e,E_-), a_0\in B_{\delta}((0,\theta),H_0)$ and $a_+\in B_{\delta}(e,E_+)$. So for the base frame $f_0\in \text{SO}(\mathbb{H}^3)$, we have $$\gamma h f_0=ga_-e^{tA}a_0a_+f_0.$$ It implies that, in the manifold $M$, the frame $(q,\vec{t}_q,\vec{n}_{q})$ (corresponding to $\gamma hf_0$) can be obtained from the frame $(p,\vec{t}_p,\vec{n}_{p})$ (corresponding to $gf_0$) by the following operations:
\begin{itemize}
\item Rotate by $a_-\in B_{\delta}(e,E_-)=B_{\delta}(e,\text{SO}(3))$ with respect to the frame $gf_0$, to get $ga_-f_0$.
\item Translation by $e^{tA}a_0$ with respect to the frame $ga_-f_0$, with $a_0\in B_{\delta}((0,\theta),H_0)=B_{\delta}((0,\theta),\mathbb{R}\times S^1)$, to get $ga_-e^{tA}a_0f_0$. (Note that translating by $e^{tA}$ gives a geodesic segment of length $t$, by the choice of $A$.)
\item Rotate by $a_+\in B_{\delta}(e,E_+)=B_{\delta}(e,\text{SO}(3))$ with respect to the frame $ga_-e^{tA}a_0f_0$, to get $ga_-e^{tA}a_0a_+f_0=\gamma hf_0$.
\end{itemize}

For any $\gamma \in \Gamma$, the expression $g^{-1}\gamma h=a_-e^{tA}a_0a_+$ is not unique, since $E$ is defined by quotient an equivalence relation. However, each element in $S_t\cap g^{-1}\Gamma h$ gives a unique oriented $\partial$-framed segment $\mathfrak{s}$ from $p$ to $q$ satisfying conditions (1) and (2), up to $(2\delta)$-closeness of framings. Let $t_0\in (-\delta,\delta)$ be the $\mathbb{R}$-component of $a_0\in B_{\delta}((0,\theta),\mathbb{R}\times S^1)$. Then the carrier segment of $\mathfrak{s}$ is the geodesic segment tangent with the tangent vector of $ga_-f_0$ and has length $t+t_0$.
The initial and terminal framings of $\mathfrak{s}$ are the projections of normal vecotrs of $ga_-f_0$ and $ga_-e^{tA}a_0f_0=\gamma ha_+^{-1}f_0$, respectively. 

Moreover, each such $\partial$-framed segment $\mathfrak{s}$ is uniquely determined by its carrier segment $s$, and two different elements in $S_t\cap g^{-1}\Gamma h$ correspond to two different geodesic segments. So we have at least $2^{-7}e^{2t}\delta^6$ oriented $\partial$-framed segments satisfying conditions (1) and (2) that are distinct from each other, even up to $(2\delta)$-closeness of framings. 

\bigskip

{\bf Step II.} Now we count the number of oriented $\partial$-framed segments whose initial and terminal cusp excursions have height at most $h_p+\min{\{\ln{(\sec{\alpha_{\vec{t}_p}})},\ln{\frac{1}{\delta}+C}\}}$ and $h_q+\min{\{\ln{(\sec{\alpha_{(-\vec{t}_p)}})},\ln{\frac{1}{\delta}+C}\}}$ respectively. 

For the initial and terminal cusp excursions, we run an argument as in Corollary 3.6 of \cite{KW1}. We take subsets of $B_{\delta}(e,E_-)$ (and $B_{\delta}(e,E_+)$) as the following, to get $S'\subset S$, such that each element in $S'_t\cap g^{-1}\Gamma h$ gives an oriented $\partial$-framed segment whose initial and terminal cusp excursions satisfy condition (3).
\begin{enumerate}
\item[(a)] If $p$ lies in the thick part of $M$, then we simply take $A_-$ to be $B_{\delta}(e,E_-)$.
\item[(b)] If $p$ lies in the thin part of $M$ and $\vec{v}_p$ points down the cusp, we take half of $B_{\delta}(e,E_-)$ and denote it by $A_-$, such that for any $a_-\in A_-$, the tangent vector of $ga_-f_0$ still points down the cusp.
\item[(c)] If $p$ lies in the thin part of $M$ and $\vec{v}_p$ points up the cusp with $\alpha_{\vec{t}_p}<\frac{\pi}{2}-\frac{\delta}{2}$,  we take half of $B_{\delta}(e,E_-)$ and denote it by $A_-$, such that for any $a_-\in A_-$, the angle between the tangent vector of $ga_-f_0$ and the horo-torus is at most $\alpha_{\vec{t}_p}$.
\item[(d)] If $p$ lies in the thin part of $M$ and $\vec{v}_p$ points up the cusp with $\alpha_{\vec{t}_p}\in[\frac{\pi}{2}-\frac{\delta}{2},\frac{\pi}{2}]$, we take the subset $A_-\subset B_{\delta}(e,E_-)$ such that for any $a_-\in A_-$, the angle between the tangent vector of $ga_-f_0$ and the horo-torus is at most $\frac{\pi}{2}-\frac{\delta}{2}$.
\end{enumerate}

In either case, $A_-$ contains a ball of radius $\frac{\delta}{4}$. The choice of this ball is obvious for cases (a), (b) and (c). For (d), there exists some $a_-\in B_{\frac{3\delta}{4}}(e,E_-)$ such that the angle between the tangent vector of $ga_-f_0$ and the horo-torus is $\frac{\pi}{2}-\frac{3\delta}{4}$, and $ga_-f_0$ is the closest such frame from $gf_0$. Then the $\frac{\delta}{4}$-neighborhood of $a_-$ is contained in $A_-$.

In case (a), for any geodesic segment starting at $p$ and tangent with the tangent vector of  $ga_-f_0$ with $a_-\in A_-$, it has no initial cusp excursion. 

In case (b), for any geodesic segment starting at $p$ and tangent with the tangent vector of  $ga_-f_0$ with $a_-\in A_-$, its initial cusp excursion has height at most $h_p$. 

In case (c),  for any geodesic segment starting at $p$ and tangent with the tangent vector of  $ga_-f_0$ with $a_-\in A_-$, its initial cusp excursion has height bounded by $h_p+\ln{(\sec{\alpha_{\vec{t}_p}})}$, by a simple computation in hyperbolic geometry (Remark 3.7 of \cite{KW1}). Since $\alpha_{\vec{t}_p} <\frac{\pi}{2}-\frac{\delta}{2}$, we have $$\ln{(\sec{\alpha_{\vec{t}_p}})}<\ln{\frac{1}{\sin{\frac{\delta}{2}}}}<\ln{(\frac{4}{\delta})}<\ln{\frac{1}{\delta}}+C.$$

In case (d),  for any $a_-\in A_-$, the angle between the tangent vector of  $ga_-f_0$ and the horo-torus is at most $\frac{\pi}{2}-\frac{\delta}{2}\leq \alpha_{\vec{t}_p}$. So for the geodesic segment starting at $p$ and tangent with the tangent vector of $ga_-f_0$, its initial cusp excursion has height bounded by both $h_p+\ln{(\sec{\alpha_{\vec{t}_p}})}$ and $h_p+\ln{(\sec{(\frac{\pi}{2}-\frac{\delta}{2})})}\leq h_p+\ln{\frac{4}{\delta}}\leq h_p+\ln{\frac{1}{\delta}}+C$.

So in either case, the height of the initial cusp excursion is bounded by $h_p+\min{\{\ln{(\sec{\alpha_{\vec{t}_p}}),\ln{\frac{1}{\delta}}+C}\}}$. We also apply the same process for $q$ to get $A_+\subset B_{\delta}(e,E_+)$, by considering various cases of $\alpha_{(-\vec{t}_q)}$.

Now we count the number of oriented $\partial$-framed segments coming from $A_-\times B_{\delta}((0,\theta),H_0)\times A_+\subset E_-\times E_0\times E_+$, and we denote the image of $A_-\times B_{\delta}((0,\theta),H_0)\times A_+$ in $E$ by $S'$. Since both $A_-$ and $A_+$ contain $\frac{\delta}{4}$-balls, the same computation as in Step I gives
$$\#(S'_t\cap g^{-1}\Gamma h)>e^{2t}(1-\delta')\eta_E(\mathcal{N}_{-\delta'}(S'))>\frac{1}{2}e^{2t}(1-\delta')(\frac{1}{4})^2(\frac{\delta}{4}-\delta')^6\geq2^{-30}e^{2t}\delta^6.$$
 Here we use the fact that $\delta>8\delta'$. Now we have $2^{-30}e^{2t}\delta^6$ many distinct oriented $\partial$-framed segments satisfying conditions (1), (2) and (3) in the theorem (even up to $(2\delta)$-closeness of framings).

\bigskip

{\bf Step III.} For the intermediate cusp excursions, we run an argument as in Lemma 3.1 of \cite{KW1} to study their heights. We will prove that, among the above $2^{-30}e^{2t}\delta^6$ oriented $\partial$-framed segments obtained in step II, for at least one of them (actually many of them), all of its intermediate cusp excursions have height at most $h_0=\frac{1}{2}\ln{t}+C\ln{\frac{1}{\delta}}+C$.

For any oriented $\partial$-framed segment $\mathfrak{s}$ from $p$ to $q$ of length $l\in (t-\delta,t+\delta)$, it is uniquely determined by its carrier segment $s$ (up to $(2\delta)$-closeness of framings). So we only need to study the height of intermediate excursions of $s$. We suppose that $s$ has an intermediate cusp excursion of height in $[h,h+1]$ with $h>2$, then we write $s$ as a concatenation $s=s_-s_0s_+$ where $s_0$ is an intermediate cusp excursion of height in $[h,h+1]$. Let $T$ be the horo-torus of height $0$ going through the endpoints of $s_0$. 

Then we replace $s_-$ by the shortest geodesic segment $\gamma_-$ from $p$ to $T$ that is homotopic to $s_-$ relative to $p$ and $T$, and we replace $s_+$ by the shortest geodesic segment $\gamma_+$ from $T$ to $q$ similarly. Then we take $\gamma_0$ to be the geodesic segment from the terminal point of $s_-$ to the initial point of $s_+$, such that $\gamma_-\gamma_0\gamma_+$ is homotopic to $s_-s_0s_+$.

Since the height of $s_0$ lies in $[h,h+1]$, by a computation in hyperbolic geometry, the distance between the initial points of $s_0$ and $\gamma_0$ is at most $e^{-h}$, and so does  the distance between the terminal points of $s_0$ and $\gamma_0$. By Lemma \ref{exponentialdecay}, the height of $s_0$ and $\gamma_0$ differ by at most $3e^{-h}$. So the length of $\gamma_-\gamma_0\gamma_+$ lies in $[t-\delta,t+\delta+4e^{-h}]\subset [t-1,t+1]$, and its height lies in $[h-3e^{-h},h+1+3e^{-h}]\subset [h-1,h+2]$.

Instead of counting geodesics segments from $p$ to $q$ of length $\delta$-close to $t$ and has an intermediate cusp excursion of height in $[h,h+1]$, we will count the number of such concatenations $\gamma_-\gamma_0\gamma_+$ with the height of $\gamma_0$ lying in $[h-1,h+2]$. 

If $\gamma_-$ has length in [l,l+1], then since the length of $\gamma_0$ lies in $[2h-2,2h+6]$ (by Lemma \ref{cuspheightvslength}), the length of $\gamma_+$ lies in $[t-2h-l-8,t-2h-l+3]$. 

In $\mathbb{H}^3$, the area of a radius-$r$ sphere is $4\pi\sinh^2{r}$. For each horosphere tangent with the radius-$r$ sphere, its projection to the sphere has area $\frac{1-\sqrt{1-e^{-2r}}}{2}\cdot 4\pi\sinh^2{r}$. So there are at most $$\frac{2}{1-\sqrt{1-e^{-2r}}}<4e^{2r}$$ many disjoint horospheres tangent with the radius-$r$ sphere. This computation implies that, for any point in $\mathbb{H}^3$, its $r$-neighborhood intersects with at most $4e^{2r}$ many disjoint horoballs.

So the possible number of $\gamma_-$ is at most $4e^{2(l+1)}$ and the possible number of $\gamma_+$ is at most $4e^{2(t-2h-l+3)}$. Given the terminal point of $\gamma_-$ and the initial point of $\gamma_+$, we count the number of possible choices of $\gamma_0$. Since the height of $\gamma_0$ lies in $[h-1,h+2]$, if we homotopy $\gamma_0$ to a geodesic segment on the Euclidean horo-torus $T$, this geodesic segment has length at most $2e^{h+2}$. For the horo-torus $T$, its injectivity radius is at least some universal constant $c_2>0$ (related to the $3$-dimensional Margulis constant). By a comparison of area, the possible choice of $\gamma_0$ is at most $$\frac{\pi(2e^{h+2}+c_2)^2}{\pi c_2^2}=e^{2(h+2)}(\frac{2}{c_2}+e^{-(h+2)})^2.$$

So the number of $\gamma_-\gamma_0\gamma_+$ is bounded by $$4e^{2(l+1)}\cdot 4e^{2(t-2h-l+3)}\cdot e^{2(h+2)}(\frac{2}{c_2}+e^{-(h+2)})^2<c_3e^{2(t-h+6)},$$ where $c_3=16(\frac{2}{c_2}+1)^2$.

Here we have the condition that $\gamma_-$ has length in $[l,l+1]$ and $\gamma_0$ has height in $[h,h+1]$. We first sum over all possible lengths of $\gamma_-$, which runs from $0$ to $t$, we get that the total possible number of $\gamma_-\gamma_0\gamma_+$ is at most $c_3te^{2(t-h+6)}$. Then we sum over all possible heights of $\gamma_0$, which runs from $\frac{1}{2}\ln{t}+C\ln{\frac{1}{\delta}}+C$ to $\infty$, the possible number of $\gamma_-\gamma_0\gamma_+$ is at most 
$$c_3te^{2(t+6)}e^{-2(\frac{1}{2}\ln{t}+C\ln{\frac{1}{\delta}}+C)}\frac{1}{1-e^{-2}}=c_4e^{-2C}\delta^{2C}e^{2t}<2^{-40}e^{2t}\delta^6.$$ 
Here $c_4=\frac{c_3e^{12}}{1-e^{-2}}$. We also use that $C>3$ and $C>20\ln{2}+\frac{1}{2}\ln{c_4}$. 

Since each oriented $\partial$-framed segment satisfying conditions (1), (2), (3) is uniquely determined by its carrier segment (up to $(2\delta)$-closeness of framings), in the $2^{-30}e^{2t}\delta^6$ many oriented $\partial$-framed segments obtained in Step II, at most $2^{-40}e^{2t}\delta^6$ many of them do not satisfy the desired height control of intermediate cusp excursions. So we have $$2^{-30}e^{2t}\delta^6-2^{-40}e^{2t}\delta^6>2^{-31}e^{2t}\delta^6>1$$ many oriented $\partial$-framed segments that satisfy all the conditions in this theorem. Here we use that $t>C(\ln{\frac{1}{\delta}}+1)$ and $C>20\ln{2}$.

\end{proof}

\bigskip
\bigskip

\section{Geometric constructions}\label{construction}

In this section, we give several geometric constructions in cusped hyperbolic $3$-manifolds. These constructions are parallel with constructions in Section 4.4 of \cite{LM}. Some proofs are similar to the proofs in \cite{LM}, while some need new ideas to take care of cusps in the manifold. We give full proofs in all the constructions, since each construction need (at least) an extra height control comparing to \cite{LM}.

In this section, we always assume the following inequalities hold for constants $\epsilon,\delta,L,R>0$. To obtain these constants, we first choose $\epsilon$, then choose $\delta$, then choose $L$, finally choose $R$. Note that $\epsilon,\delta$ are very small and $L,R$ are very large.
\begin{align}\label{epsilondelta1}
0<10^9\delta<\epsilon<10^{-2},
\end{align}
\begin{align}\label{L1}
L>100,
\end{align}
\begin{align}\label{R1}
R>\max{\{10^3L,100T,100C(\ln{\frac{1}{\delta}+1}),10(\frac{e}{\delta})^{2C},10^3\}},
\end{align}
\begin{align}\label{R2}
R\geq \kappa R>100(\alpha+10^3)\ln{R}+100,
\end{align}
 Here $T>0,C>2,\kappa\in (0,1]$ are constants given in Theorem \ref{connection_principle_with_height}, and $\alpha$ is a fixed constant that is at least $4$.

For almost every lemma in this section, we need to define a geometric object in the hyperbolic space. All these definitions were originally given in Section 4 of \cite{LM}, and we give such a definition whenever we need it in the next lemma.

\begin{definition}\label{bigon}
An {\it $(L,\delta)$-tame bigon} is an $(L,\delta)$-tame $\delta$-consecutive cycle of two oriented $\partial$-framed segments $\mathfrak{a},\mathfrak{b}$ of phase $\delta$-close to $0$. It is said to be {\it $(l,\delta)$-nearly regular} if the lengths of $\mathfrak{a}$ and $\mathfrak{b}$ are both $\delta$-close to $l$. For simplicity, we will also say the reduced cyclic concatenation $[\mathfrak{a}\mathfrak{b}]$ is an $(L,\delta)$-tame bigon.
\end{definition}

The following lemma corresponds to Construction 4.17 of \cite{LM}.

\begin{lemma}\label{split}(Splitting)

Suppose that $h\in [\alpha\ln{R},(\alpha+100)\ln{R}]$. Suppose that $[\mathfrak{s}\mathfrak{s}']$ is an $(L,\delta)$-tame bigon such that the length and phase of $\mathfrak{s}$ and $\mathfrak{s}'$ are $\delta$-close to each other, and suppose $[\mathfrak{s}\mathfrak{s}']\in{\bf \Gamma}_{R,\epsilon}^{h}$. Then there is a pair of pants $\Pi\in {\bf \Pi}_{R,\epsilon}^{h+1}$, with one cuff $[\mathfrak{s}\mathfrak{s}']\in {\bf \Gamma}_{R,\epsilon}^{h}$ and the other two cuffs in ${\bf \Gamma}_{R,10\delta}^{h+1}$.
\end{lemma}

\begin{proof}

For the sake of height control, the undesired scenario is that the vectors $\vec{t}_{\text{ter}}(\mathfrak{s})\times \vec{n}_{\text{ter}}(\mathfrak{s})$ and $\vec{t}_{\text{ini}}(\mathfrak{s})\times \vec{n}_{\text{ini}}(\mathfrak{s})$ point up in cusps of $M$ (if the end points of $\mathfrak{s}$ lie in cusps of $M$).

For any angle $\phi\in \mathbb{R}/2\pi\mathbb{Z}$, the $\phi$-frame rotation $\mathfrak{s}(\phi)$ of $\mathfrak{s}$ satisfies
$$\vec{t}_{\text{ini}}(\mathfrak{s}(\phi))\times \vec{n}_{\text{ini}}(\mathfrak{s}(\phi))=-\vec{n}_{\text{ini}}(\mathfrak{s})\sin{\phi}+(\vec{t}_{\text{ini}}(\mathfrak{s})\times \vec{n}_{\text{ini}}(\mathfrak{s}))\cos{\phi}$$ 
and 
$$\vec{t}_{\text{ter}}(\mathfrak{s}(\phi))\times \vec{n}_{\text{ter}}(\mathfrak{s}(\phi))=-\vec{n}_{\text{ter}}(\mathfrak{s})\sin{\phi}+(\vec{t}_{\text{ter}}(\mathfrak{s})\times \vec{n}_{\text{ter}}(\mathfrak{s}))\cos{\phi}.$$ 
So the set of angles $\phi$ such that $\vec{t}_{\text{ini}}(\mathfrak{s}(\phi))\times \vec{n}_{\text{ini}}(\mathfrak{s}(\phi))$ points up the cusp is contained in an open interval in $\mathbb{R}/2\pi \mathbb{Z}$ of length $\pi$, and so does the set of $\phi$  such that $\vec{t}_{\text{ter}}(\mathfrak{s}(\phi))\times \vec{n}_{\text{ter}}(\mathfrak{s}(\phi))$ points up the cusp. So there exists an angle $\phi_0$ such that neither $\vec{t}_{\text{ini}}(\mathfrak{s}(\phi_0))\times \vec{n}_{\text{ini}}(\mathfrak{s}(\phi_0))$ nor $\vec{t}_{\text{ter}}(\mathfrak{s}(\phi_0))\times \vec{n}_{\text{ter}}(\mathfrak{s}(\phi_0))$ point up the cusp.

By rotating the oriented $\partial$-framed segments $\mathfrak{s}$ and $\mathfrak{s}'$ by $\phi_0$, we  get an $(L,3\delta)$-tame bigon $[\mathfrak{s}(\phi_0)\mathfrak{s}'(\phi_0)]$.

Then we apply Theorem \ref{connection_principle_with_height} to construct an oriented $\partial$-framed segment $\mathfrak{m}$ from $p_{\text{ter}}(\mathfrak{s})$ to $p_{\text{ini}}(\mathfrak{s})$ such that the following hold.
\begin{enumerate}
\item The initial and terminal directions of $\mathfrak{m}$ are $\delta$-close to $\vec{t}_{\text{ter}}(\mathfrak{s}(\phi_0))\times \vec{n}_{\text{ter}}(\mathfrak{s}(\phi_0))$ and $-\vec{t}_{\text{ini}}(\mathfrak{s}(\phi_0))\times \vec{n}_{\text{ini}}(\mathfrak{s}(\phi_0))$ respectively. The initial and terminal framings of $\mathfrak{m}$ are $\delta$-close to $\vec{n}_{\text{ter}}(\mathfrak{s}(\phi_0))$ and $\vec{n}_{\text{ini}}(\mathfrak{s}(\phi_0))$ respectively.
\item The length and phase of $\mathfrak{m}$ are $\delta$-close to $t=2R-l(\mathfrak{s})+2I(\frac{\pi}{2})\in[R,1.1R]$ and $\theta=-\phi(\mathfrak{s})$ respectively. The height of $\mathfrak{m}$ is at most $\max{\{h,\frac{1}{2}\ln{t}+C\ln{\frac{1}{\delta}+C}\}}=h$, by the choice of $\phi_0$ and (\ref{R1}).
\end{enumerate}
Note that the height of end points of $\mathfrak{m}$ are at most $(\alpha+100)\ln{R}<\kappa R$ (by (\ref{R2})), so Theorem \ref{connection_principle_with_height} is applicable. See Figure 1 of \cite{LM} for a picture of the oriented $\partial$-framed segment $\mathfrak{m}$.

Then the desired good pants $\Pi$ can be constructed with one cuff $[\mathfrak{s}\mathfrak{s}']\in {\bf \Gamma}_{R,\epsilon}^{h}$ and the other two cuffs $\overline{[\mathfrak{s}\mathfrak{m}]}, \overline{[\bar{\mathfrak{m}}\mathfrak{s}']}\in{\bf \Gamma}_{R,10\delta}^{h+1}$ (by Lemma \ref{lengthphase} and Lemma \ref{distance}). So $\Pi\in {\bf \Pi}_{R,\epsilon}^{h+1}$ holds.

\end{proof}

\begin{definition}\label{swappair}
An {\it $(L,\delta)$-swap pair of bigons} is a pair of $(L,\delta)$-tame bigons $[\mathfrak{a}\mathfrak{b}]$ and $[\mathfrak{a}'\mathfrak{b}']$ such that $[\mathfrak{a}\mathfrak{b}']$ and $[\mathfrak{a}'\mathfrak{b}]$ are also $(L,\delta)$-tame bigons, $\mathfrak{a}$ and $\mathfrak{a}'$ have length and phase $\delta$-close to each other, and so do $\mathfrak{b}$ and $\mathfrak{b}'$.
\end{definition}

The following Lemma corresponds to Construction 4.18 of \cite{LM}, while the proof requires one more step (step III) than the proof in \cite{LM}.

\begin{lemma}\label{swap}(Swapping)

Suppose that $h\in [\alpha\ln{R},(\alpha+100)\ln{R}]$. Suppose that $[\mathfrak{a}\mathfrak{b}]$ and $[\mathfrak{a}'\mathfrak{b}']$ form a $(10L,10^5\delta)$-tame swap pair of bigons, and suppose that $[\mathfrak{a}\mathfrak{b}], [\mathfrak{a}'\mathfrak{b}'], \overline{[\mathfrak{a}\mathfrak{b}']}, \overline{[\mathfrak{a}'\mathfrak{b}]}\in {\bf \Gamma}_{R,\epsilon}^{h}$. Then there is an $(R,\epsilon)$-panted subsurface $F$ of height at most $h+\ln{R}$, with exactly four oriented boundary components $[\mathfrak{a}\mathfrak{b}], [\mathfrak{a}'\mathfrak{b}'], \overline{[\mathfrak{a}\mathfrak{b}']}, \overline{[\mathfrak{a}'\mathfrak{b}]}$. 
\end{lemma}

\begin{proof}

{\bf Step I.} Suppose that the lengths and phases of $\mathfrak{a}, \mathfrak{b}, \mathfrak{a}', \mathfrak{b}'$ are all $(10^7\delta)$-close to each other, and $[\mathfrak{a}\mathfrak{b}],[\mathfrak{a}'\mathfrak{b}']$ form a $(10L,10^7\delta)$-swap pair of bigons. We also suppose that $h\in[\alpha\ln{R},(\alpha+102)\ln{R}]$. Then there exists an $(R,\epsilon)$-panted subsurface of height at most $h+1$ that has the desired oriented boundary.

By the argument in Lemma \ref{split}, we take an angle $\phi_0\in \mathbb{R}/2\pi \mathbb{Z}$ to do frame rotation to all of $\mathfrak{a}, \mathfrak{b}, \mathfrak{a}', \mathfrak{b}'$. We abuse notation and still denote these oriented $\partial$-framed segments by $\mathfrak{a}, \mathfrak{b}, \mathfrak{a}', \mathfrak{b}'$, so that neither $\vec{t}_{\text{ini}}(\mathfrak{a})\times \vec{n}_{\text{ini}}(\mathfrak{a})$ nor $\vec{t}_{\text{ter}}(\mathfrak{a})\times \vec{n}_{\text{ter}}(\mathfrak{a})$ point up the cusp. Then we construct an oriented $\partial$-framed segment $\mathfrak{m}$ from $p_{\text{ter}}(\mathfrak{a})$ to $p_{\text{ini}}(\mathfrak{a})$ with the same property as in the proof of Construction \ref{split}, with $\mathfrak{s}(\phi_0)$ and $\mathfrak{s}'(\phi_0)$ replace by $\mathfrak{a}$ and $\mathfrak{b}$ respectively. See Figure 2 of \cite{LM} for a picture of the oriented $\partial$-framed segment $\mathfrak{m}$.

Then for any $\mathfrak{a}^{\circ}\in \{\mathfrak{a},\mathfrak{a}'\}$ and $\mathfrak{b}^{\circ}\in \{\mathfrak{b},\mathfrak{b}'\}$, by the proof of Lemma \ref{split}, there is a unique pair of pants $\Pi_{\mathfrak{a}^{\circ},\mathfrak{b}^{\circ}}\in {\bf \Pi}_{R,\epsilon}^{h+1}$ with cuffs $[\mathfrak{a}^{\circ}\mathfrak{b}^{\circ}]\in {\bf \Gamma}_{R,\epsilon}^{h}$ and $\overline{[\mathfrak{a}^{\circ}\mathfrak{m}]}, \overline{[\mathfrak{b}^{\circ}\bar{\mathfrak{m}}]}\in {\bf \Gamma}_{R,\epsilon}^{h+1}$. Here we use (\ref{epsilondelta1}) and Lemma \ref{lengthphase} to deduce that these curves lie in ${\bf \Gamma}_{R,\epsilon}$, and use Lemma \ref{distance} to control their heights. Then we paste oriented pairs of pants $\Pi_{\mathfrak{a},\mathfrak{b}}, \Pi_{\mathfrak{a}',\mathfrak{b}'},\overline{\Pi_{\mathfrak{a}',\mathfrak{b}}}, \overline{\Pi_{\mathfrak{a},\mathfrak{b}'}}$ along all cuffs involving $\mathfrak{m}$ to get the desired $(R,\epsilon)$-panted subsurface $F$ with height at most $h+1$.

\bigskip

{\bf Step II.} Suppose that the lengths of $\mathfrak{a}, \mathfrak{b}, \mathfrak{a}', \mathfrak{b}'$ are all at least $0.2R$, and $[\mathfrak{a}\mathfrak{b}],[\mathfrak{a}'\mathfrak{b}']$ form a $(10L,10^6\delta)$-swap pair of bigons. We also suppose that $h\in[\alpha\ln{R},(\alpha+101)\ln{R}]$. Then there exists an $(R,\epsilon)$-panted subsurface of height at most $h+\frac{1}{2}\ln{R}+2$ that has the desired oriented boundary.

Since $[\mathfrak{a}\mathfrak{b}], [\mathfrak{a}'\mathfrak{b}']$ are both $(R,\epsilon)$-good curves, each oriented $\partial$-framed segment in $\mathfrak{a},\mathfrak{b},\mathfrak{a}',\mathfrak{b}'$ has length at most $1.9R$. We divide $\mathfrak{a}$ as the concatenation $\mathfrak{a}_-\mathfrak{a}_+$ of two oriented $\partial$-framed subsegments of same length and phase (close to $0$), and we divide $\mathfrak{b},\mathfrak{a}',\mathfrak{b}'$ to $\mathfrak{b}_-\mathfrak{b}_+$, $\mathfrak{a}_-'\mathfrak{a}_+'$, $\mathfrak{b}_-'\mathfrak{b}_+'$ similarly. Then the length and phase of $\mathfrak{a}_{\pm}$ are $(10^6\delta)$-close to that of $\mathfrak{a}_{\pm}'$, and the length and phase of $\mathfrak{b}_{\pm}$ are $(10^6\delta)$-close to that of $\mathfrak{b}_{\pm}'$.

Take an auxiliary point $*$ in the thick part of $M$ and take a frame $(\vec{t}_*,\vec{n}_*)$ based at $*$. Then we use Theorem \ref{connection_principle_with_height} to construct four oriented $\partial$-framed segments $\mathfrak{s}_a,\mathfrak{s}_a',\mathfrak{s}_b,\mathfrak{s}_b'$ as the following. 
\begin{enumerate}
\item[(a)] The initial and terminal points of $\mathfrak{s}_a$ are $p_{\text{ini}}(\mathfrak{a}_+)$ and $*$ respectively, the length and phase of $\mathfrak{s}_a$ are $\delta$-close to $R+I(\frac{\pi}{2})-l(\mathfrak{a}_+)$ and $-\varphi(\mathfrak{a}_+)$ respectively, and the height of $\mathfrak{s}_a$ is at most $h+\frac{1}{2}\ln{R}$. 
\item[(b)] The initial and terminal directions of $\mathfrak{s}_a$ are $\delta$-close to $\vec{t}_{\text{ini}}(\mathfrak{a}_+)\times \vec{n}_{\text{ini}}(\mathfrak{a}_+)$ and $\vec{t}_*$ respectively, the initial and terminal frames of $\mathfrak{s}_a$ are $\delta$-close to $\vec{n}_{\text{ini}}(\mathfrak{a}_+)$ and $\vec{n}_*$ respectively. The construction of $\mathfrak{s}_{a}'$ is similar to the construction of $\mathfrak{s}_a$.
\item[(c)] The initial and terminal points of $\mathfrak{s}_b$ are $*$ and $p_{\text{ini}}(\mathfrak{b}_+)$ respectively, and the length and phase of $\mathfrak{s}_b$ are $\delta$-close to $R+I(\frac{\pi}{2})-l(\mathfrak{b}_+)$ and $-\varphi(\mathfrak{b}_+)$ respectively,  and the height of $\mathfrak{s}_b$ is at most $h+\frac{1}{2}\ln{R}$. 
\item[(d)] The initial and terminal directions of $\mathfrak{s}_b$ are $\delta$-close to $\vec{t}_*$ and $-\vec{t}_{\text{ini}}(\mathfrak{b}_+)\times \vec{n}_{\text{ini}}(\mathfrak{b}_+)$ respectively, the initial and terminal frames of $\mathfrak{s}_b$ are $\delta$-close to $\vec{n}_*$ and $\vec{n}_{\text{ini}}(\mathfrak{b}_+)$ respectively. The construction of $\mathfrak{s}_{b}'$ is similar to the construction of $\mathfrak{s}_b$.
\end{enumerate}

Since the length of $\mathfrak{a}$ is at most $1.9R$, we have $R+I(\frac{\pi}{2})-l(\mathfrak{a}_+)\geq 0.05R$. So when we apply Theorem \ref{connection_principle_with_height} to construct $\mathfrak{s}_a$, its two end points have height at most $h<(\alpha+101)\ln{R}<\kappa (0.05R)$ (by (\ref{R2})). Moreover, the height of $\mathfrak{s}_a$ is at most $$\max{\{h+\ln{\frac{1}{\delta}}+C,\frac{1}{2}\ln{R}+C\ln{\frac{1}{\delta}}+C\} }<h+\frac{1}{2}\ln{R}$$ (by (\ref{R1})). The same holds for $\mathfrak{s}_a'$, $\mathfrak{s}_b$ and $\mathfrak{s}_b'$. See Figure 3 of \cite{LM} for a picture of oriented $\partial$-framed segments $\mathfrak{s}_a,\mathfrak{s}_a',\mathfrak{s}_b,\mathfrak{s}_b$. In our case, the oriented $\partial$-framed segment $\mathfrak{r}$ in \cite{LM} is reduced to a point, while the points $p$ and $q$ in \cite{LM} are taken to be the same point $*\in M$.

By Lemmas \ref{lengthphase} and \ref{distance}, the following eight good curves lie in ${\bf \Gamma}_{R,\epsilon}^{h+\frac{1}{2}\ln{R}+1}$:
$$[(\mathfrak{a}_-\mathfrak{s}_a)(\mathfrak{s}_b\mathfrak{b}_+)], [(\mathfrak{a}'_-\mathfrak{s}'_a)(\mathfrak{s}_b\mathfrak{b}_+)],[(\mathfrak{a}_-\mathfrak{s}_a)(\mathfrak{s}'_b\mathfrak{b}'_+)],[(\mathfrak{a}'_-\mathfrak{s}'_a)(\mathfrak{s}'_b\mathfrak{b}'_+)],$$
$$ [(\bar{\mathfrak{s}}_a\mathfrak{a}_+)(\mathfrak{b}_-\bar{\mathfrak{s}}_b)],  [(\bar{\mathfrak{s}}'_a\mathfrak{a}'_+)(\mathfrak{b}_-\bar{\mathfrak{s}}_b)],[(\bar{\mathfrak{s}}_a\mathfrak{a}_+)(\mathfrak{b}'_-\bar{\mathfrak{s}}'_b)],[(\bar{\mathfrak{s}}'_a\mathfrak{a}'_+)(\mathfrak{b}'_-\bar{\mathfrak{s}}'_b)].$$ 
These good curves and the original ones $[(\mathfrak{a}_-\mathfrak{a}_+)(\mathfrak{b}_-\mathfrak{b}_+)],  [(\mathfrak{a}'_-\mathfrak{a}'_+)(\mathfrak{b}_-\mathfrak{b}_+)],$ $ [(\mathfrak{a}_-\mathfrak{a}_+)(\mathfrak{b}'_-\mathfrak{b}'_+)], [(\mathfrak{a}'_-\mathfrak{a}'_+)(\mathfrak{b}'_-\mathfrak{b}'_+)]$ 
form the oriented boundary of four good pants $\Pi_1,\Pi_2,\Pi_3,\Pi_4\in {\bf \Pi}_{R,\epsilon}^{h+\frac{1}{2}\ln{R}+1}$. The oriented boundary of these four pair of pants are :
$$[(\mathfrak{a}_+\mathfrak{b}_-)(\mathfrak{b}_+\mathfrak{a}_-)], [\overline{(\mathfrak{b}_+\mathfrak{a}_-)}\ \overline{(\mathfrak{s}_a\mathfrak{s}_b)}], [(\mathfrak{s}_a\mathfrak{s}_b)\overline{(\mathfrak{a}_+\mathfrak{b}_-)}];$$
$$\overline{[(\mathfrak{a}'_+\mathfrak{b}_-)(\mathfrak{b}_+\mathfrak{a}'_-)]}, [(\mathfrak{b}_+\mathfrak{a}'_-)(\mathfrak{s}'_a\mathfrak{s}_b)], [\overline{(\mathfrak{s}'_a\mathfrak{s}_b)}(\mathfrak{a}'_+\mathfrak{b}_-)];$$
$$\overline{[(\mathfrak{a}_+\mathfrak{b}'_-)(\mathfrak{b}'_+\mathfrak{a}_-)]}, [(\mathfrak{b}'_+\mathfrak{a}_-)(\mathfrak{s}_a\mathfrak{s}'_b)], [\overline{(\mathfrak{s}_a\mathfrak{s}'_b)}(\mathfrak{a}_+\mathfrak{b}'_-)];$$
$$[(\mathfrak{a}'_+\mathfrak{b}'_-)(\mathfrak{b}'_+\mathfrak{a}'_-)], [\overline{(\mathfrak{b}'_+\mathfrak{a}'_-)}\ \overline{(\mathfrak{s}'_a\mathfrak{s}'_b)}], [(\mathfrak{s}'_a\mathfrak{s}'_b)\overline{(\mathfrak{a}'_+\mathfrak{b}'_-)}],$$ respectively.

We also have two $(10L,10^7\delta)$-tame swap pair of bigons, which are
$$[(\bar{\mathfrak{s}}_a\mathfrak{a}_+)(\mathfrak{b}_-\bar{\mathfrak{s}}_b)], [(\bar{\mathfrak{s}}_a'\mathfrak{a}_+')(\mathfrak{b}_-'\bar{\mathfrak{s}}_b')]\in {\bf \Gamma}_{R,\epsilon}^{h+\frac{1}{2}\ln{R}+1}$$ and
$$[(\mathfrak{a}_-\mathfrak{s}_a)(\mathfrak{s}_b\mathfrak{b}_+)], [(\mathfrak{a}_-'\mathfrak{s}_a')(\mathfrak{s}_b'\mathfrak{b}_+')]\in {\bf \Gamma}_{R,\epsilon}^{h+\frac{1}{2}\ln{R}+1}.$$ 
By our construction, the lengths and phases of the following oriented $\partial$-framed segments
$$\bar{\mathfrak{s}}_a\mathfrak{a}_+, \bar{\mathfrak{s}}_a'\mathfrak{a}_+', \mathfrak{a}_-\mathfrak{s}_a, \mathfrak{a}_-'\mathfrak{s}_a', \mathfrak{b}_-\bar{\mathfrak{s}}_b, \mathfrak{b}_-'\bar{\mathfrak{s}}_b', \mathfrak{s}_b\mathfrak{b}_+, \mathfrak{s}_b'\mathfrak{b}_+'$$
are all $(10^7\delta)$-close to each other.

By step I, there are two $(R,\epsilon)$-panted face $F_1$ and $F_2$ of height at most $h+\frac{1}{2}\ln{R}+2,$ such that the oriented boundary of $F_1$ and $F_2$ are 
$$[(\bar{\mathfrak{s}}_a\mathfrak{a}_+)(\mathfrak{b}_-\bar{\mathfrak{s}}_b)], [(\bar{\mathfrak{s}}_a'\mathfrak{a}_+')(\mathfrak{b}_-'\bar{\mathfrak{s}}_b')], \overline{[(\bar{\mathfrak{s}}'_a\mathfrak{a}'_+)(\mathfrak{b}_-\bar{\mathfrak{s}}_b)]}, \overline{[(\bar{\mathfrak{s}}_a\mathfrak{a}_+)(\mathfrak{b}'_-\bar{\mathfrak{s}}'_b)]}$$ 
and 
$$[(\mathfrak{a}_-\mathfrak{s}_a)(\mathfrak{s}_b\mathfrak{b}_+)], [(\mathfrak{a}_-'\mathfrak{s}_a')(\mathfrak{s}_b'\mathfrak{b}_+')], \overline{[(\mathfrak{a}'_-\mathfrak{s}'_a)(\mathfrak{s}_b\mathfrak{b}_+)]}, \overline{[(\mathfrak{a}_-\mathfrak{s}_a)(\mathfrak{s}'_b\mathfrak{b}'_+)]}$$
respectively.

Then we paste the four pairs of pants $\Pi_1, \Pi_2,\Pi_3,\Pi_4$ and two $(R,\epsilon)$-panted subsurfaces $F_1,F_2$ along their common boundary to get the desired $(R,\epsilon)$-panted subsurface $F$, with height at most $h+\frac{1}{2}\ln{R}+2.$ (The oriented boundary components of $F_1$ cancel with the third column of boundary components of $\Pi_i$, and the oriented boundary components of $F_2$ cancel with the second column of boundary components of $\Pi_i$).

\bigskip

{\bf Step III.} General case. If the lengths of $\mathfrak{a},\mathfrak{b},\mathfrak{a}',\mathfrak{b}'$ are all at least $0.2R$, the proof is already given in Step II. Without loss of generality, we can assume that the length of $\mathfrak{b}$ is shorter than $0.2R$, so the lengths of $\mathfrak{a}$ and $\mathfrak{a}'$ are at least $1.7R$. 

Since the height of $\mathfrak{a}$ is at most $h<(100+\alpha)\ln{R}$, by Lemma \ref{cuspheightvslength}, each cusp excursion of $\mathfrak{a}$ has length at most $(200+2\alpha) \ln{R}+2\ln{2}<0.1R,$ by (\ref{R2}). So we can divide $\mathfrak{a}$ to a concatenation of two oriented $\partial$-framed segments $\mathfrak{x}\mathfrak{a}_1$, such that $p_{\text{ter}}(\mathfrak{x})$ lies in the thick part of $M$, the length of $\mathfrak{x}$ lies in $ (0.2R,0.4R)$ and the phase of $\mathfrak{x}$ is $0$. Similarly, we can divide $\mathfrak{a}'$ to a concatenation of two oriented $\partial$-framed segments $\mathfrak{a}_1'\mathfrak{y}$, such that $p_{\text{ini}}(\mathfrak{y})$ lies in the thick part of $M$, the length of $\mathfrak{y}$ lies in $ (0.2R,0.4R)$ and the phase of $\mathfrak{y}$ is $0$.

Now we construct an oriented $\partial$-framed segment $\mathfrak{z}$ from $p_{\text{ter}}(\mathfrak{x})$ to $p_{\text{ini}}(\mathfrak{y})$ satisfying the following conditions. 
\begin{enumerate}
\item[(a)] The length and phase of $\mathfrak{z}$ are $\delta$-close to $2R-l(\mathfrak{b})-l(\mathfrak{x})-l(\mathfrak{y})$ and $-\varphi(\mathfrak{b})$ respectively. The height of $\mathfrak{z}$ is at most $\ln{R}$.
\item[(b)] The initial and terminal directions of $\mathfrak{z}$ are $\delta$-close to $\vec{t}_{\text{ter}}(\mathfrak{x})$ and $\vec{t}_{\text{ini}}(\mathfrak{y})$ respectively. The initial and terminal frames of $\mathfrak{z}$ are $\delta$-close to $\vec{n}_{\text{ter}}(\mathfrak{x})$ and $\vec{n}_{\text{ini}}(\mathfrak{y})$ respectively.
\end{enumerate}
Here the desired length of $\mathfrak{z}$ is $2R-l(\mathfrak{b})-l(\mathfrak{x})-l(\mathfrak{y})\in (R,2R)$ and its initial and terminal points lie in the thick part of $M$. So we can apply Theorem \ref{connection_principle_with_height} to construct $\mathfrak{z}$, with height at most $\frac{1}{2}\ln{2R}+C\ln{\frac{1}{\delta}}+C<\ln{R}$  (by (\ref{R1})). A picture of the oriented $\partial$-framed segments $\mathfrak{x},\mathfrak{y},\mathfrak{z}$ is shown in Figure 1.

\begin{figure}[htbp]
\centering
\label{new picture}
\def\svgwidth{.7\textwidth}
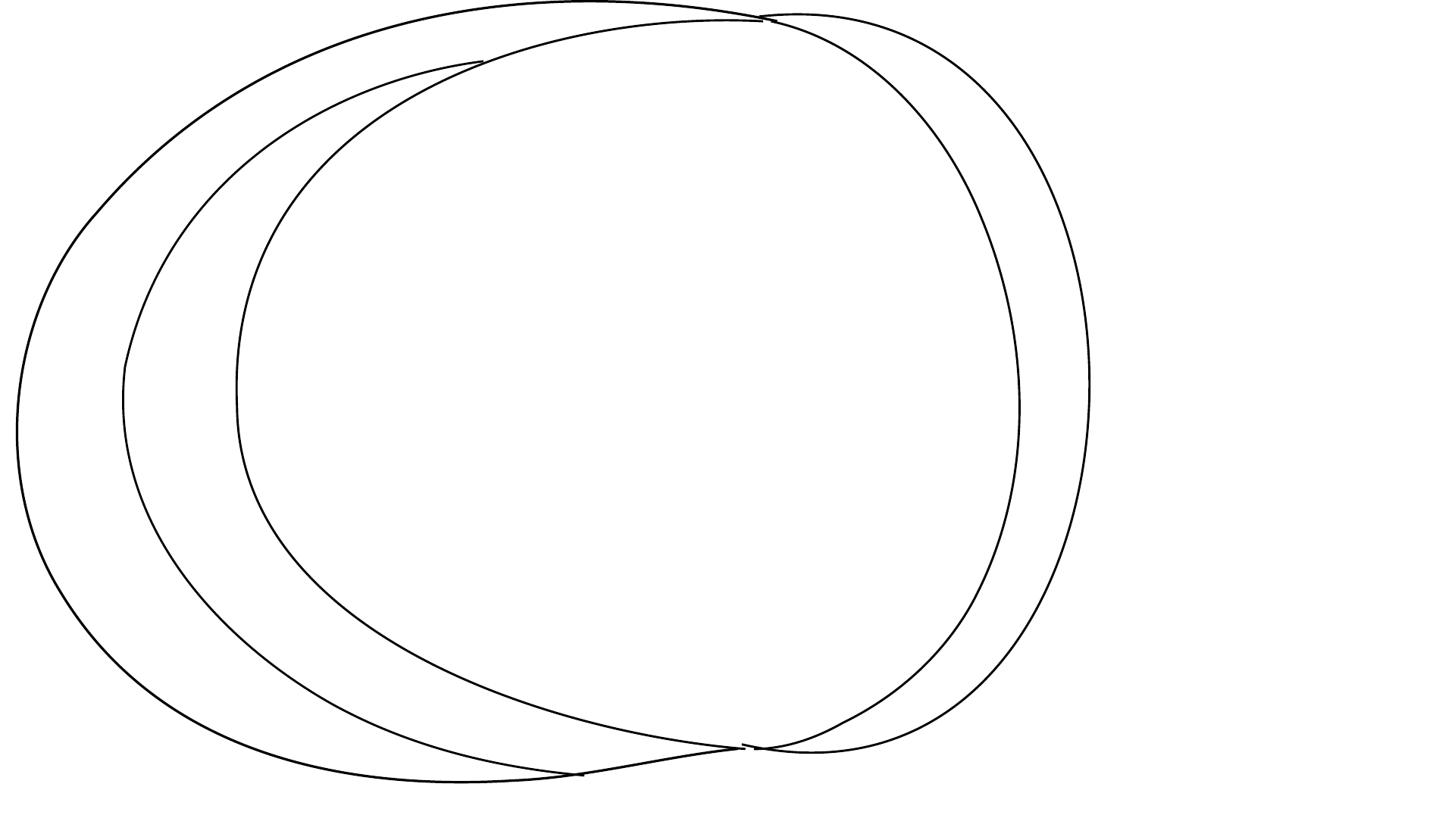
\caption{\ }
\end{figure}

Now we have two $(10L,10^6\delta)$-tame swap pairs of bigons, which are 
$$[(\mathfrak{x}\mathfrak{z})(\mathfrak{y}\mathfrak{b})],[\mathfrak{a}_1'(\mathfrak{y}\mathfrak{b}')]\ \text{and}\ [(\mathfrak{z}\mathfrak{y})(\mathfrak{b}\mathfrak{x})], [\mathfrak{a}_1(\mathfrak{b}'\mathfrak{x})].$$ 
Moreover, by Lemmas \ref{lengthphase} and \ref{distance}, all these curves and the curves obtained by swapping these pairs lie in ${\bf \Gamma}_{R,\epsilon}^{h+1}$.

Since all of $l(\mathfrak{x}),l(\mathfrak{y}),l(\mathfrak{z}),l(\mathfrak{a}_1),l(\mathfrak{a}_1')$ are greater than $0.2R$, all the oriented $\partial$-framed segments in the above swap pairs have length at least $0.2R$. So we can apply the result in Step II to get two $(R,\epsilon)$-panted subsurfaces $F_1,F_2$ of height at most $h+1+(\frac{1}{2}\ln{R}+2)<h+\ln{R}$, such that the oriented boundary of $F_1$ consists of 
$$[(\mathfrak{x}\mathfrak{z})(\mathfrak{y}\mathfrak{b})],[\mathfrak{a}_1'(\mathfrak{y}\mathfrak{b}')],\overline{[(\mathfrak{x}\mathfrak{z})(\mathfrak{y}\mathfrak{b}')]},\overline{[\mathfrak{a}_1'(\mathfrak{y}\mathfrak{b})]},$$ 
and the oriented boundary of $F_2$ consists of 
$$\overline{[(\mathfrak{z}\mathfrak{y})(\mathfrak{b}\mathfrak{x})]}, \overline{[\mathfrak{a}_1(\mathfrak{b}'\mathfrak{x})]},[(\mathfrak{z}\mathfrak{y})(\mathfrak{b}'\mathfrak{x})], [\mathfrak{a}_1(\mathfrak{b}\mathfrak{x})].$$ 
Then we paste $F_1$ and $F_2$ along the boundary components corresponding to $[(\mathfrak{x}\mathfrak{z})(\mathfrak{y}\mathfrak{b})]$ and $[(\mathfrak{z}\mathfrak{y})(\mathfrak{b}'\mathfrak{x})]$, to get an $(R,\epsilon)$-panted subsurface $F$ of height at most $h+\ln{R}$, with oriented boundary 
$$[\mathfrak{a}_1(\mathfrak{b}\mathfrak{x})]=[\mathfrak{a}\mathfrak{b}],[\mathfrak{a}_1'(\mathfrak{y}\mathfrak{b}')]=[\mathfrak{a}'\mathfrak{b}'],\overline{[\mathfrak{a}_1'(\mathfrak{y}\mathfrak{b})]}=\overline{[\mathfrak{a}'\mathfrak{b}]}, \overline{[\mathfrak{a}_1(\mathfrak{b}'\mathfrak{x})]}=\overline{[\mathfrak{a}\mathfrak{b}'].}$$

\end{proof}

In the following two lemmas, we need the definitions of $(L,\delta)$-tame tripod and $(L,\delta)$-rotation pair of tripods.

\begin{definition}\label{tripod}
An {\it $(L,\delta)$-tame tripod}, denoted as $\mathfrak{a}_0\vee \mathfrak{a}_1 \vee \mathfrak{a}_2$, is a triple $(\mathfrak{a}_0,\mathfrak{a}_1,\mathfrak{a}_2)$ of oriented $\partial$-framed segments of length at least $2L$ and of phase $\delta$-close to $0$, such that each $\bar{\mathfrak{a}}_i$ is $\delta$-consecutive to $\mathfrak{a}_{i+1}$ with bending angle $\delta$-close to $\frac{\pi}{3}$. Moreover, an $(L,\delta)$-tame tripod is {\it $(l,\delta)$-nearly regular} if the length of each $\mathfrak{a}_i$ is $\delta$-close to $l$. For each $i\in \mathbb{Z}/3\mathbb{Z}$, we define a new oriented $\partial$-framed segment $\mathfrak{a}_{i,i+1}=\bar{\mathfrak{a}}_i\mathfrak{a}_{i+1}$ (by Definition \ref{chainsandcycles} (3)).

Since the initial framings of $\mathfrak{a}_i$ are $\delta$-close to each other, the initial directions of $\mathfrak{a}_0,\mathfrak{a}_1,\mathfrak{a}_2$ rotate either counterclockwisely or clockwisely around either initial framing. Then we say $\mathfrak{a}_0\vee \mathfrak{a}_1 \vee \mathfrak{a}_2$ is {\it right-handed} if it is the former case and is {\it left-handed} if it is the later case.
\end{definition}

\begin{definition}\label{pairoftripods}
An {\it $(L,\delta)$-rotation pair of tripods} is a pair of $(L,\delta)$-tame tripods $\mathfrak{a}_0 \vee \mathfrak{a}_1 \vee \mathfrak{a}_2$ and $\mathfrak{b}_0 \vee \mathfrak{b}_1 \vee \mathfrak{b}_2$, where $\mathfrak{a}_0 \vee \mathfrak{a}_1 \vee \mathfrak{a}_2$ is $(l_a,\delta)$-nearly regular and $\mathfrak{b}_0 \vee \mathfrak{b}_1 \vee \mathfrak{b}_2$ is $(l_b,\delta)$-regular. Moreover, then chains $\mathfrak{a}_i,\bar{\mathfrak{b}}_j$ are $\delta$-consecutive and $(L,\delta)$-tame for all $i,j\in \mathbb{Z}_3$. In particular, the terminal points of all $\mathfrak{a}_i$ and $\mathfrak{b}_j$ are the same.
\end{definition}

The following two lemmas correspond to Constructions 4.19 and 4.20 of \cite{LM} respectively.

\begin{lemma}\label{rotation}(Rotation)
Suppose that $h\in [\alpha\ln{R},(\alpha+90)\ln{R}]$.  Suppose that $\mathfrak{a}_0\vee\mathfrak{a}_1\vee\mathfrak{a}_2$ and $\mathfrak{b}_0\vee\mathfrak{b}_1\vee\mathfrak{b}_2$ form a $(100L,10^4\delta)$-tame rotation pair of tripods, and
$[\mathfrak{a}_{i,i+1}\bar{\mathfrak{b}}_{j,j\pm 1}]\in {\bf \Gamma}_{R,\epsilon}^{h}$ for any $i,j\in \mathbb{Z}/3\mathbb{Z}$. Then there is an $(R,\epsilon)$-panted subsurface $F$ of height at most $h+3\ln{R}$ such that the following hold. 
\begin{enumerate}
\item If $\mathfrak{a}_0\vee\mathfrak{a}_1\vee\mathfrak{a}_2$ and $\mathfrak{b}_0\vee\mathfrak{b}_1\vee\mathfrak{b}_2$ are of opposite chiralities, then $F$ is a pair of pants $\Pi\in {\bf \Pi}_{R,\epsilon}^{h}$ with cuffs  $[\mathfrak{a}_{i,i+1}\bar{\mathfrak{b}}_{i,i+1}]$ with $i\in \mathbb{Z}/3\mathbb{Z}$.
\item If $\mathfrak{a}_0\vee\mathfrak{a}_1\vee\mathfrak{a}_2$ and $\mathfrak{b}_0\vee\mathfrak{b}_1\vee\mathfrak{b}_2$ are of identical chirality, then $F$ has exactly six oriented boundary components, which are two copies of each $[\mathfrak{a}_{i,i+1}\bar{\mathfrak{b}}_{i,i+1}]$ with $i\in \mathbb{Z}/3\mathbb{Z}$.
\end{enumerate}
\end{lemma}

\begin{proof}
For statement (1), since $\mathfrak{a}_0\vee\mathfrak{a}_1\vee\mathfrak{a}_2$ and $\mathfrak{b}_0\vee\mathfrak{b}_1\vee\mathfrak{b}_2$ have opposite chiralty, the desired subsurface $F$ is a pair of good pants whose spine is a figure-$\theta$ graph with segments $\mathfrak{a}_i\bar{\mathfrak{b}}_i$.

For statement (2), without loss of generality, we can assume that  $\mathfrak{a}_0\vee\mathfrak{a}_1\vee\mathfrak{a}_2$ and $\mathfrak{b}_0\vee\mathfrak{b}_1\vee\mathfrak{b}_2$ are both right-handed, by flipping framings. Moreover, we can also assume that the length of each $\mathfrak{b}_i$ is at least $0.4R$, by switching $\mathfrak{a}$ and $\mathfrak{b}$ if necessary.

{\bf Step I.} Suppose that $h\in[\alpha \ln{R},(\alpha+95)\ln {R}]$. Suppose that  $\mathfrak{a}_0\vee\mathfrak{a}_1\vee\mathfrak{a}_2$ and $\mathfrak{b}_0\vee\mathfrak{b}_1\vee\mathfrak{b}_2$ form a $(100L,3\times 10^4\delta)$-tame rotation pair of tripods, and also suppose that $\mathfrak{b}_0\vee\mathfrak{b}_1\vee\mathfrak{b}_2$ can be written as concatenations $\mathfrak{c}_0\mathfrak{r}\vee\mathfrak{c}_1\mathfrak{r}\vee\mathfrak{c}_2\mathfrak{r}$ such that the following holds. 
\begin{enumerate}
\item[(a)] The length of $\mathfrak{r}$ is at least $0.1R$ and the phase of $\mathfrak{r}$ is $\delta$-close to $0$.
\item[(b)] For each $i$, $\mathfrak{c}_i,\mathfrak{r}$ is a $\delta$-consecutive $(10L,\delta)$-tame chain for each $i$, and $\mathfrak{c}_0\vee\mathfrak{c}_1\vee\mathfrak{c}_2$ is $(l_{b}-l(\mathfrak{r}),\delta)$-nearly regular. 
\end{enumerate}
Then there exists an $(R,\epsilon)$-panted subsurface $F$ of height at most $h+\frac{3}{2}\ln{R}+3$ with the desired oriented boundary.

By the moreover part of Lemma \ref{distance}, the heights of $\mathfrak{a}_0\vee\mathfrak{a}_1\vee\mathfrak{a}_2$, $\mathfrak{b}_0\vee\mathfrak{b}_1\vee\mathfrak{b}_2$  and $\mathfrak{c}_0\vee\mathfrak{c}_1\vee\mathfrak{c}_2$ are at most $h+1$.

By using Theorem \ref{connection_principle_with_height}, we construct an auxiliary oriented $\partial$-famed segment $\mathfrak{s}$ from $p_{\text{ini}}(\mathfrak{r})$ to $p_{\text{ter}}(\mathfrak{r})$ satisfying the following.
\begin{enumerate}
\item[(a)] The length and phase of $\mathfrak{s}$ are $\delta$-close to $l(\mathfrak{r})$ and $0$ respectively. The height of $\mathfrak{s}$ is at most $h+1+\frac{1}{2}\ln{R}$.
\item[(b)] The initial and terminal directions of $\mathfrak{s}$ are $\delta$-close to $\vec{t}_{\text{ini}}(\mathfrak{r})$ and  $\vec{t}_{\text{ter}}(\mathfrak{r})$ respectively. The initial and terminal framings of $\mathfrak{s}$ are $\delta$-close to $-\vec{n}_{\text{ini}}(\mathfrak{r})$ and  $\vec{n}_{\text{ter}}(\mathfrak{r})$ respectively. 
\end{enumerate}
Here the end points of $\mathfrak{s}$ have height at most $h+1<\kappa(0.1R)$ (by (\ref{R2})), and the height of $\mathfrak{s}$ is bounded by $$\max{\{(h+1)+\ln{\frac{1}{\delta}}+C,\frac{1}{2}\ln{2R}+C\ln{\frac{1}{\delta}}+C\}}<h+1+\frac{1}{2}\ln{R}$$ (by (\ref{R1})). See Figure 4 of \cite{LM} for a picture of the oriented $\partial$-framed segment $\mathfrak{s}$.

Then $\mathfrak{c}_0^*\mathfrak{s}\vee \mathfrak{c}_1^*\mathfrak{s}\vee \mathfrak{c}_2^*\mathfrak{s}$ is an $(l(\mathfrak{b}),10\delta)$-nearly regular left-handed tripod, with height at most $h+\frac{1}{2}\ln{R}+2$. 

For each $i\in \mathbb{Z}/3\mathbb{Z}$, by swapping the $(10L,10^5\delta)$-tame swap pair (Lemma \ref{swap}) 
$$[(\mathfrak{a}_{i,i+1})(\bar{\mathfrak{r}}\ \bar{\mathfrak{c}}_{i,i+1}\mathfrak{r})], [(\bar{\mathfrak{a}}_{i,i+1})(\bar{\mathfrak{s}}\mathfrak{c}_{i,i+1}^*\mathfrak{s})]\in {\bf \Gamma}_{R,\epsilon}^{h+\frac{1}{2}\ln{R}+3},$$ 
we get an $(R,\epsilon)$-panted subsurface $E_i$ of height at most $h+\frac{3}{2}\ln{R}+3$, whose oriented boundary consists of 
$$[(\mathfrak{a}_{i,i+1})(\bar{\mathfrak{r}}\ \bar{\mathfrak{c}}_{i,i+1}\mathfrak{r})], [(\bar{\mathfrak{a}}_{i,i+1})(\bar{\mathfrak{s}}\mathfrak{c}_{i,i+1}^*\mathfrak{s})], \overline{[(\mathfrak{a}_{i,i+1})(\bar{\mathfrak{s}}\mathfrak{c}_{i,i+1}^*\mathfrak{s})]}, \overline{[(\bar{\mathfrak{a}}_{i,i+1})(\bar{\mathfrak{r}}\ \bar{\mathfrak{c}}_{i,i+1}\mathfrak{r})]}\in {\bf \Gamma}_{R,\epsilon}^{h+\frac{1}{2}\ln{R}+3}.$$ 

Similarly, for each $i\in \mathbb{Z}/3\mathbb{Z}$, by swapping the $(10L,10^5\delta)$-tame swap pair (Lemma \ref{swap}) 
$$[(\mathfrak{r}\mathfrak{a}_{i,i+1}\bar{\mathfrak{r}}) (\bar{\mathfrak{c}}_{i,i+1})], [(\mathfrak{s}^*\bar{\mathfrak{a}}_{i,i+1}^*\bar{\mathfrak{s}}^*)(\mathfrak{c}_{i,i+1})]\in {\bf \Gamma}_{R,\epsilon}^{h+\frac{1}{2}\ln{R}+3},$$ 
we get another $(R,\epsilon)$-panted subsurface $E_i'$ of height at most $h+\frac{3}{2}\ln{R}+3$, whose oriented boundary consists of 
$$[(\mathfrak{r}\mathfrak{a}_{i,i+1}\bar{\mathfrak{r}}) (\bar{\mathfrak{c}}_{i,i+1})], [(\mathfrak{s}^*\bar{\mathfrak{a}}_{i,i+1}^*\bar{\mathfrak{s}}^*)(\mathfrak{c}_{i,i+1})], \overline{[(\mathfrak{r}\mathfrak{a}_{i,i+1}\bar{\mathfrak{r}})(\mathfrak{c}_{i,i+1}) ]}, \overline{[(\mathfrak{s}^*\bar{\mathfrak{a}}_{i,i+1}^*\bar{\mathfrak{s}}^*)(\bar{\mathfrak{c}}_{i,i+1})]}\in {\bf \Gamma}_{R,\epsilon}^{h+\frac{1}{2}\ln{R}+3}.$$

We apply statment (1) to the $(10L,10^5\delta)$-tame rotation pair of tripods $\mathfrak{a}_0\vee\mathfrak{a}_1\vee\mathfrak{a}_2$ and  $\mathfrak{c}_0^*\mathfrak{s}\vee\mathfrak{c}_1^*\mathfrak{s}\vee\mathfrak{c}_2^*\mathfrak{s}$ of opposite chiralty, then flip the orientation of the obtained subsurface. We obtain a pair of pants $\Pi\in {\bf \Pi}_{R,\epsilon}^{h+\frac{1}{2}\ln{R}+3}$ with oriented boundary 
$$[\mathfrak{a}_{i,i+1}\overline{(\bar{\mathfrak{s}}\mathfrak{c}_{i,i+1}^*\mathfrak{s})}]=\overline{[\bar{\mathfrak{a}}_{i,i+1}\bar{\mathfrak{s}}\mathfrak{c}_{i,i+1}^*\mathfrak{s}]}\in {\bf \Pi}_{R,\epsilon}^{h+\frac{1}{2}\ln{R}+3}$$
for $i\in \mathbb{Z}/3\mathbb{Z}$. We take another copy of $\Pi$ and denote it by $\Pi'$, then its oriented boundary can be rewritten as 
$$\overline{[\bar{\mathfrak{a}}_{i,i+1}^*\bar{\mathfrak{s}}^*\mathfrak{c}_{i,i+1}\mathfrak{s}^*]}\in {\bf \Pi}_{R,\epsilon}^{h+\frac{1}{2}\ln{R}+3}$$  for $i\in \mathbb{Z}/3\mathbb{Z}$.

Then the oriented boundary of $\Pi$ cancels with the second boundary components of $E_i$ for each $i\in \mathbb{Z}/3\mathbb{Z}$, and the oriented boundary of $\Pi'$ cancels with the second components of $E_i'$ for each $i\in \mathbb{Z}/3\mathbb{Z}$. The third and fourth boundary components of $E_i$ cancel with the fourth and third components of $E'_i$, respectively. So we can paste $\Pi,\Pi'$ and $E_i,E_i'$ for $i\in \mathbb{Z}/3\mathbb{Z}$ to get the desired $(R,\epsilon)$-panted subsurface $F$ whose oriented boundary consists of two copies of 
$$[\mathfrak{a}_{i,i+1}(\bar{\mathfrak{r}}\ \bar{\mathfrak{c}}_{i,i+1}\mathfrak{r})]=[\mathfrak{a}_{i,i+1}\bar{\mathfrak{b}}_{i,i+1}]$$ 
for each $i\in \mathbb{Z}/3\mathbb{Z}$, and the height of $F$ is at most $h+\frac{3}{2}\ln{R}+3$.
 
\bigskip

{\bf Step II.} General case. We first construct an auxiliary right-handed tripod $$\mathfrak{b}_0'\vee \mathfrak{b}_1'\vee \mathfrak{b}_2'=\mathfrak{c}_0\mathfrak{r}\vee \mathfrak{c}_1\mathfrak{r}\vee \mathfrak{c}_2\mathfrak{r}$$ satisfying the following conditions.
\begin{enumerate}
\item[(a)] The length and phase of $\mathfrak{r}$ are $\delta$-close to $0.2R$ and $0$ respectively. The height of $\mathfrak{r}$ is at most $h+\frac{1}{2}\ln{R}$. The initial point of $\mathfrak{r}$ lies in the thick part of $M$ and the terminal point of $\mathfrak{r}$ is the common terminal point of $\mathfrak{a}_i$.
\item[(b)] $\mathfrak{c}_i,\mathfrak{r}$ is a $\delta$-consecutive $(10L,\delta)$-tame chain for each $i\in \mathbb{Z}/3\mathbb{Z}$, and $\mathfrak{c}_0\vee \mathfrak{c}_1\vee \mathfrak{c}_2$ is $(l_b-l(\mathfrak{r}),\delta)$-nearly regular. The initial and terminal points of $\mathfrak{c}_i$ both lie in the thick part of $M$ and the height of $\mathfrak{c}_i$ is at most $h$.
\item[(c)] $\mathfrak{a}_0\vee \mathfrak{a}_1\vee \mathfrak{a}_2$ and $\mathfrak{b}_0'\vee \mathfrak{b}_1'\vee \mathfrak{b}_2'$ form a $(100L,3\delta)$-tame rotation pair of tripods of same chiralty.
\end{enumerate}
Here when we apply Theorem \ref{connection_principle_with_height} to construct $\mathfrak{r}$ and $\mathfrak{c}_i$, we use the fact that $l_b-0.2R>0.2R$, since $l_b$ is at least $0.4R$. So each $\mathfrak{b}_i'$ has height at most $h+\frac{1}{2}\ln{R}+1$. See Figure 5 of \cite{LM} to see a picture of the tripod $\mathfrak{b}_0'\vee \mathfrak{b}_1'\vee \mathfrak{b}_2'$.

Now we apply Step I to $\mathfrak{a}_0\vee \mathfrak{a}_1\vee \mathfrak{a}_2$ and  $\mathfrak{b}'_0\vee \mathfrak{b}'_1\vee \mathfrak{b}'_2$, to get an $(R,\epsilon)$-panted subsurface $F'$ of height at most 
$(h+\frac{1}{2}\ln{R}+1)+(\frac{3}{2}\ln{R}+3)=h+2\ln{R}+4$, 
such that the oriented boundary of $F'$ consists of two copies of 
$$[\mathfrak{a}_{i,i+1}\bar{\mathfrak{b}}'_{i,i+1}]\in {\bf \Gamma}_{R,\epsilon}^{h+\frac{1}{2}\ln{R}+2}$$ 
for each $i\in \mathbb{Z}/3\mathbb{Z}$. Then we apply statement (1) to the two $(100L,10^4\delta)$-tame rotation pairs of tripods of opposite chirality 
$$\mathfrak{a}_0\vee\mathfrak{a}_{-1}\vee\mathfrak{a}_{-2},\ \mathfrak{b}_0\vee\mathfrak{b}_1\vee\mathfrak{b}_2, \ \text{and}\ \mathfrak{a}_0\vee\mathfrak{a}_{-1}\vee\mathfrak{a}_{-2},\ \mathfrak{b}'_0\vee\mathfrak{b}'_1\vee\mathfrak{b}'_2.$$ 
Since the heights of these tripods are bounded by $h+\frac{1}{2}\ln{R}+1$, we get two pair of pants $\Pi,\Pi'\in {\bf \Pi}_{R,\epsilon}^{h+\frac{1}{2}\ln{R}+2}$ such that their oriented boundaries are 
$$[\mathfrak{a}_{-i,-i-1}\bar{\mathfrak{b}}_{i,i+1}]\ \text{and}\ \overline{[\mathfrak{a}_{-i,-i-1}\bar{\mathfrak{b}}'_{i,i+1}]}\in  {\bf \Gamma}_{R,\epsilon}^{h+\frac{1}{2}\ln{R}+2}$$ for $i\in \mathbb{Z}/3\mathbb{Z}$ respectively.

Moreover, for each $i\in \mathbb{Z}/3\mathbb{Z}$, we apply Lemma \ref{swap} (swapping) to the $(100L,10^5\delta)$-tame swap pair
$$[\mathfrak{a}_{i,i+1}\bar{\mathfrak{b}}_{i,i+1}], [\mathfrak{a}_{-i,-i-1}\bar{\mathfrak{b}}'_{i,i+1}] \in {\bf \Gamma}_{R,\epsilon}^{h+\frac{1}{2}\ln{R}+2}$$ 
to get an $(R,\epsilon)$-panted subsurface $E_i$ of height at most $(h+\frac{1}{2}\ln{R}+2)+\ln{R}=h+\frac{3}{2}\ln{R}+2$ with oriented boundary 
$$[\mathfrak{a}_{i,i+1}\bar{\mathfrak{b}}_{i,i+1}], [\mathfrak{a}_{-i,-i-1}\bar{\mathfrak{b}}'_{i,i+1}], \overline{[\mathfrak{a}_{i,i+1}\bar{\mathfrak{b}}'_{i,i+1}]}, \overline{[\mathfrak{a}_{-i,-i-1}\bar{\mathfrak{b}}_{i,i+1}]} \in {\bf \Gamma}_{R,\epsilon}^{h+\frac{1}{2}\ln{R}+2}.$$

Now we paste one copy of $F'$, two copies of $\Pi,\Pi'$ and $E_i$ for each $i\in \mathbb{Z}/3\mathbb{Z}$, to get the desired $(R,\epsilon)$-panted subsurface $F$ of height at most $h+2\ln{R}+4<h+3\ln{R}$, whose oriented boundary consists of two copies of $[\mathfrak{a}_{i,i+1}\bar{\mathfrak{b}}_{i,i+1}]$ for each $i\in \mathbb{Z}/3\mathbb{Z}$.
\end{proof}

\begin{lemma}\label{antirotation}(Antirotation)
Suppose that $h\in [\alpha\ln{R},(\alpha+90)\ln{R}]$.  Suppose that $\mathfrak{a}_0\vee\mathfrak{a}_1\vee\mathfrak{a}_2$ and $\mathfrak{b}_0\vee\mathfrak{b}_1\vee\mathfrak{b}_2$ form a $(100L,10^4\delta)$-tame rotation pair of tripods, and $[\mathfrak{a}_{i,i+1}\bar{\mathfrak{b}}_{j,j\pm1}]\in {\bf \Gamma}_{R,\epsilon}^{h}$ for any $i,j\in \mathbb{Z}/3\mathbb{Z}$. Then there is an $(R,\epsilon)$-panted subsurface $F$ of height at most $h+3\ln{R}$ such that the following hold. 
\begin{enumerate}
\item If $\mathfrak{a}_0\vee\mathfrak{a}_1\vee\mathfrak{a}_2$ and $\mathfrak{b}_0\vee\mathfrak{b}_1\vee\mathfrak{b}_2$ are of opposite chiralities, then $F$ has exactly six oriented boundary component, which are two copies of each $[\mathfrak{a}_{i,i+1}\bar{\mathfrak{b}}_{i+1,i}]$ with $i\in \mathbb{Z}/3\mathbb{Z}$.
\item If $\mathfrak{a}_0\vee\mathfrak{a}_1\vee\mathfrak{a}_2$ and $\mathfrak{b}_0\vee\mathfrak{b}_1\vee\mathfrak{b}_2$ are of identical chirality, then $F$ has exactly three oriented boundary components, which consists of  $[\mathfrak{a}_{i,i+1}\bar{\mathfrak{b}}_{i+1,i}]$ with $i\in \mathbb{Z}/3\mathbb{Z}$. 
\end{enumerate}
\end{lemma}

\begin{proof}
We consider the $(100L,10^4\delta)$-tame rotation pair of tripods 
$$\mathfrak{a}_0\vee \mathfrak{a}_1\vee \mathfrak{a}_2, \mathfrak{b}_0\vee \mathfrak{b}_2\vee \mathfrak{b}_1.$$ 
By Lemma \ref{rotation}, there is an $(R,\epsilon)$-panted subsurface $E$ of height at most $h+3\ln{R}$, such that the following hold.
\begin{enumerate}
\item If $\mathfrak{a}_0\vee \mathfrak{a}_1\vee \mathfrak{a}_2$ and $\mathfrak{b}_0\vee \mathfrak{b}_1\vee \mathfrak{b}_2$ have opposite chirality, then $\mathfrak{a}_0\vee \mathfrak{a}_1\vee \mathfrak{a}_2$ and $\mathfrak{b}_0\vee \mathfrak{b}_2\vee \mathfrak{b}_1$ have same chirality. In this case the oriented boundary of $F$ consists of two copies of 
$$[\mathfrak{a}_{01}\bar{\mathfrak{b}}_{02}], [\mathfrak{a}_{12}\bar{\mathfrak{b}}_{21}], [\mathfrak{a}_{20}\bar{\mathfrak{b}}_{10}]\in {\bf \Gamma}_{R,\epsilon}^h.$$
\item If $\mathfrak{a}_0\vee \mathfrak{a}_1\vee \mathfrak{a}_2$ and $\mathfrak{b}_0\vee \mathfrak{b}_1\vee \mathfrak{b}_2$ have same chirality, then $\mathfrak{a}_0\vee \mathfrak{a}_1\vee \mathfrak{a}_2$ and $\mathfrak{b}_0\vee \mathfrak{b}_2\vee \mathfrak{b}_1$ have opposite chirality. In this case, the oriented boundary of $F$ consists of one copy of 
$$[\mathfrak{a}_{01}\bar{\mathfrak{b}}_{02}], [\mathfrak{a}_{12}\bar{\mathfrak{b}}_{21}], [\mathfrak{a}_{20}\bar{\mathfrak{b}}_{10}]\in {\bf \Gamma}_{R,\epsilon}^h.$$
\end{enumerate}

Now we consider the $(100L,10^5\delta)$-tame swap pair of bigons 
$$[\mathfrak{a}_{01}\bar{\mathfrak{b}}_{10}], [\mathfrak{a}_{20}\bar{\mathfrak{b}}_{02}].$$ 
By Lemma \ref{swap} (swapping), there is an $(R,\epsilon)$-panted subsurface $E'$ of height at most $h+\ln{R}$, such that its oriented boundary is 
$$[\mathfrak{a}_{01}\bar{\mathfrak{b}}_{10}], [\mathfrak{a}_{20}\bar{\mathfrak{b}}_{02}], \overline{[\mathfrak{a}_{01}\bar{\mathfrak{b}}_{02}]}, \overline{[\mathfrak{a}_{20}\bar{\mathfrak{b}}_{10}]}\in {\bf \Gamma}_{R,\epsilon}^h.$$
Then we paste $E$ and two (for statement (1)) or one (for statement (2)) copies of $E'$ to get an $(R,\epsilon)$-panted subsurface $F$ of height at most $h+3\ln{R}$ with the desired oriented boundary: two or one copies of $$[\mathfrak{a}_{01}\bar{\mathfrak{b}}_{10}], [\mathfrak{a}_{12}\bar{\mathfrak{b}}_{21}], [\mathfrak{a}_{20}\bar{\mathfrak{b}}_{02}]\in {\bf \Gamma}_{R,\epsilon}^h$$ respectively.
\end{proof}

\bigskip
\bigskip

\section{The panted cobordism group}\label{cobordism}

In this section, we will define the panted cobordism group $\Omega_{R,\epsilon}^{h,h'}(M)$ and prove Theorem \ref{main1}.

We will always assume the following inequalities hold, which implies all the inequalities at the beginning of Section \ref{construction}. These inequalities involve positive numbers $\epsilon,\delta,L,K,R>0$. To obtain these constants, we first choose $\epsilon$, then choose $\delta$, then choose $L$, then choose $K$, finally choose $R$, and this process also invokes Lemma \ref{setuplemma} in the following. Here $\epsilon$ and $\delta$ are very small, while $L,K,R$ are very large.
\begin{align}\label{epsilondelta2}
0<10^9\delta<\epsilon<10^{-2},
\end{align}
\begin{align}\label{L2}
L>\max{\{100,T,C(\ln{\frac{10^3}{\delta}}+2),\frac{1}{\delta}\ln{\frac{10}{\delta}}, (\frac{e}{\delta})^{10C}\}},
\end{align}
\begin{align}\label{L3}
L> \frac{1}{2}\ln{(\kappa^{-1} L)}+C(\ln{\frac{1}{\delta}}+1)+2\ \text{and\ satisfies\ Lemma\ \ref{setuplemma}\ (1)},
\end{align}
\begin{align}\label{K1}
K>30L\ \text{and\ satisfies Lemma \ref{setuplemma} (2)},
\end{align}
\begin{align}\label{R4}
R>\max{\{10^4(1+\kappa^{-1})L,100T,100C(\ln{\frac{10^3}{\delta}+1}),10^{11}(10^3\frac{e}{\delta})^{10C},\frac{e^{100(C+1)}}{\delta^{100}}, e^{4L}\}},
\end{align}
\begin{align}\label{R5}
R\geq \kappa R>10^4(\alpha+10^3)\ln{R}+100+10\ln{\frac{100\alpha}{\delta}},
\end{align} 
\begin{align}\label{R7}
R\ \text{satisfies Lemma \ref{setuplemma} (4) (which involves $K$)}.
\end{align}
Here the constants $T>0, C>2, \kappa\in (0,1]$ are constants in Theorem \ref{connection_principle_with_height}, and $\alpha$ is the constant in Theorem \ref{main1} that is greater or equal to $4$.

\subsection{Definition of the panted cobordism group}\label{pantscobordismsection}

We first give our official definition of the panted cobordism group $\Omega_{R,\epsilon}^{h,h'}(M)$, which is more convenient for our proof.

\begin{definition}\label{pantscobordismdefinition}(cf. Definition 5.1 of \cite{LM})
For $h'>h>0$, we define the {\it $(R,\epsilon)$-panted cobordism group of height $(h,h')$}, denoted by $\Omega_{R,\epsilon}^{h,h'}(M)$, as the following.

For two $(R,\epsilon)$-multicurves $L$ and $L'$ of height at most $h$, we say that they are {\it $(R,\epsilon)$-panted cobordant with height at most $h'$} if there is an $(R,\epsilon)$-panted subsurface $F$ in $M$ of height at most $h'$, such that $\partial F$ is the disjoint union $L\sqcup \bar{L}'$. Then we define $\Omega_{R,\epsilon}^{h,h'}(M)$ to be the set of all equivalence classes of $(R,\epsilon)$-multicurves of height at most $h$, under $(R,\epsilon)$-panted cobordant with height at most $h'$. For an $(R,\epsilon)$-multicurve $L$ of height at most $h$, we use $[L]_{R,\epsilon}^{h,h'}$ to denote the equivalence class of $L$ in $\Omega_{R,\epsilon}^{h,h'}(M)$. 

Moreover, we also allow the possibility of $h'=\infty$. In this case, we do not impose any height bound for the $(R,\epsilon)$-panted subsurface $F$. We denote the resulting $(R,\epsilon)$-panted cobordism group by $\Omega_{R,\epsilon}^{h}(M)$, and denote the equivalence class of $L$ by $[L]_{R,\epsilon}^h$.
\end{definition}

In the following, we prove that $(R,\epsilon)$-cobordant with certain height bound is an equivalence relation, and $\Omega_{R,\epsilon}^{h,h'}(M)$ has an abelian group structure for some $h$ and $h'$. In the remainder of this section, we will always use $h$ and $h'$ to denote $\alpha\ln{R}$ and $\beta\ln{R}$ respectively.

\begin{lemma}
For any $\epsilon\in (0,10^{-2})$ and $\beta>\alpha>1$ with $\beta-\alpha>0.1$, if $R$ satisfies all the inequalites in (\ref{R4}) that does not involve $L$, then
\begin{enumerate}
\item ${\bf \Gamma}_{R,\epsilon}^{\alpha \ln{R}}$ is not empty.
\item Being $(R,\epsilon)$-panted cobordant with height at most $h'=\beta\ln{R}$ is an equivalence relation on the set of $(R,\epsilon)$-multicurves of height at most $h=\alpha \ln{R}$.
\item Under the disjoint union operation, $\Omega_{R,\epsilon}^{\alpha \ln{R},\beta\ln{R}}(M)$ has an abelian group structure.
\end{enumerate}
\end{lemma}

\begin{proof}

We first take a positive constant $\delta<10^{-9}\epsilon$. To prove ${\bf \Gamma}_{R,\epsilon}^{\alpha \ln{R}}$ is not empty, we take a point $p=q$ in the thick part of $M$, and take two orthogonal unit tangent vectors $\vec{t}_p=\vec{t}_q\in T_pM$ and $\vec{n}_p=\vec{n}_q\in T_pM$, to get two frames $(p,\vec{t}_p,\vec{n}_p)=(q,\vec{t}_q,\vec{n}_q) \in \text{SO}(M)$. Then we apply Theorem \ref{connection_principle_with_height} to these two frames, with $t=2R$ and $\theta=0$, to get an oriented $\partial$-framed segement $\mathfrak{s}$. After straightening the carrier of $\mathfrak{s}$ to a closed geodesic, we get an $(R,\epsilon)$-good curve (by Lemma \ref{lengthphase}). By Theorem \ref{connection_principle_with_height} and Lemma \ref{distance}, the height of the resulting closed geodesic is at most $\frac{1}{2}\ln{2R}+C\ln{\frac{1}{\delta}}+C+1$, which is smaller than $h=\alpha \ln{R}$, by (\ref{R4}) and the fact that $\alpha>1$.

To prove that being $(R,\epsilon)$-panted cobordant with height at most $h'$ is an equivalence relation, we only need to prove that it is reflective, since it is obviously symmetric and transitive. The reflectivity is proved by the same way as Lemma 5.4 of \cite{LM}. For any $(R,\epsilon)$-multicurve $L$ of height at most $h$, we divide each closed geodesic $l\subset L$ to two oriented $\partial$-framed segments of same length and phase (close to $0$). We apply Lemma \ref{split} (splitting) to construct a pair of pants $\Pi\in {\bf \Pi}_{R,\epsilon}^{h+1}$ with $l$ as one of its cuff. Then we paste $\Pi$ and $\bar{\Pi}$ along the two boundary components other than $l$ to get an $(R,\epsilon)$-panted subsurface with height at most $h+1$, with oriented boundary $l\sqcup \bar{l}$. Here the height of $\Pi$ is at most $h+1<\alpha\ln{R}+1<\beta\ln{R}$, by (\ref{R4}). So $l$ is $(R,\epsilon)$-panted cobordant to itself, and so does $L$. 

The disjoint union operation is obviously well-defined, associative and commutative. The zero element is $[L\sqcup \bar{L}]_{R,\epsilon}^{h,h'}$ for any $(R,\epsilon)$-multicurve $L$, since 
$$[L\sqcup \bar{L}]_{R,\epsilon}^{h,h'}+[L']_{R,\epsilon}^{h,h'}=[L\sqcup \bar{L}\sqcup L']_{R,\epsilon}^{h,h'}=[L']_{R,\epsilon}^{h,h'}.$$
Here the last equation follows from the fact that $(R,\epsilon)$-panted cobordant with height at most $h'$ is reflective. A similar proof also shows that the zero element is unique. So the inverse of any $[L]_{R,\epsilon}^{h,h'}$ is $[\bar{L}]_{R,\epsilon}^{h,h'}$.

\end{proof}



To prove Theorem \ref{main1}, we will construct two homomorphisms 
$$\Phi:\Omega_{R,\epsilon}^{h,h'}(M)\to  H_1(\text{SO}(M);\mathbb{Z})\ \text{and}\ \Psi^{\text{ab}}: H_1(\text{SO}(M);\mathbb{Z})\to \Omega_{R,\epsilon}^{h,h'}(M).$$ 
The homomorphism $\Phi$ is easy to describe and will be defined in this section, and the more complicated homomorphism $\Psi^{\text{ab}}$ will be defined in Section \ref{inversehomomorphism}. Actually, we will first define a homomorphism $\Psi:\pi_1(\text{SO}(M),\text{e})\to \Omega_{R,\epsilon}^{h,h'}(M)$ by a family of auxiliary data, then the desired homomorphism $\Psi^{\text{ab}}: H_1(\text{SO}(M);\mathbb{Z})\to \Omega_{R,\epsilon}^{h,h'}(M)$ is the abelianization of $\Psi$. In Section \ref{verifications}, we will prove that $\Phi\circ \Psi^{\text{ab}}=id:H_1(\text{SO}(M);\mathbb{Z})\to H_1(\text{SO}(M);\mathbb{Z})$, and $\Psi$ is surjective. These two facts together imply that $\Phi$ is an isomorphism.

The definition of $\Phi:\Omega_{R,\epsilon}^{h,h'}(M)\to  H_1(\text{SO}(M);\mathbb{Z})$ is same with the definition of $\Phi$ in \cite{LM}. For each $(R,\epsilon)$-curve $\gamma$ with height at most $h=\alpha \ln{R}$, we define 
$$\Phi([\gamma]_{R,\epsilon}^{h,h'})=[\hat{\gamma}]\in H_1(\text{SO}(M);\mathbb{Z})$$ 
to be the homology class of the {\it canonical lift} $\hat{\gamma}$ of $\gamma$ (Definition 5.11 of \cite{LM}) in $\text{SO}(M)$. To defined the canonical lift $\hat{\gamma}$, we first take a point $p\in \gamma$ and two orthogonal unit vectors $\vec{t}_p,\vec{n}_p\in T_pM$ such that $\vec{t}_p$ is tangent to the positive direction of $\gamma$. Then we take $\hat{\gamma}$ to be the concatenation of following paths in $\text{SO}(M)$.
\begin{itemize}
\item Parallel transport the frame ${\text e}=(p,\vec{t}_p,\vec{n}_p)$ along $\gamma$ once to get another frame ${\text e}'$ based at $p$.
\item Rotate ${\text e}'$ along $\vec{n}_p$ by $2\pi$ counterclockwisely to return to ${\text e}'$,
\item Connect ${\text e}'$ to ${\text e}$ by an $\epsilon$-short path in $\text{SO}(M)|_p=\text{SO}(3)$.
\end{itemize}

By Lemma 5.14 of \cite{LM}, the function $\Phi$ is a well-defined map and is also a group homomorphism. The proof of Lemma 5.14 of \cite{LM} also uses Lemma 5.12 of \cite{LM}, and neither of them really uses the condition that $M$ is a closed hyperbolic $3$-manifold. So the proof is also valid in the cusped manifold case.

\subsection{Define the homomorphism $\Psi:\pi_1(\text{SO}(M),*)\to \Omega_{R,\epsilon}^{h,h'}(M)$}\label{inversehomomorphism}

The definition of the homomorphism $\Psi:\pi_1(\text{SO}(M),\text{e})\to \Omega_{R,\epsilon}^{h,h'}(M)$ is quite complicated, and we will follow the strategy in \cite{LM} to define it. In Section \ref{setup}, we first choose a set of auxiliary data, and the following notations will be defined there. The auxiliary data includes a finite triangular generating set 
$$\tau_k(\mathscr{B}_K)\subset \pi_1(M,*),$$ 
and another finite set 
$$\mathscr{C}_{DL}\cap \mathscr{B}_{\frac{3}{2}R}^{h-2} \subset \pi_1(M,*) $$ that contains $\tau_k(\mathscr{B}_K)$. Then we use 
$\hat{\mathscr{C}}_{DL}\cap \hat{\mathscr{B}}_{\frac{3}{2}R}^{h-2}$ to denote the preimage of $\mathscr{C}_{DL}\cap \mathscr{B}_{\frac{3}{2}R}^{h-2}\subset \pi_1(M,*)$ in $\pi_1(\text{SO(M)},\text{e})$.

In Section \ref{settheory}, we will define a set-theoretic map 
$$\Psi_1:\hat{\mathscr{C}}_{DL}\cap \hat{\mathscr{B}}_{\frac{3}{2}R}^{h-2} \to \Omega_{R,\epsilon}^{h,h'}(M),$$ 
and prove that it satisfies the triangular relation. In Section \ref{definitionPsi}, we prove that this set-theoretic map $\Psi_1$ induces a homomorphism 
$$\Psi:\pi_1(\text{SO(M)},\text{e})\to  \Omega_{R,\epsilon}^{h,h'}(M).$$

\subsubsection{Setup}\label{setup}

Let $*\in M$ be a point in the thick part of $M$ such that it lies on a closed oriented geodesic $\gamma$. We take a unit vector $\vec{t}\in T_*M$ that is tangent to the positive direction of $\gamma$, then take another unit vector $\vec{n}\in T_*M$ orthogonal to $\vec{v}$. We identify the universal cover of $M$ with $\mathbb{H}^3$. Let $O\in \mathbb{H}^3$ be a point in the preimage of $*$, and let $\vec{t}_O, \vec{n}_O\in T_OM$ be the unit vectors corresponding to $\vec{t}, \vec{n}$ respectively. Now we get base frames $\text{e}=(*,\vec{t}, \vec{n})\in \text{SO}(M)$ and $\tilde{\text{e}}=(O,\vec{t}_O,\vec{n}_O)\in \text{SO}(\mathbb{H}^3)$. 

We use $B_r$ to denote the radius-$r$ open ball in $\mathbb{H}^3$ centered at $O$. Let $O_r$ be the intersection of $\partial B_r$ and the geodesic ray starting from $O$ and tangent with $\vec{t}_O$. Let $U_r$ be the open half space that contains $B_r$, such that $\partial U_r$ is tangent to $\partial B_r$ at $O_r$. Then for any $r,s>0$, we define 
$$\mathscr{B}_r=\{g\in \pi_1(M,*)\ |\ gO\in B_r,\ g\ne e\},$$
$$\mathscr{B}_r^s=\{ g\in \mathscr{B}_r\ |\  \text{the\ height\ of\ the\ geodesic\ representative\ of\ } g\ \text{is\ at\ most\ } s\}$$ and 
$$\mathscr{C}_r=\{g\in \pi_1(M,*)\ |\ gU_r\cap U_r=\emptyset\}.$$
Here we identify $\pi_1(M,*)$ with a subgroup of $\text{Isom}_+(\mathbb{H}^3)$, and this definition of $\mathscr{B}_r$ is equivalent to the definition of $\mathscr{B}_r$ in Proposition \ref{triangulargeneratingset}. Also note that $\mathscr{B}_r$ is a finite set for all $r>0$.

The following lemma set up a collection of auxiliary data that will be used for proving Theorem \ref{main1}. The conditions in this lemma contain conditions (\ref{epsilondelta2}) - (\ref{R7}) at the beginning of this section.

\begin{lemma}\label{setuplemma}
Let $M$ be an oriented cusped hyperbolic $3$-manifold. For any $\epsilon\in (0,10^{-2})$, there is a collection of data depending on $M$ and $\epsilon$ as follows.
\begin{enumerate}
\item There exist constants $\delta, L>0$, such that $\delta$ satisfies (\ref{epsilondelta2}), $L$ is greater than the height of $\gamma$ and satisfies the inequalities in (\ref{L2}), (\ref{L3}).
\item There exists a positive constant $K>30L$, such that $\mathscr{B}_K$ is a triangular generating set of $\pi_1(M,*)$ (which is condition (\ref{K1})).
\item Let $D=300(1+\kappa^{-1})$, then there exists $k\in \pi_1(M,*)$, such that $k\in \mathscr{C}_{DL}$ and $\tau_k(\mathscr{B}_K)\subset \mathscr{C}_{DL}$. Here $\tau_k(g)=k^{-1}gk$.
\item There exists a positive constant $R_0>0$, such that for any $R>R_0$, the following conditions hold.
\begin{enumerate}
\item $R$ satisfies the inequalities in (\ref{R4}) and  (\ref{R5}).
\item $R>20DL+10|k|$ and $\alpha \ln{R}\geq 5+\text{height}(k).$ Here $|k|$ is the length of the geodesic representative of $k\in \pi_1(M,*)$, and $\text{height}(k)$ is the height of $k$.
\item $\tau_k(\mathscr{B}_K)\subset \mathscr{B}_{\frac{3}{2}R-4DL}^{\alpha\ln{R}-3}$.
\end{enumerate}
\end{enumerate}
\end{lemma}

\begin{proof}
For (1), the existence of $\delta$ and $L$ follows directly from the desired condition. For (2), the existence of $K$ follows from Proposition \ref{triangulargeneratingset}. If (3) holds true, (4) follows from the fact that $\tau_k(\mathscr{B}_K)$ is a finite set and we can take $R_0$ to be arbitrarily large.

Now we work on condition (3). The condition that $\tau_k(g)\in \mathscr{C}_{DL}$ is equivalent to $gk(U_{DL})\cap k(U_{DL})=\emptyset$. Take a hyperbolic element $k'\in \pi_1(M,*)$, such that neither of its fixed points on $\partial \mathbb{H}^3$ lie in the closure of $U_{DL}$ in $\mathbb{H}^3\cup \partial \mathbb{H}^3$, and neither of its fixed points on $\mathbb{H}^3$ coincide with the (finitely many) fixed points of elements in $\mathscr{B}_K$.

For each $g\in \mathscr{B}_K$, if $g$ is a hyperbolic element, there exists a positive integer $N_g$, such that for all $n>N_g$, $(k')^n(U_{DL})$ is sufficiently far away from the axis of $g$ so that $g(k')^n(U_{DL})\cap (k')^n(U_{DL})=\emptyset$. If $g\in \mathscr{B}_K$ is a parabolic element, up to conjugation, we can assume that $g$ is representated by $g(x,y,z)=(x+1,y,z)$ in the upper half space model. Then there exists a positive integer $N_g$, such that for all $n>N_g$, the Euclidean diameter of $(k')^n(U_{DL})$ is at most $\frac{1}{2}$, then  $g(k')^n(U_{DL})\cap (k')^n(U_{DL})=\emptyset$. Moreover, since neither fixed points of $k'$ lie in the closure of $U_{DL}$, there exists $N$, such that for all $n>N$, $(k')^n(U_{DL})\cap U_{DL}=\emptyset$. Now we take any $n$ larger than the $N$ and all the $N_g$, then $k=(k')^n$ satisfies condition (3).
\end{proof}

In the following, we will always assume $R$ is greater than the $R_0$ provided in Lemma \ref{setuplemma} (4).

For $\mathscr{B}_r $, $\mathscr{B}_r^s $ and $\mathscr{C}_r\subset \pi_1(M,*)$, we use $\hat{\mathscr{B}}_r$, $\hat{\mathscr{B}}_r^s$ and $\hat{\mathscr{C}}_r$ to denote their preimages in $\pi_1(\text{SO}(M),\text{e})$ respectively. We also use $\widehat{\tau_k(\mathscr{B}_K)}$ to denote the preimage of $\tau_k(\mathscr{B}_K)\subset \pi_1(M,*)$ in $\pi_1(\text{SO}(M),\text{e})$. Note that Lemma \ref{setuplemma} (3) (4) implies that $\widehat{\tau_k(\mathscr{B}_K)}\subset \hat{\mathscr{C}}_{DL}\cap \hat{\mathscr{B}}_{\frac{3}{2}R-4DL}^{\alpha\ln{R}-3}$ holds.

In Definition 5.7 of \cite{LM}, Liu-Markovic defined $\delta$-sharp elements in $\pi_1(\text{SO}(M),\text{e})$, and we slightly modify it as in the following.

\begin{definition}\label{sharpelement}
For any $\delta\in (0,10^{-11})$, an element $\hat{g}\in \pi_1(\text{SO}(M),\text{e})$ is {\it $\delta$-sharp} if its image $g\in \pi_1(M,*)$ is represented by a geodesic segment whose initial and terminal directions are $(100\delta)$-close to $\vec{t}$ and $-\vec{t}$, respectively. 

For a $\delta$-sharp element $\hat{g}$, an oriented $\partial$-framed segment $\mathfrak{g}$ is associated to $\hat{g}$, and vice versa, if the following conditions hold.
\begin{itemize}
\item The carrier segment of $\mathfrak{g}$ is the geodesic representative of $g$. The phase of $\mathfrak{g}$ is $\delta$-close to $0$. The initial and terminal framings of $\mathfrak{g}$ are $(200\delta)$-close to each other.
\item The element $\hat{g}\in \pi_1(\text{SO}(M),\text{e})$ is represented by the following closed path. First parallel transport $\text{e}$ along $g$, then rotate by $\pi$ counterclockwisely along $\vec{n}_{\text{ter}}(\mathfrak{g})$, then returns to $\text{e}$ along a $(10^3\delta)$-short path in $\text{SO}(M)|_*=\text{SO}(3)$.
\end{itemize}
\end{definition}

\begin{remark}
In Definition \ref{sharpelement}, although we can replace $(100\delta)$-closeness by $\delta$-closeness, as in the definition in \cite{LM}, we do prefer the $(100\delta)$-closeness condition, for convenience of the proof of Theorem \ref{main1}.
\end{remark}

By a similar proof as in Lemma 5.8 of \cite{LM}, for each $\delta$-sharp element $\hat{g}\in \pi_1(\text{SO(M)},\text{e})$, there is a unique oriented $\partial$-framed segment $\mathfrak{s}_{\hat{g}}$ associated to $\hat{g}$, up to $(200\delta)$-closeness of framings at end points. By Lemma 5.17 of \cite{LM}, if $L>\frac{1}{D}\ln{\frac{4}{\delta}}$ (which holds by (\ref{L2})), any $\hat{g}\in \hat{\mathscr{C}}_{DL}$ is a $\delta$-sharp element, even if we require $\delta$-closeness in Definition \ref{sharpelement}.

\subsubsection{A set-theoretic map $\Psi_1$ and its properties}\label{settheory}

In this section, we define a set-theoretic map 
$$\Psi_1:\hat{\mathscr{C}}_{DL}\cap \hat{\mathscr{B}}_{\frac{3}{2}R}^{h-2} \to \Omega_{R,\epsilon}^{h,h'}(M)$$ 
and check that it satisfies the triangular relation on a slightly smaller subset $\hat{\mathscr{C}}_{DL}\cap \hat{\mathscr{B}}_{\frac{3}{2}R-4DL}^{h-3}\subset \hat{\mathscr{C}}_{DL}\cap \hat{\mathscr{B}}_{\frac{3}{2}R}^{h-2}$. Note that we always have $h=\alpha \ln{R}$.

For any $\hat{g}\in \hat{\mathscr{C}}_{DL}\cap \hat{\mathscr{B}}_{\frac{3}{2}R}^{h-2}$, it is automatically $\delta$-sharp, and we take an oriented $\partial$-framed segment $\mathfrak{s}_{\hat{g}}$ associated to $\hat{g}$. Let $\vec{n}_{\hat{g}}\in T_*M$ be a unit vector orthogonal to $\vec{t}$ such that it is $(300\delta)$-close to both $\vec{n}_{\text{ini}}(\mathfrak{s}_{\hat{g}})$ and $\vec{n}_{\text{ter}}(\mathfrak{s}_{\hat{g}})$.

To define $\Psi_1(\hat{g})$, we need some auxiliary data as the following.
\begin{condition}\label{defining}
We take a right-handed nearly regular tripod $\mathfrak{a}_0\vee \mathfrak{a}_1\vee\mathfrak{a}_2$ and an oriented $\partial$-framed segment $\mathfrak{b}$ satisfying the following. 
\begin{itemize}
\item The right-handed tripod $\mathfrak{a}_0\vee \mathfrak{a}_1\vee\mathfrak{a}_2$ is $(R-\frac{1}{2}l(\mathfrak{s}_{\hat{g}})+\frac{1}{2}I(\frac{\pi}{3}),\delta)$-nearly regular, and with height at most $\frac{3}{5}\ln{R}$. For each $i$, the terminal point $p_{\text{ter}}(\mathfrak{a}_i)$ is $*$ and the initial point $p_{\text{ini}}(\mathfrak{a}_i)$ is a point in the thick part of $M$. The terminal direction $\vec{t}_{\text{ter}}(\mathfrak{a}_i)$ is $(10^3\delta)$-close to $\vec{t}$, and the terminal framing $\vec{n}_{\text{ter}}(\mathfrak{a}_i)$ is $(10^3\delta)$-close to $\vec{n}_{\hat{g}}$.
\item The oriented $\partial$-framed segment $\mathfrak{b}$ has length $\delta$-close to $2R-l(\mathfrak{s}_{\hat{g}})$ and phase $\delta$-close to $0$, with height at most $\frac{3}{5}\ln{R}$. The initial and terminal directions of $\mathfrak{b}$ are $(10^3\delta)$-close to $-\vec{t}$ and $\vec{t}$ respectively, and the initial and terminal framings of $\mathfrak{b}$ are both $(10^3\delta)$-close to $\vec{n}_{\hat{g}}$.
\end{itemize}
\end{condition}
See Figure 7 of \cite{LM} to see a picture of the tripod $\mathfrak{a}_0\vee \mathfrak{a}_1\vee\mathfrak{a}_2$ and the oriented $\partial$-framed segment $\mathfrak{b}$.

The existence of these objects rely on the connection principle (Theorem \ref{connection_principle_with_height}) and the fact that $*$ lies in the thick part. Here we have length estimate 
$$R-\frac{1}{2}l(\mathfrak{s}_{\hat{g}})+\frac{1}{2}I(\frac{\pi}{3}), 2R-l(\mathfrak{s}_{\hat{g}})\in (\frac{R}{4}, 2R),$$
 and the oriented $\partial$-framed segments have height at most 
 $$\frac{1}{2}\ln{2R}+C\ln{\frac{1}{\delta}}+C<\frac{3}{5}\ln{R},\ \text{by (\ref{R4})}.$$
 Actually Theorem \ref{connection_principle_with_height} can be used to construct objects in Condition \ref{defining} with $(10^3\delta)$-closeness replaced by $\delta$-closeness. Here the $(10^3\delta)$-closeness is also for convenience of the proof of Theorem \ref{main1}.

Now for any $\hat{g}\in \hat{\mathscr{C}}_{DL}\cap \hat{\mathscr{B}}_{\frac{3}{2}R}^{h-2}$, we define 
$$\Psi_1(\hat{g})=[\mathfrak{s}_{\hat{g}}\mathfrak{a}_{01}]_{R,\epsilon}^{h,h'}+[\mathfrak{s}_{\hat{g}}\mathfrak{a}_{12}]_{R,\epsilon}^{h,h'}+[\mathfrak{s}_{\hat{g}}\mathfrak{a}_{20}]_{R,\epsilon}^{h,h'}-[\mathfrak{s}_{\hat{g}}\mathfrak{b}]_{R,\epsilon}^{h,h'}-[\mathfrak{s}_{\hat{g}}\bar{\mathfrak{b}}]_{R,\epsilon}^{h,h'}\in \Omega_{R,\epsilon}^{h,h'}(M).$$

The following lemma is parallel to Lemma 5.18 of \cite{LM}, which proves that $\Psi_1(\hat{g})$ is well-defined.

\begin{lemma}\label{welldefined}
The set-theoretic map $\Psi_1: \hat{\mathscr{C}}_{DL}\cap \hat{\mathscr{B}}_{\frac{3}{2}R}^{h-2}\to \Omega_{R,\epsilon}^{h,h'}(M)$ is well-defined, i.e. the following conditions hold.
\begin{enumerate}
\item All curves in the expression of $\Psi_1(\hat{g})$ lie in ${\bf \Gamma}_{R,\epsilon}^{h}$;
\item For any $\hat{g}\in \hat{\mathscr{C}}_{DL}\cap \hat{\mathscr{B}}_{\frac{3}{2}R}^{h-2}$, $\Psi_1(\hat{g})$ does not depend on the choices in the definition.
\end{enumerate}
\end{lemma}

\begin{proof}
For (1), by Lemma \ref{lengthphase}, one can check that each curve lies in ${\bf \Gamma}_{R,10^4\delta}$ directly. For the height control, we know that $\hat{g}$ has height at most $h-2$, while $\mathfrak{a}_i$ and $\mathfrak{b}$ have height at most $\frac{3}{5}\ln{R}<h-2$. Then Lemma \ref{distance} implies that $[\mathfrak{s}_{\hat{g}}\mathfrak{a}_{01}]\in {\bf \Gamma}_{R,\epsilon}^{h}$ holds. The same argument works for all other terms.

For (2), we first fix $\mathfrak{s}_{\hat{g}}$ and $\vec{n}_{\hat{g}}$, then take another right-handed tripod $\mathfrak{a}'_0\vee \mathfrak{a}_1'\vee \mathfrak{a}_2'$ and another oriented $\partial$-framed segment $\mathfrak{b}'$ satisfying Condition \ref{defining}. Then we need to prove that 
\begin{align*}
&[\mathfrak{s}_{\hat{g}}\mathfrak{a}_{01}]_{R,\epsilon}^{h,h'}+[\mathfrak{s}_{\hat{g}}\mathfrak{a}_{12}]_{R,\epsilon}^{h,h'}+[\mathfrak{s}_{\hat{g}}\mathfrak{a}_{20}]_{R,\epsilon}^{h,h'}-[\mathfrak{s}_{\hat{g}}\mathfrak{b}]_{R,\epsilon}^{h,h'}-[\mathfrak{s}_{\hat{g}}\bar{\mathfrak{b}}]_{R,\epsilon}^{h,h'}\\
=&[\mathfrak{s}_{\hat{g}}\mathfrak{a}'_{01}]_{R,\epsilon}^{h,h'}+[\mathfrak{s}_{\hat{g}}\mathfrak{a}'_{12}]_{R,\epsilon}^{h,h'}+[\mathfrak{s}_{\hat{g}}\mathfrak{a}'_{20}]_{R,\epsilon}^{h,h'}-[\mathfrak{s}_{\hat{g}}\mathfrak{b}']_{R,\epsilon}^{h,h'}-[\mathfrak{s}_{\hat{g}}\bar{\mathfrak{b}}']_{R,\epsilon}^{h,h'}.
\end{align*}

We take an auxiliary left-handed tripod $\mathfrak{c}_0\vee \mathfrak{c}_2\vee \mathfrak{c}_2$ such that the following conditions hold.
\begin{enumerate}
\item[(a)] The tripod $\mathfrak{c}_0\vee \mathfrak{c}_2\vee \mathfrak{c}_2$ is $(\frac{1}{2}(l(\mathfrak{s}_{\hat{g}})+I(\frac{\pi}{3})),\delta)$-nearly regular.
\item[(b)] For each $i$, the initial point of $\mathfrak{c}_i$ is a point in the thick part of $M$, the terminal point of $\mathfrak{c}_i$ is $*$,  and the height of $\mathfrak{c}_i$ is at most $\frac{3}{5}\ln{R}$.
\item[(c)] The terminal direction of $\vec{t}_{\text{ter}}(\mathfrak{c}_i)$ is $\delta$-close to $-\vec{t}$, and the terminal framing of $\vec{n}_{\text{ter}}(\mathfrak{c}_i)$ is $\delta$-close to $\vec{n}_{\hat{\mathfrak{g}}}$.
\end{enumerate}
Note that the desired length of $\mathfrak{c}_i$ lies in the interval 
$$\big(\frac{1}{2}(2DL+I(\frac{\pi}{3})),\frac{1}{2}(\frac{3}{2}R+I(\frac{\pi}{2}))\big)\subset (DL,R).$$ By our choice of $L$ in (\ref{L2}), Theorem \ref{connection_principle_with_height} can be applied to construct the left-handed tripod $\mathfrak{c}_0\vee \mathfrak{c}_2\vee \mathfrak{c}_2$. The height control follows from $\frac{1}{2}\ln{R}+C\ln{\frac{1}{\delta}}+C<\frac{3}{5}\ln{R}$, by (\ref{R4}).

Here $\mathfrak{a}_0\vee \mathfrak{a}_1\vee \mathfrak{a}_2$ and $\mathfrak{c}_0\vee \mathfrak{c}_1\vee \mathfrak{c}_2$  form a $(100L,10^4\delta)$-tame rotation pair of tripods of opposite chirality, with height at most $\frac{3}{5}\ln{R}$, and the same holds for $\mathfrak{a}'_0\vee \mathfrak{a}'_1\vee \mathfrak{a}'_2$ and $\mathfrak{c}_0\vee \mathfrak{c}_1\vee \mathfrak{c}_2$. Then Lemma \ref{rotation} (1) (rotation) implies that $$\sum_{i\in \mathbb{Z}/3\mathbb{Z}}[\mathfrak{a}_{i,i+1}\bar{\mathfrak{c}}_{i,i+1}]_{R,\epsilon}^{h,h'}=0\ \text{and}\ \sum_{i\in \mathbb{Z}/3\mathbb{Z}}[\mathfrak{a}'_{i,i+1}\bar{\mathfrak{c}}_{i,i+1}]_{R,\epsilon}^{h,h'}=0\ \text{in}\ \Omega_{R,\epsilon}^{h,h'}(M).$$ 

Note that for all lemmas in Section \ref{construction}, the height of the output $(R,\epsilon)$-panted subsurface is greater than the height of the input $(R,\epsilon)$-curves by at most $3\ln{R}$. Since $\beta-\alpha>3$, we have $h-h'=(\beta-\alpha)\ln{R}>3\ln{R}$, and all those lemmas can be translated to equations in $\Omega_{R,\epsilon}^{h,h'}$. We will use such equations for many times in the remainder of this paper, without explicitly referring to the condition $\beta-\alpha>3$.

For each $i\in \mathbb{Z}/3\mathbb{Z}$, since $[\mathfrak{s}_{\hat{g}}\mathfrak{a}_{i,i+1}]$ and $[\bar{\mathfrak{c}}_{i,i+1}\mathfrak{a}'_{i,i+1}]$ form a $(10L,10^5\delta)$-tame swap pair of bigons, by Lemma \ref{swap} (swapping), we have 
$$[\mathfrak{s}_{\hat{g}}\mathfrak{a}_{i,i+1}]_{R,\epsilon}^{h,h'}+[\bar{\mathfrak{c}}_{i,i+1}\mathfrak{a}'_{i,i+1}]_{R,\epsilon}^{h,h'}=[\mathfrak{s}_{\hat{g}}\mathfrak{a}'_{i,i+1}]_{R,\epsilon}^{h,h'}+[\bar{\mathfrak{c}}_{i,i+1}\mathfrak{a}_{i,i+1}]_{R,\epsilon}^{h,h'}\ \text{in}\ \Omega_{R,\epsilon}^{h,h'}(M).$$
Moreover, since $[\mathfrak{s}_{\hat{g}}\mathfrak{b}]$ and $[\bar{\mathfrak{s}}_{\hat{g}}\mathfrak{b}']$ form a $(10L,10^5\delta)$-tame swap pair of bigons, by Lemma \ref{swap} (swapping) again, we have 
$$[\mathfrak{s}_{\hat{g}}\mathfrak{b}]_{R,\epsilon}^{h,h'}+[\bar{\mathfrak{s}}_{\hat{g}}\mathfrak{b}']_{R,\epsilon}^{h,h'}=[\mathfrak{s}_{\hat{g}}\mathfrak{b}']_{R,\epsilon}^{h,h'}+[\bar{\mathfrak{s}}_{\hat{g}}\mathfrak{b}]_{R,\epsilon}^{h,h'}\ \text{in}\ \Omega_{R,\epsilon}^{h,h'}(M),$$ which is equivalent to 
$$-[\mathfrak{s}_{\hat{g}}\mathfrak{b}]_{R,\epsilon}^{h,h'}-[\mathfrak{s}_{\hat{g}}\bar{\mathfrak{b}}]_{R,\epsilon}^{h,h'}=-[\mathfrak{s}_{\hat{g}}\mathfrak{b}']_{R,\epsilon}^{h,h'}-[\mathfrak{s}_{\hat{g}}\bar{\mathfrak{b}}']_{R,\epsilon}^{h,h'}\ \text{in}\ \Omega_{R,\epsilon}^{h,h'}(M).$$

The sum of these equations implies that the two definitions of $\Psi_1(\hat{g})$ given by different choices of $\mathfrak{a}_0\vee \mathfrak{a}_1 \vee \mathfrak{a}_2$ and $\mathfrak{b}$ are same with each other.

Now we check that $\Psi_1(\hat{g})$ also does not depend on the choice of $\mathfrak{s}_{\hat{g}}$ and $\vec{n}_{\hat{g}}$. If one take another set of $\mathfrak{s}_{\hat{g}}'$ and $\vec{n}_{\hat{g}}'$, then $\mathfrak{s}_{\hat{g}}$ and $\mathfrak{s}_{\hat{g}}'$ have the same carrier segments, the initial and terminal framings of $\mathfrak{s}_{\hat{g}}$ and $\mathfrak{s}_{\hat{g}}'$ are $(300\delta)$-close to each other respectively, while $\vec{n}_{\hat{g}}$ and $\vec{n}_{\hat{g}}'$ are $(900\delta)$-close to each other. Now we take $\mathfrak{a}_0\vee \mathfrak{a}_1\vee \mathfrak{a}_2$ and $\mathfrak{b}$ satisfying Condition \ref{defining} with respect to $\vec{n}_{\hat{g}}$, with $(10^3\delta)$-closeness replaced by $\delta$-closeness. Then they also satisfy Condition \ref{defining} with respect to $\vec{n}_{\hat{g}}'$. Then the two definitions of $\Psi_1(\hat{g})$ based on $\vec{n}_{\hat{g}}$ and $\vec{n}_{\hat{g}}'$ are equal to each other. So $\Psi_1(\hat{g})$ is well-defined.
\end{proof}

Now we prove that $\Psi_1$ satisfies the triangular relation, which parallels with Lemma 5.19 of \cite{LM}.

\begin{lemma}\label{relation}
For the set-theoretic map $\Psi_1: \hat{\mathscr{C}}_{DL}\cap \hat{\mathscr{B}}_{\frac{3}{2}R}^{h-2}\to \Omega_{R,\epsilon}^{h,h'}(M)$, the following relations hold.
\begin{enumerate}
\item For any $\hat{g}\in \hat{\mathscr{C}}_{DL}\cap \hat{\mathscr{B}}_{\frac{3}{2}R}^{h-2}$, 
$$\Psi_1(\hat{g})+\Psi_1(\hat{g}^{-1})=0.$$
\item For any triple $\hat{g}_0,\hat{g}_1,\hat{g}_2\in \hat{\mathscr{C}}_{DL}\cap \hat{\mathscr{B}}_{\frac{3}{2}R-4DL}^{h-3}$ satisfying $\hat{g}_0\hat{g}_1\hat{g}_2=1,$ 
$$\Psi_1(\hat{g}_0)+\Psi_1(\hat{g}_1)+\Psi_1(\hat{g}_2)=0.$$
\end{enumerate}
\end{lemma}

Before proving Lemma \ref{relation}, we first prove a lemma that characterizes some elements of $\pi_1(M,*)$ lying in $\mathscr{C}_{DL}$.

\begin{lemma}\label{simplecharacterize}
Suppose that $g\in \pi_1(M,*)$ can be represented by a concatenation of geodesic segments $\gamma_1\gamma_2\gamma_3$, such that the one of the following holds, then $g\in \mathscr{C}_{DL}$ holds. 
\begin{enumerate}
\item Both $\gamma_1$ and $\bar{\gamma}_3$ have length at least $DL$ and tangent to $\vec{t}$ at their initial points, $\gamma_2$ has length at least $2$, and the bending angles at the initial and terminal points of $\gamma_2$ are at most $2\delta$.
\item $\gamma_1=k^{-1}$, $\gamma_3=k$, and the bending angles at the initial and terminal points of $\gamma_2$ are both at most $20\delta$.
\end{enumerate}
\end{lemma}

\begin{proof}
We start with proving (1). We lift the path $\gamma_1\gamma_2\gamma_3$ to the universal cover $\mathbb{H}^3$ of $M$, with initial point $O$, and we still denote these three geodesic segments by the same notation. 

Let  $P_1$ be the hyperbolic plane perpendicular to $\gamma_1$ at the end point of $\gamma_1$, and let $P_2$ be the hyperbolic plane perpendicular to $\gamma_2$ at the middle point of $\gamma_2$. Then we claim that $P_1\cap P_2=\emptyset$.

The claim follows from the following fact in hyperbolic geometry. In a right-angled triangle $\Delta$ with one ideal vertex, let $a$ be the length of the finite edge of $\Delta$, and let $\beta$ be the angle adjacent to the finite edge that is not the right-angle. Then $\sin{\beta}\cosh{a}=1$ holds. Here we have the condition $\sin{(\pi-2\delta)}\cosh{1}>1$, so $P_1\cap P_2=\emptyset$.

Let  $P_3$ be the hyperbolic plane perpendicular to $\gamma_3$ at the initial point of $\gamma_3$. The same argument implies that  $P_3\cap P_2=\emptyset$. Then we know that $P_2$ separates $U_{DL}$ and $g(U_{DL})$ in $\mathbb{H}^3$, since it separates $P_1$ and $P_3$. So $g\in \mathscr{C}_{DL}$ holds.

The proof of (2) is similar to (1). Actually, the statement follows from the fact that the following two hyperbolic planes are disjoint from each other.
\begin{itemize}
\item The hyperbolic plane going through the initial point of $\gamma_2$ and perpendicular to $\gamma_2$.
\item Boundary of the translation of $U_{DL}$ by the isometry $\gamma_1=k^{-1}\in\text{Isom}_+(\mathbb{H}^3)$.
\end{itemize}
\end{proof}

\begin{proof}[Proof of Lemma \ref{relation}]
For statement (1), let $\mathfrak{s}_{\hat{g}}$ be an oriented $\partial$-framed segment associated to $\hat{g}$ and let $\mathfrak{a}_0\vee \mathfrak{a}_1\vee \mathfrak{a}_2$ and $\mathfrak{b}$ be the auxiliary data for defining $\Psi_1(\hat{g})$. Then by Definition \ref{sharpelement}, $\bar{\mathfrak{s}}^*_{\hat{g}}$ is associated to $\hat{g}^{-1}$, and we can choose the auxiliary data for defining $\Psi_1(\hat{g}^{-1})$ to be $\mathfrak{a}_0^*\vee \mathfrak{a}_2^*\vee \mathfrak{a}^*_1$ and $\bar{\mathfrak{b}}^*$. Then we compute 
\begin{align*}
&\Psi_1(\hat{g})+\Psi_1(\hat{g}^{-1})\\
=\ &( [\mathfrak{s}_{\hat{g}}\mathfrak{a}_{01}]_{R,\epsilon}^{h,h'}+[\mathfrak{s}_{\hat{g}}\mathfrak{a}_{12}]_{R,\epsilon}^{h,h'}+[\mathfrak{s}_{\hat{g}}\mathfrak{a}_{20}]_{R,\epsilon}^{h,h'}-[\mathfrak{s}_{\hat{g}}\mathfrak{b}]_{R,\epsilon}^{h,h'}-[\mathfrak{s}_{\hat{g}}\bar{\mathfrak{b}}]_{R,\epsilon}^{h,h'})\\
&+([\bar{\mathfrak{s}}^*_{\hat{g}}\mathfrak{a}^*_{02}]_{R,\epsilon}^{h,h'}+[\bar{\mathfrak{s}}^*_{\hat{g}}\mathfrak{a}_{21}]_{R,\epsilon}^{h,h'}+[\bar{\mathfrak{s}}^*_{\hat{g}}\mathfrak{a}_{10}]_{R,\epsilon}^{h,h'}-[\bar{\mathfrak{s}}^*_{\hat{g}}\bar{\mathfrak{b}}^*]_{R,\epsilon}^{h,h'}-[\bar{\mathfrak{s}}^*_{\hat{g}}\mathfrak{b}^*]_{R,\epsilon}^{h,h'})\\
=\ & 0\in \Omega_{R,\epsilon}^{h,h'}(M).
\end{align*}
Here the last equation holds since the ten items cancel with each other pairwisely. For example, we have 
$$[\bar{\mathfrak{s}}^*_{\hat{g}}\mathfrak{a}^*_{02}]_{R,\epsilon}^{h,h'}=[\overline{\mathfrak{s}^*_{\hat{g}}\mathfrak{a}^*_{20}}]_{R,\epsilon}^{h,h'}=-[\mathfrak{s}^*_{\hat{g}}\mathfrak{a}^*_{20}]_{R,\epsilon}^{h,h'}=-[\mathfrak{s}_{\hat{g}}\mathfrak{a}_{20}]_{R,\epsilon}^{h,h'}\in \Omega_{R,\epsilon}^{h,h'}(M).$$

\bigskip

For statement (2), we first need to set up some preparational work.

For any triple of elements $\hat{g}_0,\hat{g}_1,\hat{g}_2\in \hat{\mathscr{C}}_{DL}\cap \hat{\mathscr{B}}_{\frac{3}{2}R}^{h-2}$ satisfying $\hat{g}_0\hat{g}_1\hat{g}_2=1,$ there is a left handed $(100L,\delta)$-tame tripod $\mathfrak{t}_0\vee \mathfrak{t}_1\vee \mathfrak{t}_2$ of height at most $h-1$, such that the carrier segment of $\mathfrak{t}_{i,i+1}$ is the projection of $\hat{g}_{i+2}\in\pi_1(\text{SO}(M),\text{e})$ to $M$ for any $i\in \mathbb{Z}/3\mathbb{Z}$. Actually, the geodesic segments corresponding to $\hat{g}_0,\hat{g}_2,\hat{g}_2$ bound a geodesic triangle in $M$, and the common initial point of $\mathfrak{t}_0\vee \mathfrak{t}_1\vee \mathfrak{t}_2$ is the Fermat point of this triangle. Since each $\hat{g}_i \in \hat{\mathscr{C}}_{DL}$, by Lemma 4.11 of \cite{LM}, the length of each $\mathfrak{t}_i$ is at least $DL-1>200L$. The height bound of $\mathfrak{t}_0\vee \mathfrak{t}_1\vee \mathfrak{t}_2$ follows from the fact that the carrier segment of each $\mathfrak{t}_i$ lies in the $1$-neighborhood of the union of geodesic segments corresponding to $\hat{g}_0$, $\hat{g}_1$ and $\hat{g}_2$ (since the bending angle of $\bar{\mathfrak{t}}_i\mathfrak{t}_{i+1}$ is $\frac{\pi}{3}$).

By Lemma 5.9 of \cite{LM}, there are angles $\phi_0,\phi_1,\phi_2\in \mathbb{R}/2\pi \mathbb{Z}$ with $\phi_0+\phi_1+\phi_2=0$, such that $2\phi_{i+2}$ is $(2\delta)$-close to the angle from $\vec{n}_{\text{ter}}(\mathfrak{t}_i)$ to $\vec{n}_{\text{ter}}(\mathfrak{t}_{i+1})$ with respect to their common orthogonal vector $\delta$-close to $-\vec{t}$, and $\mathfrak{t}_{i,i+1}(\phi_{i+2})$ is associated to the $\delta$-sharp element $\hat{g}_{i+2}\in \pi_1(\text{SO}(M),\text{e})$. So we can simply take $\mathfrak{s}_{\hat{g}_{i+2}}$ to be  $\mathfrak{t}_{i,i+1}(\phi_{i+2})$.

{\bf Step I.} We assume that for $\hat{g}_0,\hat{g}_1,\hat{g}_2\in \hat{\mathscr{C}}_{DL}\cap \hat{\mathscr{B}}_{\frac{3}{2}R}^{h-2}$, the corresponding tripod $\mathfrak{t}_0\vee \mathfrak{t}_1\vee\mathfrak{t}_2$ is $(l,10\delta)$-nearly regular for some constant $l\geq 200L$, all $\vec{n}_{\text{ter}}(\mathfrak{t}_i)$ are $\delta$-close to the vector $\vec{n}$, and all $\phi_i$ are $0$. Then we prove that 
$$\Psi_1(\hat{g}_0)+\Psi_1(\hat{g}_1)+\Psi_1(\hat{g}_2)=0.$$

For each $r\in \mathbb{Z}/3\mathbb{Z}$, we can take $\mathfrak{s}_{\hat{g}_{r+2}}$ to be $\mathfrak{t}_{r,r+1}$. Then there exists a right-handed tripod $\mathfrak{a}_0\vee\mathfrak{a}_1\vee \mathfrak{a}_2$ of height at most $\frac{3}{5}\ln{R}$, such that for any $r\in \mathbb{Z}/3\mathbb{Z}$, $\mathfrak{a}_0\vee\mathfrak{a}_1\vee \mathfrak{a}_2$ and $\mathfrak{a}_{r,r+1}$ can be used to define $\Psi_1(\hat{g}_{r+2})$ (satisfying Condition \ref{defining} with respect to $\hat{g}_{r+2}$).

So we have 
\begin{align*}
&\Psi_1(\hat{g}_0)+\Psi_1(\hat{g}_1)+\Psi_1(\hat{g}_2)\\
=\ & \sum_{r=0}^2\Big(\sum_{i=0}^2[\mathfrak{t}_{r,r+1}\mathfrak{a}_{i,i+1}]_{R,\epsilon}^{h,h'}-[\mathfrak{t}_{r,r+1}\mathfrak{a}_{r,r+1}]_{R,\epsilon}^{h,h'}-[\mathfrak{t}_{r,r+1}\bar{\mathfrak{a}}_{r,r+1}]_{R,\epsilon}^{h,h'}\Big)\\
=\ & \sum_{r=0}^2[\mathfrak{t}_{r,r+1}\mathfrak{a}_{r+1,r+2}]_{R,\epsilon}^{h,h'}+ \sum_{r=0}^2[\mathfrak{t}_{r,r+1}\mathfrak{a}_{r+2,r}]_{R,\epsilon}^{h,h'}-\sum_{r=0}^2[\mathfrak{t}_{r,r+1}\bar{\mathfrak{a}}_{r,r+1}]_{R,\epsilon}^{h,h'}\\
=\ & 2 \sum_{r=0}^2[\mathfrak{t}_{r,r+1}\mathfrak{a}_{r,r+1}]_{R,\epsilon}^{h,h'}-0\\
=\ & 2 \sum_{r=0}^2[\mathfrak{t}_{r,r+1}\bar{\mathfrak{a}}_{r+1,r}]_{R,\epsilon}^{h,h'}\\
=\ &0
\end{align*} 

For the third equation, the first and second terms are both equal to $\sum_{r=0}^2[\mathfrak{t}_{r,r+1}\mathfrak{a}_{r,r+1}]_{R,\epsilon}^{h,h'}$. The first term follows from the following computation, by invoking Lemma \ref{swap} (swapping). The second term follows from a similar computation.
\begin{align*}
&[\mathfrak{t}_{01}\mathfrak{a}_{12}]_{R,\epsilon}^{h,h'}+([\mathfrak{t}_{12}\mathfrak{a}_{20}]_{R,\epsilon}^{h,h'}+[\mathfrak{t}_{20}\mathfrak{a}_{01}]_{R,\epsilon}^{h,h'})\\
=\ & [\mathfrak{t}_{01}\mathfrak{a}_{12}]_{R,\epsilon}^{h,h'}+[\mathfrak{t}_{12}\mathfrak{a}_{01}]_{R,\epsilon}^{h,h'}+[\mathfrak{t}_{20}\mathfrak{a}_{20}]_{R,\epsilon}^{h,h'}\\
=\ & ([\mathfrak{t}_{01}\mathfrak{a}_{01}]_{R,\epsilon}^{h,h'}+[\mathfrak{t}_{12}\mathfrak{a}_{12}]_{R,\epsilon}^{h,h'})+[\mathfrak{t}_{20}\mathfrak{a}_{20}]_{R,\epsilon}^{h,h'}. 
\end{align*}
The third term equals $0$ follows from Lemma \ref{rotation} (1) (rotation), since $\mathfrak{a}_0\vee \mathfrak{a}_1 \vee \mathfrak{a}_2$ and $\mathfrak{t}_0\vee \mathfrak{t}_1 \vee \mathfrak{t}_2$  have opposite chiralty. The last equation follows from Lemma \ref{antirotation} (1) (antirotation).

\bigskip

{\bf Step II.} In this step, we prove a connecting result. Suppose that $\mathfrak{c}_0\mathfrak{r}_0\vee \mathfrak{c}_1\mathfrak{r}_1\vee \mathfrak{c}_2\mathfrak{r}_2$ and $\mathfrak{c}_0\mathfrak{r}'_0\vee \mathfrak{c}_1\mathfrak{r}'_1\vee \mathfrak{c}_2\mathfrak{r}'_2$ are $(100L,10\delta)$-tame left-handed tripods of height at most $h-1$ satisfying the following conditions.
\begin{enumerate}
\item[(a)] The left-handed tripod $\mathfrak{c}_0\vee \mathfrak{c}_1\vee \mathfrak{c}_2$ is $(2L,\delta)$-nearly regular. The phases of $\mathfrak{c}_i$ and $\mathfrak{r}_i$ are all $\delta$-close to $0$.
\item[(b)] For each $i\in \mathbb{Z}/3\mathbb{Z}$, $\mathfrak{c}_i$ and $\mathfrak{r}_i$ are $(10\delta)$-consecutive and $(L,10\delta)$-tame. The terminal point of $\mathfrak{r}_i$ is $*$.
\item[(c)] For each $i\in \mathbb{Z}/3\mathbb{Z}$, $l(\mathfrak{r}_i)$ is $(10\delta)$-close to $l(\mathfrak{r}_i')$, $\vec{t}_{\text{ter}}(\mathfrak{r}_i)$ and $\vec{t}_{\text{ter}}(\mathfrak{r}'_i)$ are both $\delta$-close to $-\vec{t}$, and $\vec{n}_{\text{ter}}(\mathfrak{r}_i)$ is $(10\delta)$-close to $\vec{n}_{\text{ter}}(\mathfrak{r}'_i)$.
\item[(d)] The group elements $\hat{g}_0,\hat{g}_1,\hat{g}_2, \hat{g}'_0,\hat{g}'_1,\hat{g}'_2$ defined in the following lie in $\hat{\mathscr{C}}_{DL}\cap \hat{\mathscr{B}}_{\frac{3}{2}R}^{h-2}$.
\end{enumerate}

Let $\phi_0,\phi_1,\phi_2\in \mathbb{R}/2\pi\mathbb{Z}$ be the angles associated to $\mathfrak{c}_0\mathfrak{r}_0\vee \mathfrak{c}_1\mathfrak{r}_1\vee \mathfrak{c}_2\mathfrak{r}_2$ given by the setup before step I. Then these angles also work for $\mathfrak{c}_0\mathfrak{r}_0'\vee \mathfrak{c}_1\mathfrak{r}_1'\vee \mathfrak{c}_2\mathfrak{r}_2'$ up to $(30\delta)$-closeness for the condition on framings before Step I.
We defined $\hat{g}_{i+2},\hat{g}_{i+2}'\in \pi_1(\text{SO}(M),\text{e})$ to be the $\delta$-sharp elements accociated to $(\bar{\mathfrak{r}}_i\mathfrak{c}_{i,i+1}\mathfrak{r}_{i+1})(\phi_{i+2}), (\bar{\mathfrak{r}}'_i\mathfrak{c}_{i,i+1}\mathfrak{r}'_{i+1})(\phi_{i+2})$ respectively. 

Then we want to prove
$$\Psi_1(\hat{g}_0)+\Psi_1(\hat{g}_1)+\Psi_1(\hat{g}_2)=\Psi_1(\hat{g}_0')+\Psi_1(\hat{g}_1')+\Psi_1(\hat{g}_2').$$

Actually, it suffices to consider the case that $\mathfrak{r}_i \ne \mathfrak{r}'_i$ for only one $i\in\{0,1,2\}$. Without loss of generality, we work on the pair $\mathfrak{c}_0\mathfrak{r}_0 \vee \mathfrak{c}_1\mathfrak{r}_1\vee \mathfrak{c}_2\mathfrak{r}_2$ and $\mathfrak{c}_0\mathfrak{r}'_0 \vee \mathfrak{c}_1\mathfrak{r}_1\vee \mathfrak{c}_2\mathfrak{r}_2$. Since $\hat{g}_0=\hat{g}_0'$ holds in this case, we only need to prove
$$\Psi_1(\hat{g}_1)+\Psi_1(\hat{g}_2)=\Psi_1(\hat{g}'_1)+\Psi_1(\hat{g}'_2).$$

Now we take auxiliary data $\mathfrak{a}_0^{(1)}\vee\mathfrak{a}_1^{(1)}\vee \mathfrak{a}_2^{(1)}$ and $\mathfrak{b}^{(1)}$ for computing $\Psi_1(\hat{g}_1)$ that statisfy Condition \ref{defining}, with $(10^3\delta)$-closeness replaced by $\delta$-closeness. Then $\mathfrak{a}_0^{(1)}\vee\mathfrak{a}_1^{(1)}\vee \mathfrak{a}_2^{(1)}$ and $\mathfrak{b}^{(1)}$ also statisfy Condition \ref{defining} for computing $\Psi_1(\hat{g}_1')$. Similarly, we take  auxiliary data $\mathfrak{a}_0^{(2)}\vee\mathfrak{a}_1^{(2)}\vee \mathfrak{a}_2^{(2)}$ and $\mathfrak{b}^{(2)}$ satisfying Condition \ref{defining} for computing both $\Psi_1(\hat{g}_2)$ and $\Psi_1(\hat{g}_2')$.

Since the lengths of $\mathfrak{c}_i\mathfrak{r}_i$ and $\mathfrak{c}_i\mathfrak{r}_i'$ are at least $100L$ and $\mathfrak{c}_0\vee \mathfrak{c}_1\vee \mathfrak{c}_2$ is $(2L,\delta)$-regular, the lengths of $\mathfrak{r}_i$ and $\mathfrak{r}_i'$ are at least $50L$. By using Lemma \ref{swap} (swapping), we swap the $(10L,10^5\delta)$-pair of bigon $[(\mathfrak{r}_0)(\mathfrak{a}_{01}^{(1)}(-\phi_1)\bar{\mathfrak{r}}_2\mathfrak{c}_{20})]$ and $[(\mathfrak{r}'_0)(\bar{\mathfrak{a}}_{01}^{(2)}(-\phi_2)\bar{\mathfrak{r}}_1\mathfrak{c}_{10})]$ to get 
\begin{align*}
&[(\mathfrak{r}_0)(\mathfrak{a}_{01}^{(1)}(-\phi_1)\bar{\mathfrak{r}}_2\mathfrak{c}_{20})]_{R,\epsilon}^{h,h'}+[(\mathfrak{r}'_0)(\bar{\mathfrak{a}}_{01}^{(2)}(-\phi_2)\bar{\mathfrak{r}}_1\mathfrak{c}_{10})]_{R,\epsilon}^{h,h'}\\
=\ &[(\mathfrak{r}_0)(\bar{\mathfrak{a}}_{01}^{(2)}(-\phi_2)\bar{\mathfrak{r}}_1\mathfrak{c}_{10})]_{R,\epsilon}^{h,h'}+[(\mathfrak{r}'_0)(\mathfrak{a}_{01}^{(1)}(-\phi_1)\bar{\mathfrak{r}}_2\mathfrak{c}_{20})]_{R,\epsilon}^{h,h'}\ \text{in}\ \Omega_{R,\epsilon}^{h,h'}(M).
\end{align*} Since $\mathfrak{r}_1=\mathfrak{r}_1'$ and $\mathfrak{r}_2=\mathfrak{r}_2'$, this equation is equivalent to

\begin{align*}
&[(\bar{\mathfrak{r}}_0\mathfrak{c}_{01}\mathfrak{r}_{1})(\phi_2)\mathfrak{a}_{01}^{(2)}]_{R,\epsilon}^{h,h'}+[(\bar{\mathfrak{r}}_2\mathfrak{c}_{20}\mathfrak{r}_{0})(\phi_1)\mathfrak{a}_{01}^{(1)}]_{R,\epsilon}^{h,h'}\\
=\ &[(\bar{\mathfrak{r}}'_0\mathfrak{c}_{01}\mathfrak{r}'_{1})(\phi_2)\mathfrak{a}_{01}^{(2)}]_{R,\epsilon}^{h,h'}+[(\bar{\mathfrak{r}}'_2\mathfrak{c}_{20}\mathfrak{r}'_{0})(\phi_1)\mathfrak{a}_{01}^{(1)}]_{R,\epsilon}^{h,h'}. 
\end{align*}

Similarly, we have 
\begin{align*}
&[(\bar{\mathfrak{r}}_0\mathfrak{c}_{01}\mathfrak{r}_{1})(\phi_2)\mathfrak{a}_{12}^{(2)}]_{R,\epsilon}^{h,h'}+[(\bar{\mathfrak{r}}_2\mathfrak{c}_{20}\mathfrak{r}_{0})(\phi_1)\mathfrak{a}_{12}^{(1)}]_{R,\epsilon}^{h,h'}\\
=\ &[(\bar{\mathfrak{r}}'_0\mathfrak{c}_{01}\mathfrak{r}'_{1})(\phi_2)\mathfrak{a}_{12}^{(2)}]_{R,\epsilon}^{h,h'}+[(\bar{\mathfrak{r}}'_2\mathfrak{c}_{20}\mathfrak{r}'_{0})(\phi_1)\mathfrak{a}_{12}^{(1)}]_{R,\epsilon}^{h,h'},
\end{align*}
\begin{align*}
&[(\bar{\mathfrak{r}}_0\mathfrak{c}_{01}\mathfrak{r}_{1})(\phi_2)\mathfrak{a}_{20}^{(2)}]_{R,\epsilon}^{h,h'}+[(\bar{\mathfrak{r}}_2\mathfrak{c}_{20}\mathfrak{r}_{0})(\phi_1)\mathfrak{a}_{20}^{(1)}]_{R,\epsilon}^{h,h'}\\
=\ &[(\bar{\mathfrak{r}}'_0\mathfrak{c}_{01}\mathfrak{r}'_{1})(\phi_2)\mathfrak{a}_{20}^{(2)}]_{R,\epsilon}^{h,h'}+[(\bar{\mathfrak{r}}'_2\mathfrak{c}_{20}\mathfrak{r}'_{0})(\phi_1)\mathfrak{a}_{20}^{(1)}]_{R,\epsilon}^{h,h'},
\end{align*}
and 
\begin{align*}
&-[(\bar{\mathfrak{r}}_0\mathfrak{c}_{01}\mathfrak{r}_{1})(\phi_2)\mathfrak{b}^{(2)}]_{R,\epsilon}^{h,h'}-[(\bar{\mathfrak{r}}_2\mathfrak{c}_{20}\mathfrak{r}_{0})(\phi_1)\mathfrak{b}^{(1)}]_{R,\epsilon}^{h,h'}\\
=\ &-[(\bar{\mathfrak{r}}'_0\mathfrak{c}_{01}\mathfrak{r}'_{1})(\phi_2)\mathfrak{b}^{(2)}]_{R,\epsilon}^{h,h'}-[(\bar{\mathfrak{r}}'_2\mathfrak{c}_{20}\mathfrak{r}'_{0})(\phi_1)\mathfrak{b}^{(1)}]_{R,\epsilon}^{h,h'}, 
\end{align*}
\begin{align*}
&-[(\bar{\mathfrak{r}}_0\mathfrak{c}_{01}\mathfrak{r}_{1})(\phi_2)\bar{\mathfrak{b}}^{(2)}]_{R,\epsilon}^{h,h'}-[(\bar{\mathfrak{r}}_2\mathfrak{c}_{20}\mathfrak{r}_{0})(\phi_1)\bar{\mathfrak{b}}^{(1)}]_{R,\epsilon}^{h,h'}\\
=\ &-[(\bar{\mathfrak{r}}'_0\mathfrak{c}_{01}\mathfrak{r}'_{1})(\phi_2)\bar{\mathfrak{b}}^{(2)}]_{R,\epsilon}^{h,h'}-[(\bar{\mathfrak{r}}'_2\mathfrak{c}_{20}\mathfrak{r}'_{0})(\phi_1)\bar{\mathfrak{b}}^{(1)}]_{R,\epsilon}^{h,h'}.
\end{align*}

The sum of these five equations gives $\Psi_1(\hat{g}_1)+\Psi_1(\hat{g}_2)=\Psi_1(\hat{g}'_1)+\Psi_1(\hat{g}'_2)$.

\bigskip

{\bf Step III.} We assume that the group elements $\hat{g}_0,\hat{g}_1,\hat{g}_2\in \hat{\mathscr{C}}_{DL}\cap \hat{\mathscr{B}}_{\frac{3}{2}R}^{h-2}$, and the length of each $\mathfrak{t}_i$ is at least $2DL$. Then we prove that 
$$\Psi_1(\hat{g}_0)+\Psi_1(\hat{g}_1)+\Psi_1(\hat{g}_2)=0.$$

For the tripod 
$\mathfrak{t}_0\vee \mathfrak{t}_1\vee \mathfrak{t}_2$, 
let 
$\phi_0,\phi_1,\phi_2\in \mathbb{R}/2\pi \mathbb{Z}$ 
be the angles chosen at the setup before Step I. We write each $\mathfrak{t}_i$ as a concatenation of oriented $\partial$-framed segments 
$\mathfrak{c}_i\mathfrak{r}_i$ so that $\mathfrak{c}_0\vee\mathfrak{c}_1\vee \mathfrak{c}_2$ 
is $(2L,\delta)$-regular. Now we interpolate 
$\mathfrak{t}_0^{(0)}\vee\mathfrak{t}_1^{(0)}\vee\mathfrak{t}_2^{(0)}=\mathfrak{c}_0\mathfrak{r}_0\vee \mathfrak{c}_1\mathfrak{r}_1\vee \mathfrak{c}_2\mathfrak{r}_2$
to a nearly regular tripod, along a sequence of tripods 
$\mathfrak{t}_0^{(k)}\vee\mathfrak{t}_1^{(k)}\vee\mathfrak{t}_2^{(k)}$ 
of height at most $h-2$ with $k=1,\cdots,N$. More precisely, we have 
$\mathfrak{t}_i^{(k)}=\mathfrak{c}_i\mathfrak{s}_i^{(k)}\mathfrak{n}_i^{(k)}$ 
for any $k=1,2,\cdots,N$ and the following hold.

\begin{enumerate}
\item[(a)] For all $i\in \mathbb{Z}/3\mathbb{Z}$ and $k>0$, $\mathfrak{c}_i,\mathfrak{s}_i^{(k)},\mathfrak{n}_i^{(k)}$ form an $(L,\delta)$-tame $\delta$-consecutive chain.
\item[(b)] For all $i\in \mathbb{Z}/3\mathbb{Z}$ and $k>0$, the carrier segment of $\bar{\mathfrak{n}}_i^{(k)}$ is the length-$DL$ geodesic segment starting at $*$ and tangent to $\vec{t}$ (so it runs around the closed geodesic $\gamma$ defined at the beginning of Section \ref{setup}, which has height at most $L$), and the phase of $\mathfrak{n}_i^{(k)}$ is $0$. The terminal framings of $\mathfrak{n}_i^{(k)}$ and $\mathfrak{n}_i^{(k+1)}$ are $(5\delta)$-close to each other.
\item[(c)]  For all $i\in \mathbb{Z}/3\mathbb{Z}$ and $k>0$, the lengths of $\mathfrak{s}_i^{(k)}$ and $\mathfrak{s}_i^{(k+1)}$ are $(5\delta)$-close to each other, the length of  $\mathfrak{s}_i^{(k)}$ is at least $\kappa^{-1}L$ and the height of $\mathfrak{s}_i^{(k)}$ is at most $h-3$.
\item[(d)] For all $i\in \mathbb{Z}/3\mathbb{Z}$, the lengths of $\mathfrak{r}_i$ and $\mathfrak{s}_i^{(1)}\mathfrak{n}_i^{(1)}$ are $(5\delta)$-close to each other, and the terminal framings of $\mathfrak{r}_i$ and $\mathfrak{s}_i^{(1)}\mathfrak{n}_i^{(1)}$ are $(5\delta)$-close to each other.
\item[(e)] For all $i\in \mathbb{Z}/3\mathbb{Z}$, the lengths of $\mathfrak{s}_i^{(N)}$ are $\delta$-close to each other, and  the terminal framings of $\mathfrak{n}_i^{(N)}$'s are $\delta$-close to each other.
\end{enumerate}

To construct these 
$\mathfrak{t}_0^{(k)}\vee \mathfrak{t}_1^{(k)}\vee \mathfrak{t}_2^{(k)}$ for $k=1,\cdots, N$, we first choose the lengths of $\mathfrak{c}_i\mathfrak{s}_i^{(k)}\mathfrak{n}_i^{(k)}$ by gradually decreasing them from $l(\mathfrak{c}_i\mathfrak{r}_i)$ to $2DL$. Then we choose the sequence of terminal framings of $\mathfrak{n}_i^{(k)}$, such that all of them are orthogonal to $\vec{t}$ and gradually converge to $\vec{n}$. Then $\mathfrak{n}_i^{(k)}$ is automatically determined by its terminal framing, and we apply Theorem \ref{connection_principle_with_height} to construct $\mathfrak{s}_i^{(k)}$. Note that the initial and terminal points of $\mathfrak{s}_i^{(k)}$ have height at most $2L$, by our choice of $\mathfrak{c}_i$ and $\mathfrak{n}_i^{(k)}$ and Lemma \ref{setuplemma} (1). Since the length of $\mathfrak{s}_i^{(k)}$ is at least $2DL-DL-2L=(D-2)L$, which is greater than $\kappa^{-1}L$, Theorem \ref{connection_principle_with_height} is applicable to construct $\mathfrak{s}_i^{(k)}$. The height bound of $\mathfrak{s}_i^{(k)}$ follows from the fact that 
$$\max{\{L+\ln{\frac{1}{\delta}}+C,\frac{1}{2}\ln{2R}+C(\ln{\frac{1}{\delta}}+1)\}}<h-3,\ \text{by}\ (\ref{R4}).$$

For each $\mathfrak{t}_0^{(k)}\vee\mathfrak{t}_1^{(k)}\vee\mathfrak{t}_2^{(k)}$, let $\phi_0^{(k)},\phi_1^{(k)},\phi_2^{(k)}\in \mathbb{R}/2\pi \mathbb{Z}$ be a triple of angles giving oriented $\partial$-framed segments $\mathfrak{t}_{i,i+1}^{(k)}(\phi_{i+2})$ and associated $\delta$-sharp elements $\hat{g}_i^{(k)}\in \pi_1(\text{SO}(M),\text{e})$ with $\hat{g}_0^{(k)}\hat{g}_1^{(k)}\hat{g}_2^{(k)}=1$, as in the setup before Step I. By our construction, we can assume that $\phi_i^{(0)}=\phi_i$ and $\phi_i^{(N)}=0$ holds for each $i\in \mathbb{Z}/3\mathbb{Z}$. 

By Lemma \ref{simplecharacterize} (1) and our construction of $\hat{g}_i^{(k)}$, each $\hat{g}_i^{(k)}$ lies in $\hat{\mathscr{C}}_{DL}$. Since $\hat{g}_i^{(k)}\in \hat{\mathscr{B}}_{\frac{3}{2}R}^{h-2}$ also holds, $\Psi_1(\hat{g}_i^{(k)})$ is defined for each $i$ and $k$.

Now Step I implies that 
$$\Psi_1(\hat{g}_0^{(N)})+\Psi_1(\hat{g}_1^{(N)})+\Psi_1(\hat{g}_2^{(N)})=0,$$ 
and Step II implies that for any $k=0,\cdots,N-1,$
$$\Psi_1(\hat{g}_0^{(k)})+\Psi_1(\hat{g}_1^{(k)})+\Psi_1(\hat{g}_2^{(k)})=\Psi_1(\hat{g}_0^{(k+1)})+\Psi_1(\hat{g}_1^{(k+1)})+\Psi_1(\hat{g}_2^{(k+1)}).$$ 
So we have  
\begin{align*}
&\Psi_1(\hat{g}_0)+\Psi_1(\hat{g}_1)+\Psi_1(\hat{g}_2)=\Psi_1(\hat{g}_0^{(0)})+\Psi_1(\hat{g}_1^{(0)})+\Psi_1(\hat{g}_2^{(0)})\\
=\ &\Psi_1(\hat{g}_0^{(N)})+\Psi_1(\hat{g}_1^{(N)})+\Psi_1(\hat{g}_2^{(N)})=0.
\end{align*}

\bigskip

{\bf Step IV.} General case. We have group elements $\hat{g}_0,\hat{g}_1,\hat{g}_2\in \hat{\mathscr{C}}_{DL}\cap \hat{\mathscr{B}}_{\frac{3}{2}R-4DL}^{h-2}$, such that $\hat{g}_0\hat{g}_1\hat{g}_2=1$, with associated tripod $\mathfrak{t}_0\vee\mathfrak{t}_1\vee \mathfrak{t}_2$ and angles $\phi_0,\phi_1,\phi_2\in \mathbb{R}/2\pi \mathbb{Z}$. Then we will prove that $$\Psi_1(\hat{g}_0)+\Psi_1(\hat{g}_1)+\Psi_1(\hat{g}_2)=0.$$

If the length of each $\mathfrak{t}_i$ is at least $2DL$, the result follows from step III. So we assume that some $\mathfrak{t}_i$ has length at most $2DL$. Recall that, by Lemma 4.11 of \cite{LM} and the fact that $\hat{g}_i\in  \hat{\mathscr{C}}_{DL}$, each $\mathfrak{t}_i$ has length at least $DL-1$.

We construct a geodesic segment $s$ of height at most $\frac{1}{2}\ln{R}$ with intial and terminal points at $*$ as the following.
\begin{enumerate}
\item[(a)] The geodesic segment $s$ is homotopic to a concatenation of geodesic segments $s_1s_2$, with bending angle at most $\delta$.
\item[(b)] For the geodesic segment $\bar{s}_2$, its initial point is $*$ and its initial direction is $\vec{t}$, with length $DL$. So $s_2$ runs around the closed geodesic $\gamma$ and its initial and terminal points have height at most $L$.
\item[(c)] The geodesic $s_1$ has length $\delta$-close to $\frac{1}{2}DL$, and its initial vector is $\delta$-close to $-\vec{t}$.
\end{enumerate}
Again we use Theorem \ref{connection_principle_with_height} to construct $s_1$. Here we use the fact that $*$ and the initial point of $s_2$ have height at most $L$, while the length of $s_1$ is $\frac{1}{2}DL>\kappa^{-1}L$. The height control of $s$ follows from Theorem \ref{connection_principle_with_height}, (\ref{L3}), (\ref{K1}) and (\ref{R4}). Moreover, by Lemma \ref{simplecharacterize} (1), each $g_i'=\bar{s}g_is$ is an element in $\mathscr{C}_{DL}$. By our construction of $s$ and the fact that $g_i\in \mathscr{B}_{\frac{3}{2}R-4DL}^{h-3}$, each $g_i'=\bar{s}g_is$ is an element in $\mathscr{B}_{\frac{3}{2}R}^{h-2}$.

Now we extend the tripod $\mathfrak{t}_0\vee\mathfrak{t}_1\vee\mathfrak{t}_2$ to  $\mathfrak{t}_0\mathfrak{s}_0\vee\mathfrak{t}_1\mathfrak{s}_1\vee\mathfrak{t}_2\mathfrak{s}_2$. Here each oriented $\partial$-framed segment $\mathfrak{s}_i$ has $s$ as its carrier segment, with phase $0$, and we take the frame of each $\mathfrak{s}_i$ so that it is $\delta$-consecutive with $\mathfrak{t}_i$. Then there are group elements $\hat{g}_0',\hat{g}_1',\hat{g}_2'\in \hat{\mathscr{C}}_{DL}\cap \hat{\mathscr{B}}_{\frac{3}{2}R}^{h-2}$ associated to $\mathfrak{t}_0\mathfrak{s}_0\vee\mathfrak{t}_1\mathfrak{s}_1\vee\mathfrak{t}_2\mathfrak{s}_2$ with $\hat{g}_0'\hat{g}_1'\hat{g}_2'=1$.

Here the length of each $\mathfrak{t}_i\mathfrak{s}_i$ is at least $(DL-1)+(DL+\frac{1}{2}DL)-1>2DL$.
By step III, we have $$\Psi_1(\hat{g}_0')+\Psi_1(\hat{g}_1')+\Psi_1(\hat{g}_2')=0.$$ Now it suffices to prove that $\Psi_1(\hat{g}_i')=\Psi_1(\hat{g}_i)$ holds for each $i$.

Suppose the $\delta$-sharp element $\hat{g}_i\in \pi_1(\text{SO}(M),\text{e})$ is associated with the oriented $\partial$-framed segment $\mathfrak{t}_{i,i+1}(\phi_{i+2})$, then $\hat{g}_i'$ is associated with $(\bar{\mathfrak{s}}_i\mathfrak{t}_{i,i+1}\mathfrak{s}_{i+1})(\phi_{i+2})$. Actually, the oriented $\partial$-framed segment $\mathfrak{s}_{i}(-\phi_{i+2})$ equals $\mathfrak{s}_{i+1}(\phi_{i+2})$ up to $(2\delta)$-closeness of framings. 

 For defining $\Psi_1(\hat{g}_{i+2}')$, we take auxiliary data $\mathfrak{a}_0\vee\mathfrak{a}_1\vee\mathfrak{a}_2$ and $\mathfrak{b}$ as in Condition \ref{defining}, with $(10^3\delta)$-closeness replaced by $\delta$-closeness. Then we can take 
 $$\mathfrak{a}_0\bar{\mathfrak{s}}_{i}(\phi_{i+2})\vee\mathfrak{a}_1\bar{\mathfrak{s}}_{i}(\phi_{i+2})\vee \mathfrak{a}_2\bar{\mathfrak{s}}_{i}(\phi_{i+2})$$ 
 and $$\mathfrak{s}_i(-\phi_{i+2})\mathfrak{b}\bar{\mathfrak{s}}_{i}(\phi_{i+2})$$ 
 as the auxiliary date for defining $\Psi_1(\hat{g}_{i+2})$. So we have 
 \begin{align*}
& \Psi_1(\hat{g}_{i+2})\\
=\ &\sum_{j=0}^2[\mathfrak{t}_{i,i+1}(\phi_{i+2})\Big(\mathfrak{s}_i(-\phi_{i+2})\mathfrak{a}_{j,j+1}\bar{\mathfrak{s}}_{i}(\phi_{i+2})\Big)]_{R,\epsilon}^{h.h'}-[\mathfrak{t}_{i,i+1}(\phi_{i+2})\Big(\mathfrak{s}_i(-\phi_{i+2})\mathfrak{b}\bar{\mathfrak{s}}_{i}(\phi_{i+2})\Big)]_{R,\epsilon}^{h.h'}\\
&-[\mathfrak{t}_{i,i+1}(\phi_{i+2})\Big(\mathfrak{s}_i(-\phi_{i+2})\bar{\mathfrak{b}}\bar{\mathfrak{s}}_{i}(\phi_{i+2})\Big)]_{R,\epsilon}^{h.h'}\\
=\ &\sum_{j=0}^2[(\bar{s}_i\mathfrak{t}_{i,i+1}\mathfrak{s}_{i+1})(\phi_{i+2})\mathfrak{a}_{j,j+1}]_{R,\epsilon}^{h.h'}-[(\bar{s}_i\mathfrak{t}_{i,i+1}\mathfrak{s}_{i+1})(\phi_{i+2})\mathfrak{b}]_{R,\epsilon}^{h.h'}-[(\bar{s}_i\mathfrak{t}_{i,i+1}\mathfrak{s}_{i+1})(\phi_{i+2})\bar{\mathfrak{b}}]_{R,\epsilon}^{h.h'}\\
=\ &\Psi_1(\hat{g}_{i+2}').
 \end{align*}
Here the second equality holds since $\mathfrak{s}_i(-\phi_{i+2})$ equals $\mathfrak{s}_{i+1}(\phi_{i+2})$, up to $(2\delta)$-closeness of framings.

So we have $$\Psi_1(\hat{g}_0)+\Psi_1(\hat{g}_1)+\Psi_1(\hat{g}_2)=\Psi_1(\hat{g}_0')+\Psi_1(\hat{g}_1')+\Psi_1(\hat{g}_2')=0,$$ and the proof is done.
\end{proof}

\subsubsection{Definition of $\Psi$}\label{definitionPsi}

Now we give the definition of $\Psi:\pi_1(\text{SO}(M),\text{e})\to \Omega_{R,\epsilon}^{h,h'}(M)$. 

\begin{proposition}\label{definitionPsilemma}
The restriction of $\Psi_1:\hat{\mathscr{C}}_{DL}\cap \hat{\mathscr{B}}_{\frac{3}{2}R}^{h-2}\to \Omega_{R,\epsilon}^{h,h'}(M)$ to the subset $\widehat{\tau_{k}(\mathscr{B}_K)}$ extends uniquely to a homomorphism $\Psi:\pi_1(\text{SO}(M),\text{e})\to \Omega_{R,\epsilon}^{h,h'}(M)$. 
\end{proposition}

\begin{proof}
Since $\mathscr{B}_K$ is a triangular generating set of $\pi_1(M,*)$ and $\pi_1(\text{SO}(M),\text{e})\cong \pi_1(M,*)\times \mathbb{Z}/2\mathbb{Z}$, $\widehat{\tau_{k}(\mathscr{B}_K)}$ is a triangular generating set of $\pi_1(\text{SO}(M),\text{e})$. So for each element $c\in \pi_1(\text{SO}(M),\text{e})$, we can write $c=c_1c_2\cdots c_n$ with $c_i\in \widehat{\tau_{k}(\mathscr{B}_K)}\subset \hat{\mathscr{C}}_{DL}\cap \hat{\mathscr{B}}_{\frac{3}{2}R-4DL}^{h-3}$. Then we define that 
$$\Psi(c)=\Psi_1(c_1)+\Psi_1(c_2)+\cdots+\Psi_1(c_n).$$

Suppose that we alternatively write $c=c_1'c_2'\cdots c_m'$ with $c_j'\in \widehat{\tau_{k}(\mathscr{B}_K)}$. Since $\widehat{\tau_{k}(\mathscr{B}_K)}$ is a triangular generating set, $c_1c_2\cdots c_n$ can be transformed to $c_1'c_2'\cdots c_m'$ by a finite sequence of relations in $\widehat{\tau_{k}(\mathscr{B}_K)}$ with length at most $3$. Then Lemma \ref{relation} implies that $\Psi(c)$ is well defined.

\end{proof}

\begin{remark}\label{twodefinitiondifference}
For each element 
$$c\in \hat{\mathscr{C}}_{DL}\cap \hat{\mathscr{B}}_{\frac{3}{2}R}^{h-2},$$
 it lies in the domain of two functions: 
 $\Psi_1:\hat{\mathscr{C}}_{DL}\cap \hat{\mathscr{B}}_{\frac{3}{2}R}^{h-2}\to \Omega_{R,\epsilon}^{h,h'}(M)$ and $\Psi:\pi_1(\text{SO}(M),\text{e})\to \Omega_{R,\epsilon}^{h,h'}(M)$. 
If $c\notin  \widehat{\tau_{k}(\mathscr{B}_K)}$, we do not know whether $\Psi_1(c)=\Psi(c)$ holds here. 
\end{remark}

\subsection{Verifications}\label{verifications}

Let $\Psi^{\text{ab}}:H_1(\text{SO}(M);\mathbb{Z})\to \Omega_{R,\epsilon}^{h,h'}(M)$ be the abelianization of the homomorphism $\Psi:\pi_1(\text{SO}(M);\text{e})\to \Omega_{R,\epsilon}^{h,h'}(M)$ defined in Proposition \ref{definitionPsilemma}.
In this section, we will prove that $\Phi\circ \Psi^{\text{ab}}=id:H_1(\text{SO}(M);\mathbb{Z})\to H_1(\text{SO}(M);\mathbb{Z})$ and $\Psi$ is surjective. These two results together imply that $\Phi$ is an isomorphism.

The fact that $\Phi\circ \Psi^{\text{ab}}=id$ is proved in Lemma 5.23 of \cite{LM}, so we only state the following result without giving a proof. Note that the proof of Lemma 5.23 in \cite{LM} also works for cusped hyperbolic $3$-manifolds.

\begin{proposition}\label{identity}
For any $\hat{g}\in \hat{\mathscr{C}}_{DL}\cap \hat{\mathscr{B}}_{\frac{3}{2}R}^{h-2}\subset \pi_1(\text{SO}(M),\text{e})$, we have that $\Phi\circ \Psi_1(\hat{g})=[\hat{g}]$ holds. Here $[\hat{g}]$ denotes the homology class in $H_1(\text{SO}(M);\mathbb{Z})$ represented by $\hat{g}$.
Hence, $\Phi\circ \Psi^{\text{ab}}:H_1(\text{SO}(M);\mathbb{Z})\to H_1(\text{SO}(M);\mathbb{Z})$ is the identity homomorphism.
\end{proposition}

In the following proposition, we prove that $\Psi:\pi_1(\text{SO}(M),\text{e})\to \Omega_{R,\epsilon}^{h,h'}(M)$ is surjective. 

\begin{proposition}\label{surjective}
For any $\gamma\in {\bf \Gamma}_{R,\epsilon}^h$, the cobordism class $[\gamma]_{R,\epsilon}^{h,h'}\in \Omega_{R,\epsilon}^{h,h'}(M)$ is equal to an integral linear combination of elements in the $\Psi$-image of $\widehat{\tau_{k}(\mathscr{B}_K)}\subset \pi_1(\text{SO}(M),\text{e})$. Hence, $\Psi$ is a surjective homomorphism onto $\Omega_{R,\epsilon}^{h,h'}(M)$.
\end{proposition}

The proof of Proposition \ref{surjective} is significantly more complicated than the proof of the corresponding result (Lemma 5.24) in \cite{LM}. We will first prove two techinical lemmas.

In the definition of $\Psi_1$, we have to make the height in $ \hat{\mathscr{C}}_{DL}\cap \hat{\mathscr{B}}_{\frac{3}{2}R}^{h-2}$ smaller than $h$ by some definite amount (we take $h-2$ here), so that for any $\hat{g}\in  \hat{\mathscr{C}}_{DL}\cap \hat{\mathscr{B}}_{\frac{3}{2}R}^{h-2}$, we can make sure all terms in the definition of $\Psi_1(\hat{g})$ have length at most $h$. However, once we choose this $h-2<h$, for $[\gamma]_{R,\epsilon}^{h,h'}\in \Omega_{R,\epsilon}^{h,h'}(M)$ such that $\gamma$ has cusp excursions of height very close to $h$, it is not easy to write it as a linear combination of $\Psi$-images. So we need the following lemma to decrease heights of $(R,\epsilon)$-good curves that represent $[\gamma]_{R,\epsilon}^{h,h'}$.

\begin{lemma}\label{reduceheight}
For any $\gamma \in {\bf \Gamma}_{R,\epsilon}^h$, $[\gamma]_{R,\epsilon}^{h,h'}$ equals an integer linear combination of $[\gamma_i]_{R,\epsilon}^{h,h'}$ with $\gamma_i\in {\bf \Gamma}_{R,10\delta}^{h-6}$.
\end{lemma}

\begin{proof}
By Lemma \ref{split} (splitting), it suffices to prove that $[\gamma]_{R,\epsilon}^{h,h'}$ equals an integer linear combination of $[\gamma_i]_{R,\epsilon}^{h,h'}$ with $\gamma_i\in {\bf \Gamma}_{R,\epsilon}^{h-7}$.

{\bf Step I.} We prove that there exists an $(R,\epsilon)$-panted subsurface $F$ of height at most $h+2$, such that the oriented boundary of $F$ consists of $\gamma$ and $\delta_i$ for $i=1,\cdots,n$. Here each $\delta_i\in {\bf \Gamma}_{R,\epsilon}^{h+2}$, and $\delta_i$ contains a sub-segment of length at least $\frac{5}{4}R$ and height at most $\frac{3}{5}\ln{R}$.

We write $\gamma$ as a concatenation of oriented $\partial$-framed segments $[\mathfrak{s}\mathfrak{s}']$ of same length and phase (close to $0$). Then by Theorem \ref{connection_principle_with_height} and the proof of Lemma \ref{split} (splitting), there exists a $\partial$-framed segment $\mathfrak{t}$ from $p_{\text{ini}}(\mathfrak{s})$ to $p_{\text{ter}}(\mathfrak{s})$ of length $\delta$-close $2R-l(\mathfrak{s})+2I(\frac{\pi}{2})$, such that its initial and terminal cusp excursions have height at most $h$ (by choosing $\vec{t}_p$ and $\vec{t}_q$ pointing down cusps), and all of its intermediate cusp excursions have height at most $\frac{1}{2}\ln{2R}+C\ln{\frac{1}{\delta}}+C<\frac{3}{5}\ln{R}-2$ (by (\ref{R4})). Moreover, we have a pair of good pants $\Pi\in {\bf \Pi}_{R,\epsilon}^{h+1}$ such that its oriented boundary consists of $\gamma,\delta_1'=[\mathfrak{t}\bar{\mathfrak{s}}], \delta_2'=[\bar{\mathfrak{t}}\bar{\mathfrak{s}}']$, with $\delta_i'\in {\bf \Gamma}_{R,\epsilon}^{h+1}$ (by Lemma \ref{distance}). By the choice of $\mathfrak{t}$, each $\delta_i'$ has a sub-segment $d_i$ (corresponding to the complement of initial and terminal cusp excursions of $\mathfrak{t}$) of length at least $R-3h$ and height at most $\frac{3}{5}\ln{R}-1$.

Then we take the middle point of $d_i$ and use it to divide $\delta_i'$ to a concatenation of two oriented $\partial$-framed segments of same length and phase. Then we apply the same argument as above again for each $i$, then we get a good pants $\Pi_i\in {\bf \Pi}_{R,\epsilon}^{h+2}$ such that its oriented boundary consists of $\delta_i',\delta_{i1}, \delta_{i2}$ with $\delta_{ij}\in {\bf \Gamma}_{R,\epsilon}^{h+2}$, and each $\delta_{ij}$ has a subsegment (corresponding to the union of half of $d_i$ and the new oriented $\partial$-framed segment) of length at least $\frac{3}{2}R-3h>\frac{5}{4}R$ and height at most $\frac{3}{5}\ln{R}$. So step I is done.

\bigskip

{\bf Step II.} Suppose that  $\gamma\in {\bf \Gamma}_{R,\epsilon}^{h+2}$ and it contains a sub-segment of length at least $\frac{5}{4}R$ and height at most $\frac{3}{5}\ln{R}$. We prove that there exists an $(R,\epsilon)$-panted subsurface $F$ of height at most $h+3$, such that the oriented boundary of $F$ consists of $\gamma$ and $\delta_i$ for $i=1,\cdots,n$. Here each $\delta_i\in {\bf \Gamma}_{R,\epsilon}^{h+3}$ has at most one cusp excursion of height at least $h-8$, and it contains a sub-segment of length at least $\frac{5}{4}R$ and height at most $\frac{21}{10}\ln{R}$.

For each cusp excursion of $\gamma$ of height at least $h-9$, it has length at least $2(h-9)>\ln{R}$, by (\ref{R4}). Since all such cusp excursions are disjoint from each other and are contained in a sub-segment $c\subset \gamma$ of length at most $(2R-\frac{5}{4}R)+2\cdot \frac{3}{5}\ln{R}+2<R$ (by (\ref{R5})), there are at most $\frac{R}{\ln{R}}$ many cusp excursions of height at least $h-9$. If the number of such cusp excursions is at most one, there is nothing to prove. 

Note that for any two distinct cusp excursions in $c$, there is a point in $c$ between them that lies in the thick part of $M$. We take a point $p\in c$ in the thick part of $M$ that divides $c$ to $c_1\cup c_2$, such that the number of cusp excursions of height at least $h-9$ in $c_1$ and $c_2$ differ by at most $1$. Let $q$ be the opposite point of $p$ on $\gamma$, then it lies in the sub-segment of height at most $\frac{3}{5}\ln{R}$, and both $p$ and $q$ have height at most $\frac{3}{5}\ln{R}$. By the construction in step I, there exists an oriented $\partial$-framed segment $\mathfrak{t}$ from $p$ to $q$ of length $\delta$-close $2R-\frac{1}{2}l(\gamma)+2I(\frac{\pi}{2})$, such that all cusp excursions of $\mathfrak{t}$ have height at most $\frac{3}{5}\ln{R}$ (since $\frac{1}{2}\ln{2R}+C\ln{\frac{1}{\delta}}+C<\frac{3}{5}\ln{R}$, by (\ref{R4})). 

Now we have a good pants $\Pi$ whose oriented boundary consists of $\gamma, \delta_1,\delta_2$ such that the following  hold.
\begin{enumerate}
\item[(a)] Each $\delta_i\in {\bf \Gamma}_{R,\epsilon}^{(h+2)+\frac{3}{R}}$ and $\Pi\in {\bf \Pi}_{R,\epsilon}^{(h+2)+\frac{3}{R}}$.
\item[(b)] Each $\delta_i$ contains a sub-segment of length at least $\frac{5}{4}R$ and height at most $\frac{3}{5}\ln{R}+1$.
\item[(c)] Each $\delta_i$ contains at most $\lceil\frac{1}{2}\frac{R}{\ln{R}}\rceil$ many cusp excursions of height at least $(h-9)+\frac{3}{R}$.
\end{enumerate} Here $\lceil x\rceil$ denotes the least integer greater or equal to $x$.

The reasons for conditions (a), (b), (c) are provided in the following. Since $\alpha\geq 4$, for any point on $\gamma$ with height at least $h-10$, its distances from both $p$ and $q$ are at least $\ln{R}$ (even if $\frac{3}{5}\ln{R}$ is replaced by $\frac{21}{10}\ln{R}$ in the following part of the proof). So each cusp excursion of $\gamma$ of height at least $h-10$ only gives rise to one cusp excursion of $\delta_1$ or $\delta_2$ with height increase by at most $\frac{3}{R}$ (by Lemma \ref{distance} and Lemma \ref{exponentialdecay}). On the other hand, each cusp excursion of $\gamma$ of height at most $h-10$ may only give rise to one cusp excursion of $\delta_1$ or $\delta_2$, with height increase by at most $1$ (Lemma \ref{distance}), which is smaller than $h-9$. So items (a) and (c) hold. Now item (b) follows from the fact that $\mathfrak{t}$ has length at least $R$, while $p$ (and $q$) has a $\frac{3}{5}\ln{R}$ (and $\frac{1}{4}R$) neighborhood in both components of $\gamma \setminus \{p,q\}$ with height at most $\frac{3}{5}\ln{R}$.

 Now we repeat this construction for $\lceil \log_2(\frac{R}{\ln{R}})\rceil+2\leq \frac{1}{\ln{2}}(\ln{R}-\ln{\ln{R}})+3$ times. In this process, all the good curves contain a sub-segment of length at least $\frac{5}{4}R$ and height at most 
 $$\frac{3}{5}\ln{R}+\frac{1}{\ln{2}}(\ln{R}-\ln{\ln{R}})+3<\frac{21}{10}\ln{R},\ \text{by}\ (\ref{R4}).$$ 
All the good curves and good pants obtained by this process have height at most 
 $$(h+2)+\frac{3}{R}(\frac{1}{\ln{2}}(\ln{R}-\ln{\ln{R}}+3))\leq h+3,\ \ \text{by\ (\ref{R5})}.$$ 
 Most importantly, for a sequence of numbers $\{A_k\}$ with $A_0=\frac{R}{\ln{R}}$ and $A_{k+1}\leq \lceil \frac{1}{2} A_k \rceil$, if $n\geq \lceil \log_2(\frac{R}{\ln{R}})\rceil+2$, we have $A_n\leq 1$ holds. So each resulting curve after the $ (\lceil \log_2(\frac{R}{\ln{R}})\rceil+2)$-th step has at most one cusp excursion of height at most  $$(h-9)+\frac{3}{R}(\frac{1}{\ln{2}}(\ln{R}-\ln{\ln{R}}+3))\leq h-8.$$ 
 
 \bigskip
 
 {\bf Step III.} Suppose that $\gamma\in {\bf \Gamma}_{R,\epsilon}^{h+3}$ has only one cusp excursion of height at least $h-8$, and contains a sub-segment of length at least $\frac{5}{4}R$ and height at most $\frac{21}{10}\ln{R}$. There exists an $(R,\epsilon)$-panted subsurface $F$ of height at most $h+3$, such that its oriented boundary consists of $\gamma$ and $\delta_i$ for $i=1,\cdots,n$, with $\delta_i\in {\bf \Gamma}_{R,\epsilon}^{h-7}$.
 
We denote the only cusp excursion of $\gamma$ of height at least $h-8$ by $c$. Let $p$ be the highest point of $c$ and let $q$ be the opposite point of $p$ on $\gamma$, so that they divide $\gamma$ to a concatenation $[\mathfrak{s}\mathfrak{s}']$ of same length and phase (close to $0$). Note that $q$ lies in the subsegment of height at most $\frac{21}{10}\ln{R}$. Here the initial and terminal points of $\mathfrak{s}$ are $p$ and $q$ respectively, and we choose the framing of $\mathfrak{s}$ such that $\vec{t}_{\text{ini}}(\mathfrak{s})\times \vec{n}_{\text{ini}}(\mathfrak{s})$ points straightly down the cusp (has angle $\frac{\pi}{2}$ with the horotorus going through $p$). Now we apply Theorem \ref{connection_principle_with_height} to construct an oriented $\partial$-framed segment $\mathfrak{t}$ such that the following hold.
\begin{enumerate}
\item[(a)] The initial and terminal points of $\mathfrak{t}$ are $p$ and $q$ respectively.
\item[(b)] The length and phase of $\mathfrak{t}$ are $\delta$-close to $2R-l(\mathfrak{s})+2I(\frac{\pi}{2})$ and $-\phi(\mathfrak{s})$ respectively.
\item[(c)] The initial and terminal directions of $\mathfrak{t}$ are $\delta$-close to $\vec{t}_{\text{ini}}(\mathfrak{s})\times \vec{n}_{\text{ini}}(\mathfrak{s})$ and $-\vec{t}_{\text{ter}}(\mathfrak{s})\times \vec{n}_{\text{ter}}(\mathfrak{s})$ respectively.
\item[(d)] The initial and terminal framings of $\mathfrak{t}$ are $\delta$-close to $\vec{n}_{\text{ini}}(\mathfrak{s})$ and $\vec{n}_{\text{ter}}(\mathfrak{s})$ respectively.
\item[(e)] The initial and terminal cusp excursions of $\mathfrak{t}$ have height at most $h+3$ and $\frac{21}{10}\ln{R}+(\ln{\frac{1}{\delta}}+C)$ respectively, and all intermediate cusp excursions of $\mathfrak{t}$ have height at most $\ln{R}$.
\end{enumerate} 
Here the length of $\mathfrak{t}$ is at least $R$, while the height of $p$ and $q$ are at most $h+3$, which are smaller than $\kappa R$ (by (\ref{R5})). The height of cusp excursions of $\mathfrak{t}$ follows from the height control in Theorem \ref{connection_principle_with_height} and the fact that $\vec{t}_{\text{ini}}(\mathfrak{t})$ almost points straightly down the cusp. 

Now we take the pants $\Pi$ given by $\gamma$ and $\mathfrak{t}$, then $\Pi\in{\bf \Pi}_{R,\epsilon}^{h+3}$ and its oriented boundary consists of $\gamma,[\mathfrak{t}\bar{\mathfrak{s}}],[\bar{\mathfrak{t}}\bar{\mathfrak{s}}']$. Moreover, both $[\mathfrak{t}\bar{\mathfrak{s}}]$ and $[\bar{\mathfrak{t}}\bar{\mathfrak{s}}']$ satisfy the following.
\begin{enumerate}
\item[(a)] $[\mathfrak{t}\bar{\mathfrak{s}}], [\bar{\mathfrak{t}}\bar{\mathfrak{s}}']\in {\bf \Gamma}_{R,\epsilon}^{h+\frac{5}{2}}$.
\item[(b)] It has only one cusp excursion of height at least $h-8+\frac{3}{R}$.
\item[(c)] It contains a sub-segment of length at least $\frac{5}{4}R-2h$ and of height at most $\frac{21}{10}\ln{R}+(\ln{\frac{1}{\delta}}+C+1)$
\end{enumerate}

Here condition (a) follows from the following simple observation in hyperbolic geometry. In the upper-half space model, let $\alpha$ be the geodesic between $-1\in \mathbb{C}$ and $1\in \mathbb{C}$, and let $\delta$ be the geodesic ray from the highest point $(0,0,1)\in \alpha$ to $0\in \mathbb{C}$ that is perpendicular to $\alpha$, and let $\beta$ be the geodesic between $0$ and $1$. Then the height of $\alpha$ is greater than the height of $\beta$ by $\ln{2}>0.693$. Our construction of $\mathfrak{t}$ fits $\delta$-closely into this model. 
So the heights of $[\mathfrak{t}\bar{\mathfrak{s}}]$ and $[\bar{\mathfrak{t}}\bar{\mathfrak{s}}']$ decrease from the height of $\gamma$ by at least $\frac{1}{2}$ by this observation and the moreover part of Lemma \ref{distance}.

 Condition (b) follows from the fact that the highest point of any other cusp excursions of $\gamma$ has distance at least $\ln{R}$ from $p$ and Lemma \ref{exponentialdecay}. Condition (c) follows from the fact that $\mathfrak{t}$ has a sub-segment of length at least $R+2I(\frac{\pi}{2})-(h+3)$ and height at most $\frac{21}{10}\ln{R}+(\ln{\frac{1}{\delta}}+C)$, while each of $\mathfrak{s}$ and $\mathfrak{s}'$ has  a sub-segment of length at least $\frac{5}{4}R-R$ and height at most $\frac{21}{10}\ln{R}$.

By induction, we repeat this process for $20$ times. In each step, we decrease the height of the curve by $\frac{1}{2}$. So by the end, we have curves of height at most $h+3-\frac{1}{2}\times 20=h-7$. In the $k$-th step ($k\leq 20$), the resulting curve has only one cusp excursion of height at least $h-8+\frac{3}{R}k<h-7$, and contains a sub-segment of length at least $\frac{5}{4}R-2kh>R$ and of height at most $\frac{21}{10}\ln{R}+k(\ln{\frac{1}{\delta}}+C+1)<\frac{5}{2}\ln{R}<\alpha-10$, by (\ref{R5}) and (\ref{R4}).

Finally, we get an $(R,\epsilon)$ panted subsurface $F$ (consists of at most $2^{20}$ pairs of pants) of height at most $h+3$ whose oriented boundary consists of $\gamma$ and $\delta_i$ with $\delta_i\in {\bf \Gamma}_{R,\epsilon}^{h-7}.$

\end{proof}

 Let $\mathcal{A}$ denote the subset of $\pi_1(M,*)$ consists of nontrivial elements $u$ such that its geodesic representative (still denoted by $u$) satisfies the following conditions.
 \begin{enumerate}
 \item The length of $u$ is at most $\frac{5}{4}R$, and the height of $u$ is at most $h-5$.
 \item The initial and terminal directions of $u$ are $(10\delta)$-close to $-\vec{t}_{\text{ini}}(k)$ and $\vec{t}_{\text{ini}}(k)$ respectively.
 \end{enumerate} 
 We use $\hat{\mathcal{A}}$ to denote the preimage of $\mathcal{A}$ in $\pi_1(\text{SO(M)},\text{e})$.
 One basic property of $\mathcal{A}$ is that $\tau_k(\mathcal{A})$ is contained in $\mathscr{C}_{DL}\cap \mathscr{B}_{\frac{3}{2}R-4DL}^{h-3}$, which is proved in the following lemma.
 
 \begin{lemma}\label{containinintersection}
 $$\tau_k(\mathcal{A})\subset \mathscr{C}_{DL}\cap \mathscr{B}_{\frac{3}{2}R-4DL}^{h-3}.$$
 \end{lemma}
 
 \begin{proof}
 For any $u\in \mathcal{A}$, we prove that $\tau_k(u)=k^{-1}uk\in \mathscr{C}_{DL}\cap\mathscr{B}_{\frac{3}{2}R-4DL}^{h-3}.$
 
 By Lemma \ref{simplecharacterize} (2), we have $\tau_{k}(u)=k^{-1}uk\in \mathscr{C}_{DL}$. By Lemma \ref{setuplemma} (4), we have $|k^{-1}uk|\leq 2|k|+|u|\leq \frac{5}{4}R+2|k|<\frac{3}{2}R-4DL$, so $\tau_k(u)\in \mathscr{B}_{\frac{3}{2}R-4DL}$ holds. Since the heights of $k$ and $u$ are at most $h-5$, by Lemma \ref{setuplemma} (4) (b). The height of $\tau_k(u)$ is at most $h-3$, by Lemma \ref{distance}. So
 $\tau_k(u)\in \mathscr{B}_{\frac{3}{2}R-4DL}^{h-3} $ and the proof is done.
 \end{proof}

In the following lemma, we prove that each element in $\mathcal{A}$ can be written as a product of elements in $\mathscr{B}_{30L}$, via a sequence of nice triangular relations.

\begin{lemma}\label{reduce}
For any element $u \in \mathcal{A} \subset \pi_1(M,*)$, such that the length of $u$ is at most $\frac{6}{5}R$, the height of $u$ is at most $h-5$ and
 the initial and terminal directions of $u$ are $(2\delta)$-close to $-\vec{t}_{\text{ini}}(k)$ and $\vec{t}_{\text{ini}}(k)$ respectively. Then there is a finite sequence $(S_i)_{i=1}^N$, such that each $S_i$ is a sequance $(g_{ij})_{j=1}^i$ of elements in $\pi_1(M,*)$ with length $i$ and the following conditions hold.
 \begin{enumerate}
 \item We have $S_1=(g_{11})=(u)$.
 \item For each $i=1,\cdots,N$ and $j=1,\cdots,i$, $g_{ij}\in \mathcal{A}$ holds.
 \item For each $i=1,\cdots,N-1$, there is an $m_i$ between $1$ and $i$, such that $g_{ij}=g_{i+1,j}$ for each $j<m_i$, $g_{ij}=g_{i+1,j+1}$ for each $j>m_i$, and $g_{i,m_i}=g_{i+1,m_i}g_{i+1,m_{i}+1}$ hold.  
 \item For each $g_{Nj}$ with $j=1,\cdots,N$ in the last sequence, $|g_{Nj}|<30L$ holds.
 \end{enumerate}
 \end{lemma}
 
 \begin{proof}
 We prove this lemma in two steps.
 
{\bf Step I.} Suppose that the height of $u$ is at least $L+2$. We prove that there exists a sequence of sequences in $\pi_1(M,*)$ as in this lemma that satisfies conditions (1), (2), (3) and the following (4)'.
\begin{enumerate}[start=4, label={(\arabic*)'}]
\item For each $g_{Nj}$ with $j=1,\cdots,N$ in the last sequence, its length is at most $\frac{6}{5}R$ and its height is at most $L+2$. Moreover, the initial and terminal directions of $g_{Nj}$ are $(5\delta)$-close to $-\vec{t}_{\text{ini}}(k)$ and $\vec{t}_{\text{ini}}(k)$ respectively.
\end{enumerate}

Let $h_u$ be the height of $u$, then $h_u\in (L+2,h-5)$. We take all cusp excursions of $u$ with height at least $h_u-2$, and denote them by $c_1,c_2,\cdots,c_k$ (by following the linear order on $u$). Let $p_i$ be the highest point of $c_i$. For each $i=1,\cdots,k$, let $u_i$ be the sub-segment of $u$ from the initial point $*$ to $p_i$. For each $i=1,\cdots,k+1$, let $v_i$ be the sub-segment of $u$ from $p_{i-1}$ to $p_i$, with $p_0=p_{k+1}=*$. Then for each $i=1,\cdots,k$, $u_iv_{i+1}=u_{i+1}$ and $u_iv_{i+1}\cdots v_{k+1}=u$ holds.

For each $i$, we apply Thoerem \ref{connection_principle_with_height} to construct a geodesic segment $\gamma_i$ such that the following hold.
\begin{enumerate}
\item[(a)] The initial and terminal points of $\gamma_i$ are $p_i$ and $*$ respectively.

\item[(b)] The initial direction of $\gamma_i$ is $\delta$-close to the tangent vector at $p_i$ that points straightly down the cusp, and the terminal direction of $\gamma_i$ is $\delta$-close to $\vec{t}_{\text{ini}}(k)$.
\item[(c)] The length of $\gamma_i$ is $\delta$-close to $\kappa^{-1}h_u<\frac{1}{40}R$ (by (\ref{R5})), and the height of all intermediate cusp excursions of $\gamma_i$ are at most $\frac{1}{2}\ln{(\kappa^{-1}h_u})+C\ln{\frac{1}{\delta}}+C<h_u-2$ (by (\ref{L3})).
\end{enumerate}

By using the relation $u_i\gamma_i=(u_{i-1}\gamma_{i-1})(\gamma_{i-1}^{-1}v_i\gamma_i)$, we get a sequence of sequences in $\pi_1(M,*)$:
\begin{align*}
& S_1=(u),\ S_2=(u_k\gamma_k, \gamma_k^{-1}v_{k+1}),\ S_3=(u_{k-1}\gamma_{k-1},\gamma_{k-1}^{-1}v_k\gamma_k,\gamma_k^{-1}v_{k+1}),\\
\cdots,\ & S_{k+1}=(u_1\gamma_1,\gamma_1^{-1}v_2\gamma_2,\gamma_2^{-1}v_3\gamma_3, \cdots,\gamma_{k-1}^{-1}v_k\gamma_k,\gamma_k^{-1}v_{k+1}).
\end{align*}
This sequence obviously satisfies (1) and (3), and we will check that (2) holds. The lengths of $u_i\gamma_i$ and $\gamma_{i-1}^{-1}v_i\gamma_i$ are bounded above by 
$$|u|+2\frac{1}{40}R<\frac{6}{5}R+\frac{1}{20}R=\frac{5}{4}R.$$
Since the length of each $u_i,v_i,\gamma_i$ are at least  $h_u-2$, by Lemma \ref{anglechange}, the initial and terminal directions of $u_i\gamma_i, \gamma_{i-1}^{-1}v_i\gamma_i, \gamma_k^{-1}v_{k+1}$ are $(40e^{-\frac{h_u-2}{2}})$-close to the initial direction of its initial segment and the terminal direction of its terminal segment, respectively. So these initial and terminal directions are $(2\delta+40e^{-\frac{h_u-2}{2}})$-close to $-\vec{t}_{\text{ini}}(k)$ and $\vec{t}_{\text{ini}}(k)$ respectively. Moreover, since each $u_1,v_i,\gamma_i$ has no intermediate cusp excursion of height at least $h_u-2$ and the initial direction of each $\gamma_i$ straightly points down the cusp, by the argument in Step III of the proof of Lemma \ref{reduceheight}, the heights of $u_1\gamma_1$, $\gamma_{i-1}^{-1}v_i\gamma_i$, $\gamma_k^{-1}v_{k+1}$ are bounded above by $h_u-\frac{1}{2}$.

If some term $c$ in the last sequence $S_{k+1}$ has length greater than $\frac{6}{5}R$ (there is at most one such term), we do the following construction. Take a point $p$ in $c$ that lies in the thick part of $M$, such that its distance from the middle point of $c$ is at most $h$ (since the height of $c$ is bounded by $h_u<h-5$ and Lemma \ref{cuspheightvslength}). Then $p$ divides $c$ to a concatenation $c_1c_2$ such that each $c_i$ has length at most $\frac{5}{8}R+h<\frac{3}{4}R$ and at least $\frac{3}{5}R-h>\frac{1}{2}R$. Now we apply Theorem \ref{connection_principle_with_height} to construct a geodesic segment $c'$ from $p$ to $*$ such that the following hold.
\begin{enumerate}
\item[(a)] The length of $c'$ is $\delta$-close to $L$, and the height of $c'$ is at most $L$.
\item[(b)] The initial direction of $c'$ is $\delta$-close to be perpendicular to $c$, and the terminal direction of $c'$ is $\delta$-close to $\vec{t}_{\text{ini}}(k)$.
\end{enumerate}

Then we take the concatenations $c_1c'$ and $c'^{-1}c_2$. Their lengths are bounded above by $\frac{3}{4}R+L<R$. Their heights are bounded above by 
$$\max\{h_u-\frac{1}{2}+3e^{-(h_u-\frac{3}{2})},L+1\}<h_u-\frac{2}{5},$$
since any point on $c$ with height at least $h_u-\frac{3}{2}$ has distance at least $h_u-\frac{3}{2}$ from $p$. Moreover, by Lemma \ref{anglechange}, the initial direction of $c_1c'$ is $(20e^{-\frac{R}{2}})$-close to the initial direction of $c_1$, and the terminal direction of $c_1c'$ is $(20e^{-L})$-close to the terminal direction of $c'$. So the initial and terminal directions of $c_1c'$ are $(2\delta+40e^{-\frac{h_u-2}{2}}+20e^{-\frac{R}{2}})$-close to $-\vec{t}_{\text{ini}}(k)$ and $\vec{t}_{\text{ini}}(k)$ respectively, and the same hold for $c'c_2$. Here we use the fact that $\delta+20e^{-L}<2\delta$, by (\ref{L2}).

After doing this last step, we get a sequence of sequences in $\pi_1(M,*)$ satisfying (1), (2), (3) and the following (4)''.
\begin{enumerate}[start=4, label={(\arabic*)''}]
\item For each $g_{Nj}$ with $j=1,\cdots,N$ in the last sequence, its length is at most $\frac{6}{5}R$, its height is at most $h_u-\frac{2}{5}$, and its initial and terminal directions  are $(2\delta+40e^{-\frac{h_u-2}{2}}+20e^{-\frac{R}{2}})$-close to $-\vec{t}_{\text{ini}}(k)$ and $\vec{t}_{\text{ini}}(k)$ respectively.
\end{enumerate} 

We will call the term $(2\delta+40e^{-\frac{h_u-2}{2}}+20e^{-\frac{R}{2}})$ the ``angle difference'' after this process, and call the term $(40e^{-\frac{h_u-2}{2}}+20e^{-\frac{R}{2}})$ the ``increase of angle difference''. By this process, we have successfully decreased the height of $u$ by at least $\frac{2}{5}$, with the angle difference increased by $40e^{-\frac{h_u-2}{2}}+20e^{-\frac{R}{2}}$. Now we repeat applying this process to each $g_{Nj}$ in the last sequence we obtained above, and we inductively do this process until all the elements in the last sequence have height at most $L+2$. Since the height decreases by at least $\frac{2}{5}$ in each step, we apply this process for at most $\lceil(h_u-(L+2))(\frac{2}{5})\rceil\leq \lceil\frac{5h}{2}\rceil$ times.

It remains to bound the angle difference for the elements in the last sequence obtained above. In the last step, the height decreases from some number greater than $L+2$ to some number smaller than $L+2$, so the angle difference increases by at most 
$$40e^{-\frac{(L+2)-2}{2}}+20e^{-\frac{R}{2}}=40e^{-\frac{L}{2}}+20e^{-\frac{R}{2}}$$ 
after this step. Then in the previous step, since we started with height at least $L+2+\frac{2}{5}$, the angle difference increases by at most 
$$40e^{-\frac{L}{2}-\frac{1}{5}}+20e^{-\frac{R}{2}}.$$ By taking sum, the angle difference increases from $2\delta$ by at most
\begin{align*}
&\sum_{k=0}^{\infty}40e^{-\frac{L}{2}-\frac{1}{5}k}+ \lceil\frac{5h}{2}\rceil\cdot 20e^{-\frac{R}{2}}\leq \frac{40e^{-\frac{L}{2}}}{1-e^{-\frac{1}{5}}}+50\alpha e^{-\frac{R}{2}}\ln{R}+20e^{-\frac{R}{2}}<3\delta,
\end{align*}
by (\ref{L2}) and (\ref{R4}). So the last sequence we obtained satisfies condition (4)'.

\bigskip
{\bf Step II.} By step I, we can assume that $u\in \mathcal{A}$ has length at most $\frac{6}{5}R$ and height at most $L+2$, with initial and terminal directions $(5\delta)$-close to $-\vec{t}_{\text{ini}}(k)$ and $\vec{t}_{ini}(k)$ respectively.

Then we take points $q_0=*,q_1,\cdots,q_{n-1},q_n=*$ on $u$ that follow the linear order on $u$, such that the distance between $q_i$ and $q_{i+1}$ on $u$ lies in $[10L,20L]$. This can be done by taking $q_i$ inductively such that its distance from $q_{i-1}$ is $20L$. If the last point $q_{n-1}$ is too close to the terminal point $*$ of $u$, we replace it by the average of $q_{n-2}$ and $*$. Then since the height of $u$ is at most $L+2$, for each $q_i$, by Lemma \ref{cuspheightvslength}, there is a point $p_i$ on $u$ that lies in the thick part of $M$ and has distance at most $L+2+\ln{2}<2L$ from $q_i$. So we obtain a sequence of points  $p_0=*,p_1,\cdots,p_{n-1},p_n=*$ on $u$ such that each $p_i$ lies in the thick part of $M$ and the distance between $p_i$ and $p_{i+1}$ lies in $(6L,24L)$. 

We take similar notations as in Step I. For each $i=1,\cdots,n-1$, let $u_i$ be the sub-segment of $u$ from $p_0$ to $p_i$. For each $i=1,\cdots,n$, let $v_i$ be the sub-segment of $u$ from $p_{i-1}$ to $p_i$. Then for each $i=1,\cdots,k$, $u_iv_{i+1}=u_{i+1}$ and $u_iv_{i+1}\cdots v_{n}=u$ holds. So we have $|u_1|<24L$ and $|v_i|<24L$ for all $i$. 

For each $i$, we apply Thoerem \ref{connection_principle_with_height} to construct a geodesic $\gamma_i$ from $p_i$ to $*$ such that the following hold.
\begin{enumerate}
\item[(a)] The length of $\gamma_i$ is $\delta$-close to $L$.
\item[(b)] The initial direction of $\gamma_i$ is $\delta$-close to be perpendicular to $u$ and the terminal direction of $\gamma_i$ is $\delta$-close to $\vec{t}_{\text{ini}}(k)$.
\end{enumerate}

Then we have a sequence of sequences in $\pi_1(M,*)$:
\begin{align*}
& S_1=(u),\ S_2=(u_k\gamma_k, \gamma_k^{-1}v_{k+1}),\ S_3=(u_{k-1}\gamma_{k-1},\gamma_{k-1}^{-1}v_k\gamma_k,\gamma_k^{-1}v_{k+1}),\\
\cdots,\ & S_{k+1}=(u_1\gamma_1,\gamma_1^{-1}v_2\gamma_2,\gamma_2^{-1}v_3\gamma_3, \cdots,\gamma_{k-1}^{-1}v_k\gamma_k,\gamma_k^{-1}v_{k+1}).
\end{align*}
It obviously satisfies conditions (1) and (3) of the lemma. We can check that it satisfies condition (2) as the following. The length of each element in these sequences is at most $\frac{6}{5}R+2L<\frac{5}{4}R$. Since the heights of $u_i,v_i,\gamma_i$ are at most $L+2$, the elements in $S_{k+1}$ have height at most $L+3<h-5$, by (\ref{R4}). The angle difference is at most $5\delta+40e^{-L}<10\delta$, by (\ref{L2}) and Lemma \ref{anglechange}.

Finally, condition (4) follows from the simple estimate $$|u_1\gamma_1|\leq |u_1|+|\gamma_1|<24L+L<30L,$$ 
$$|\gamma_iv_{i+1}\gamma_{i+1}|\leq |\gamma_i|+|v_{i+1}|+|\gamma_{i+1}|<24L+L+L<30L.$$
 \end{proof}

Now we are ready to prove Proposition \ref{surjective}.
 
 \begin{proof}[Proof of Proposition \ref{surjective}]
 The ``hence'' part follows directly from the main statement, since $\Psi$ is a homomorphism.
 
 For any $[\gamma]_{R,\epsilon}^{h,h'}\in \Omega_{R,\epsilon}^{h,h'}(M)$, we need to prove that it is an integer linear combination of elements in $\Psi(\widehat{\tau_{k}(\mathscr{B}_K)})$. By Lemma \ref{reduceheight}, we can assume that $\gamma\in {\bf \Gamma}_{R,10\delta}^{h-6}$ holds.

 {\bf Step I.} We lift $k\in \pi_1(M,*)$ in Lemma \ref{setuplemma} (3) to an element $\hat{k}\in \pi_1(\text{SO}(M),\text{e})$.
Then we prove that there exist $\hat{x}_{\pm}=\hat{k}^{-1}\hat{y}_{\pm}\hat{k}\in \tau_{\hat{k}}(\hat{\mathcal{A}})$ such that the following hold.
 \begin{enumerate}
\item[(a)] $[\gamma]_{R,\epsilon}^{h,h'}=\Psi_1(\hat{x}_+)+\Psi_1(\hat{x}_-).$
\item[(b)] The length of $\hat{y}_{\pm}$ is at most $\frac{6}{5}R$, and the height of $\hat{y}_{\pm}$ is at most $h-5$.
\item[(c)] The initial and terminal directions of $\hat{y}_{\pm}$ are $(2\delta)$-close to $-\vec{t}_{\text{ini}}(k)$ and $\vec{t}_{\text{ini}}(k)$ respectively.
\end{enumerate}
 
 Since $\gamma\in{\bf \Gamma}_{R,10\delta}^{h-6}$, we bisect $\gamma$ into $[\mathfrak{s}_-\mathfrak{s}_+]$ of two oriented $\partial$-framed segments of same length and phase (close to $0$), with height at most $h-6$, such that the two frames at each intersection point of $\mathfrak{s}_-$ and $\mathfrak{s}_+$ coincide with each other. Moreover, we apply frame rotation as in the proof of Lemma \ref{split} (splitting), such that if the terminal point of $\mathfrak{s}_{\pm}$ lies in a cusp of $M$, then $\vec{t}_{\text{ter}}(\mathfrak{s}_{\pm})\times \vec{n}_{\text{ter}}(\mathfrak{s}_{\pm})$ does not point up the cusp. We take an oriented $\partial$-framed segment $\mathfrak{k}$ associated to the $\delta$-sharp element $\hat{k}$ (since any element in $\hat{\mathscr{C}}_{DL}$ is a $\delta$-sharp element).
 
 By Theorem \ref{connection_principle_with_height}, there are oriented $\partial$-framed segments $\mathfrak{u}_{\pm}$ from $p_{\text{ter}}(\mathfrak{s}_{\pm})$ to $*$ such that the following hold:
 \begin{enumerate}
 \item[(a)] The length and phase of $\mathfrak{u}_{\pm}$ are $\delta$-close to $\frac{1}{20}R$ and $0$ respectively, and the height of $\mathfrak{u}_{\pm}$ is at most $h-6$.
\item[(b)] The initial direction of $\mathfrak{u}_{\pm}$ is $\delta$-close to $\vec{t}_{\text{ter}}(\mathfrak{s}_{\pm})\times \vec{n}_{\text{ter}}(\mathfrak{s}_{\pm})$ and the initial framing of $\mathfrak{u}_{\pm}$ is $\delta$-close to $\vec{n}_{\text{ter}}(\mathfrak{s}_{\pm})$.
\item[(c)] The terminal direction of $\mathfrak{u}_{\pm}$ is $\delta$-close to $\vec{t}_{\text{ini}}(\mathfrak{k})$, and the terminal framing of $\mathfrak{u}_{\pm}$ is $\delta$-close to $\vec{n}_{\text{ini}}(\mathfrak{k})$.
 \end{enumerate}
 Here the height of initial and terminal points of $\mathfrak{u}_{\pm}$ are at most $h$, which satisfies $h<\kappa (\frac{1}{20}R)$, by (\ref{R5}). The height bound of $\mathfrak{u}_{\pm}$ follows from Theorem \ref{connection_principle_with_height} and the choice of $\vec{t}_{\text{ini}}(\mathfrak{u}_{\pm})$. See Figure 8 of \cite{LM} to see a picture of the oriented $\partial$-framed segments $\mathfrak{u}_{\pm}$.

Let $\hat{x}_{\pm}\in \pi_1(\text{SO}(M),\text{e})$ be the elements associated to the concatenations of oriented $\partial$-framed segments $[\bar{\mathfrak{k}}\bar{\mathfrak{u}}_-\mathfrak{s}_+\mathfrak{u}_+\mathfrak{k}]$ and $[\bar{\mathfrak{k}}\bar{\mathfrak{u}}_+\mathfrak{s}_-\mathfrak{u}_-\mathfrak{k}]^*$ respectively.
 
 Let $\hat{y}_{\pm}=\hat{k}\hat{x}_{\pm}\hat{k}^{-1}$, then its projection to $\pi_1(M,*)$ is the carrier of $\bar{\mathfrak{u}_{\mp}}\mathfrak{s}_{\pm}\mathfrak{u}_{\pm}$, and we denote it by $y_{\pm}$.
So $y_{\pm}$ has length at most $\frac{6}{5}R$ and height at most $h-5$ (by Lemma \ref{distance} and our construction of $u_{\pm}$). Moreover, its initial and terminal directions are $(2\delta)$-close to $-\vec{t}_{\text{ini}}(k)$ and $\vec{t}_{\text{ini}}(k)$ respectively, by Lemma \ref{anglechange} and the fact that $\delta+40e^{-\frac{R}{40}}<2\delta$ (\ref{R4}). So conditions (b) and (c) hold for $\hat{y}_{\pm}$, and  $\hat{x}_{\pm}\in \tau_{\hat{k}}(\hat{\mathcal{A}})\subset \hat{\mathscr{C}}_{DL}\cap \hat{\mathscr{B}}_{\frac{3}{2}R-4DL}^{h-3}$ holds. 
 
 To compute $\Psi_1(\hat{x}_+)+\Psi_1(\hat{x}_-)$, we take the auxiliary data $\mathfrak{a}_0\vee\mathfrak{a}_1\vee \mathfrak{a}_2$ and $\mathfrak{b}$ for computing $\Psi_1(\hat{x}_+)$, with the $(10^3\delta)$-closeness in Condition \ref{defining} replaced by $\delta$-closeness. Then $\mathfrak{a}_0^*\vee\mathfrak{a}_{-1}^*\vee \mathfrak{a}_{-2}^*$ and $\bar{\mathfrak{b}}^*$ can be used to compute $\Psi_1(\hat{x}_-)$.
 
 For simplicity, we use $\mathfrak{r}_+$ and $\mathfrak{r}_-$ to denote $\bar{\mathfrak{k}}\bar{\mathfrak{u}}_-\mathfrak{s}_+\mathfrak{u}_+\mathfrak{k}$ and $\bar{\mathfrak{k}}\bar{\mathfrak{u}}_+\mathfrak{s}_-\mathfrak{u}_-\mathfrak{k}$ respectively. Then we have 
\begin{align*}
&[\mathfrak{r}_+\mathfrak{a}_{i,i+1}]_{R,\epsilon}^{h,h'}+[\mathfrak{r}_-^*\mathfrak{a}_{i+1,i}^*]_{R,\epsilon}^{h,h'}=[\mathfrak{r}_+\mathfrak{a}_{i,i+1}]_{R,\epsilon}^{h,h'}+[\mathfrak{r}_-\overline{\mathfrak{a}_{i,i+1}}]_{R,\epsilon}^{h,h'}\\
=\ &[\mathfrak{s}_+(\mathfrak{u}_+\mathfrak{k}\mathfrak{a}_{i,i+1}\bar{\mathfrak{k}}\bar{\mathfrak{u}}_-)]_{R,\epsilon}^{h,h'}+[\mathfrak{s}_-\overline{(\mathfrak{u}_+\mathfrak{k}\mathfrak{a}_{i,i+1}\bar{\mathfrak{k}}\bar{\mathfrak{u}}_-)}]_{R,\epsilon}^{h,h'}=[\mathfrak{s}_-\mathfrak{s}_+]_{R,\epsilon}^{h,h'}=[\gamma]_{R,\epsilon}^{h,h'}.
\end{align*}
Here the third equality follows from the inverse of Lemma \ref{split} (splitting).

So we have equalities 
\begin{align*}
& [\mathfrak{r}_+\mathfrak{a}_{01}]_{R,\epsilon}+[\mathfrak{r}_-^*\mathfrak{a}_{10}^{*}]_{R,\epsilon}^{h,h'}=[\gamma]_{R,\epsilon}^{h,h'},\ [\mathfrak{r}_+\mathfrak{a}_{12}]_{R,\epsilon}+[\mathfrak{r}_-^*\mathfrak{a}_{21}^{*}]_{R,\epsilon}^{h,h'}=[\gamma]_{R,\epsilon}^{h,h'},\\
&  [\mathfrak{r}_+\mathfrak{a}_{20}]_{R,\epsilon}+[\mathfrak{r}_-^*\mathfrak{a}_{02}^{*}]_{R,\epsilon}^{h,h'}=[\gamma]_{R,\epsilon}^{h,h'}.
\end{align*}
Similarly, we have 
$$-[\mathfrak{r}_+\mathfrak{b}]_{R,\epsilon}-[\mathfrak{r}_-^*\bar{\mathfrak{b}}^{*}]_{R,\epsilon}^{h,h'}=-[\gamma]_{R,\epsilon}^{h,h'},\ -[\mathfrak{r}_+\bar{\mathfrak{b}}]_{R,\epsilon}-[\mathfrak{r}_-^*\mathfrak{b}^{*}]_{R,\epsilon}^{h,h'}=-[\gamma]_{R,\epsilon}^{h,h'}.$$
Then the sum of these equalities gives $\Psi_1(\hat{x}_+)+\Psi_1(\hat{x}_-)=[\gamma]_{R,\epsilon}^{h,h'}$.

\bigskip

{\bf Step II.} Now we only need to prove that $\Psi_1(\hat{x}_{\pm})$ is an integer linear combination of $\Psi$-image of $\widehat{\tau_k(\mathscr{B}_K)}$. We use $y_{\pm}\in \pi_1(M,*)$ to denote the projection of $\hat{y}_{\pm}=\hat{k}\hat{x}_{\pm}\hat{k}^{-1}\in\pi_1(\text{SO}(M),\text{e})$.

By the proof of step I, the geodesic representative of $y_{\pm}$ is the carrier segment of  $[\bar{\mathfrak{u}}_{\mp}\mathfrak{s}_{\pm}\mathfrak{u_{\pm}}]$. In the previous step, we have checked that $y_{\pm}\in \mathcal{A}$, with length at most $\frac{6}{5}R$, height at most $h-5$, and its initial and terminal directions are $(2\delta)$-close to $-\vec{t}_{\text{ini}}(k)$ and $\vec{t}_{\text{ini}}(k)$ respectively. So $y_{\pm}$ satisfies the assumption of Lemma \ref{reduce}.
 
For simplicity, we use $\hat{x}$ to denote $\hat{x}_+$ or $\hat{x}_-\in \pi_1(\text{SO}(M),\text{e})$, use $\hat{y}$ to denote the corresponding $\hat{y}_+$ or $\hat{y}_-\in \pi_1(\text{SO}(M),\text{e})$, and use $y$ to denote the corresponding $y_+$ or $y_-\in \pi_1(M,*)$. Then we apply Lemma \ref{reduce} to the element $y\in \mathcal{A}$, to get a sequence $(S_i)_{i=1}^N$, such that each $S_i$ is a sequence $(g_{ij})_{j=1}^N$ in $\pi_1(M,*)$ of length $i$. For each $i=1,\cdots,N$, by lifting elements in $\pi_1(M,*)$ to $\pi_1(\text{SO}(M),\text{e})$ inductively, we obtain a lifting length-$i$ sequence $\hat{S}_i=(\hat{g}_{ij})_{j=1}^i$ of $S_i$ in $\pi_1(\text{SO(M)},\text{e})$ such that the following hold.
 \begin{enumerate}
 \item[(a)] We have $\hat{S}_1=(\hat{g}_{11})=(\hat{y})$.
 \item[(b)] For each $i=1,\cdots,N$ and $j=1,\cdots,i$, $\hat{g}_{ij}\in \hat{\mathcal{A}}$ holds.
 \item[(c)] For each $i=1,\cdots,N-1$, there is an $m_i$ between $1$ and $i$, such that $\hat{g}_{ij}=\hat{g}_{i+1,j}$ for each $j<m_i$, $\hat{g}_{ij}=\hat{g}_{i+1,j+1}$ for each $j>m_i$, while $\hat{g}_{i,m_i}=\hat{g}_{i+1,m_i}\hat{g}_{i+1,m_{i}+1}$ hold.  
 \item[(d)] For each $\hat{g}_{Nj}$ with $j=1,\cdots,N$, we have $|\hat{g}_{Nj}|<30L<K$ holds, by (\ref{K1}).
 \end{enumerate}

By (b) and Lemma \ref{containinintersection}, for each $i$ and $j$, $\tau_{\hat{k}}(\hat{g}_{ij})\in \hat{\mathscr{C}}_{DL}\cap \hat{\mathscr{B}}_{\frac{3}{2}R-4DL}^{h-3}$ holds, thus $\Psi_1(\tau_{\hat{k}}(\hat{g}_{ij}))$ is defined. By (a), we have 
$$\Psi_1(\hat{x})=\Psi_1(\tau_{\hat{k}}(\hat{y}))=\Psi_1(\tau_{\hat{k}}(\hat{g}_{11})).$$
By (c) and Lemma \ref{relation}, for each $i=1,\cdots,N-1$, we have $$\Psi_1(\tau_{\hat{k}}(\hat{g}_{im_i}))=\Psi_1(\tau_{\hat{k}}(\hat{g}_{i+1,m_i}))+\Psi_1(\tau_{\hat{k}}(\hat{g}_{i+1,m_i+1})),$$ and
$$\Psi_1(\tau_{\hat{k}}(\hat{g}_{i1}))+\cdots+\Psi_1(\tau_{\hat{k}}(\hat{g}_{ii}))=\Psi_1(\tau_{\hat{k}}(\hat{g}_{i+1,1}))+\cdots+\Psi_1(\tau_{\hat{k}}(\hat{g}_{i+1,i+1})).$$
So we have 
\begin{align*}
&\Psi_1(\hat{x})=\Psi_1(\tau_{\hat{k}}(\hat{g}_{11}))=\Psi_1(\tau_{\hat{k}}(\hat{g}_{N1}))+\cdots+\Psi_1(\tau_{\hat{k}}(\hat{g}_{NN})).
\end{align*}

Finally, by (d), each $\tau_{\hat{k}}(\hat{g}_{Nj})$ lies in $\widehat{\tau_{k}(\mathscr{B}_K)}$. So $\Psi(\tau_{\hat{k}}(\hat{g}_{Nj}))=\Psi_1(\tau_{\hat{k}}(\hat{g}_{Nj}))$ holds by the construction of $\Psi$ in Proposition \ref{definitionPsilemma}. Then both $\Psi_1(\hat{x}_+)$ and $\Psi_1(\hat{x}_-)$ are integer linear combinations of elements in $\Psi(\widehat{\tau_{k}(\mathscr{B}_K)})$, and so does $[\gamma]_{R,\epsilon}^{h,h'}$. So $\Psi$ is surjective.
 
 \end{proof}
 
 Now we are ready to prove Theorem \ref{main1}.
 
 \begin{proof}[Proof of Theorem \ref{main1}]
 We take $\epsilon\in (0,10^{-2})$, take $R_0>0$ that satisfies all the conditions (\ref{R4}), (\ref{R5}) and (\ref{R7}), and take any $R>R_0$. Then all the statements in this section holds for $\epsilon$ and $R$.
Now we have three homomorphisms, which are $\Phi:\Omega_{R,\epsilon}^{h,h'}(M)\to H_1(\text{SO}(M);\mathbb{Z})$, $\Psi:\pi_1(\text{SO}(M),\text{e})\to \Omega_{R,\epsilon}^{h,h'}(M)$, and the abelianization of $\Psi$, which is denoted by $\Psi^{\text{ab}}:H_1(\text{SO}(M);\mathbb{Z})\to \Omega_{R,\epsilon}^{h,h'}(M)$.

By Proposition \ref{identity}, we know that $\Phi\circ \Psi^{\text{ab}}$ is the identity homomorphism on $H_1(\text{SO}(M);\mathbb{Z})$, so $\Psi^{\text{ab}}$ is an injective homomorphism. By Proposition \ref{surjective}, $\Psi$ is a surjective homomorphism, and so does $\Psi^{\text{ab}}$. So $\Psi^{\text{ab}}$ is an isomorphism, and so does $\Phi:\Omega_{R,\epsilon}^{h,h'}(M)\to H_1(\text{SO}(M);\mathbb{Z})$.

By the definition of $\Phi$, the isomorphism $\Phi:\Omega_{R,\epsilon}^{h,h'}(M)\to H_1(\text{SO}(M);\mathbb{Z})$ actually factors through $\Omega_{R,\epsilon}^h(M)=\Omega_{R,\epsilon}^{h,\infty}(M)$. Moreover, since the natural projection $\Omega_{R,\epsilon}^{h,h'}(M)\to \Omega_{R,\epsilon}^h(M)$ is surjective, $\Omega_{R,\epsilon}^h(M)$ is also isomorphic to $H_1(\text{SO}(M);\mathbb{Z})$.
 \end{proof}

\end{document}

%% file: drawing.pdf_tex
\begingroup%
  \makeatletter%
  \providecommand\color[2][]{%
    \errmessage{(Inkscape) Color is used for the text in Inkscape, but the package 'color.sty' is not loaded}%
    \renewcommand\color[2][]{}%
  }%
  \providecommand\transparent[1]{%
    \errmessage{(Inkscape) Transparency is used (non-zero) for the text in Inkscape, but the package 'transparent.sty' is not loaded}%
    \renewcommand\transparent[1]{}%
  }%
  \providecommand\rotatebox[2]{#2}%
  \newcommand*\fsize{\dimexpr\f@size pt\relax}%
  \newcommand*\lineheight[1]{\fontsize{\fsize}{#1\fsize}\selectfont}%
  \ifx\svgwidth\undefined%
    \setlength{\unitlength}{545.74846088bp}%
    \ifx\svgscale\undefined%
      \relax%
    \else%
      \setlength{\unitlength}{\unitlength * \real{\svgscale}}%
    \fi%
  \else%
    \setlength{\unitlength}{\svgwidth}%
  \fi%
  \global\let\svgwidth\undefined%
  \global\let\svgscale\undefined%
  \makeatother%
  \begin{picture}(1,0.58500076)%
    \lineheight{1}%
    \setlength\tabcolsep{0pt}%
    \put(0,0){\includegraphics[width=\unitlength,page=1]{drawing.pdf}}%
    \put(0.44214606,0.00619939){\color[rgb]{0,0,0}\makebox(0,0)[lt]{\lineheight{1.25}\smash{\begin{tabular}[t]{l}$\mathfrak{y}$\end{tabular}}}}%
    \put(0.05430798,0.2659321){\color[rgb]{0,0,0}\makebox(0,0)[lt]{\lineheight{1.25}\smash{\begin{tabular}[t]{l}$\mathfrak{z}$\end{tabular}}}}%
    \put(0.40362982,0.51984057){\color[rgb]{0,0,0}\makebox(0,0)[lt]{\lineheight{1.25}\smash{\begin{tabular}[t]{l}$\mathfrak{x}$\end{tabular}}}}%
    \put(0.23784081,0.17395403){\color[rgb]{0,0,0}\makebox(0,0)[lt]{\lineheight{1.25}\smash{\begin{tabular}[t]{l}$\mathfrak{a}_1$\end{tabular}}}}%
    \put(0.02974017,0.46875394){\color[rgb]{0,0,0}\makebox(0,0)[lt]{\lineheight{1.25}\smash{\begin{tabular}[t]{l}$\mathfrak{a}_1'$\end{tabular}}}}%
    \put(0.64837782,0.35083381){\color[rgb]{0,0,0}\makebox(0,0)[lt]{\lineheight{1.25}\smash{\begin{tabular}[t]{l}$\mathfrak{b}$\end{tabular}}}}%
    \put(0.78101559,0.28951571){\color[rgb]{0,0,0}\makebox(0,0)[lt]{\lineheight{1.25}\smash{\begin{tabular}[t]{l}$\mathfrak{b}'$\end{tabular}}}}%
    \put(0,0){\includegraphics[width=\unitlength,page=2]{drawing.pdf}}%
  \end{picture}%
\endgroup%